\documentclass[11pt]{article}
\usepackage{smile}

\usepackage{fullpage}
\usepackage{framed}
\usepackage{mdframed}
\usepackage{enumerate}
\usepackage[inline]{enumitem}
\usepackage[T1]{fontenc}
\usepackage{moresize}
\usepackage{bm}
\usepackage{dsfont}
\usepackage{amsmath}
\usepackage{amssymb}
\usepackage{amsthm}
\usepackage{amsfonts}
\usepackage{stmaryrd}
\usepackage{array}
\usepackage{mathrsfs}
\usepackage{mathtools} % \prescript
\usepackage{extarrows}
\usepackage{stackrel}
\usepackage[nodisplayskipstretch]{setspace}
\usepackage{color}
\usepackage[usenames,dvipsnames]{xcolor}
\usepackage{cancel}
\usepackage{soul}
\usepackage{xfrac}
\usepackage{siunitx}
\usepackage{graphicx}
\usepackage{float}
\usepackage{rotating}
\usepackage{subcaption}
\usepackage{overpic}
\usepackage[all]{xy}
\usepackage{tikz}
\usetikzlibrary{patterns}
\usepackage{booktabs}
\usepackage{dcolumn}
\usepackage{multirow}
\usepackage{diagbox}
\usepackage{verbatim}
\usepackage{listings}
\usepackage{algorithm}
\usepackage{algpseudocode}
\usepackage{fancyvrb}
\usepackage{hyperref}
\usepackage[round]{natbib}
\usepackage{sectsty}

\hypersetup{
    bookmarks=true,         % show bookmarks bar?
    unicode=false,          % non-Latin characters in Acrobat’s bookmarks
    pdftoolbar=true,        % show Acrobat’s toolbar?
    pdfmenubar=true,        % show Acrobat’s menu?
    pdffitwindow=false,     % window fit to page when opened
    pdfstartview={FitH},    % fits the width of the page to the window
    pdftitle={My title},    % title
    pdfauthor={Author},     % author
    pdfsubject={Subject},   % subject of the document
    pdfcreator={Creator},   % creator of the document
    pdfproducer={Producer}, % producer of the document
    pdfkeywords={key1, key2}, % list of keywords
    pdfnewwindow=true,      % links in new PDF window
    colorlinks=true,        % false: boxed links; true: colored links
    linkcolor=blue,         % color of internal links (change box color with linkbordercolor)
    citecolor=blue,         % color of links to bibliography
    filecolor=blue,         % color of file links
    urlcolor=cyan           % color of external links
}

\def\given{\,|\,}

\def\tr{\mathop{\text{tr}}\kern.2ex}

\def\P{{ P}}
\def\E{{ E}}
\def\R{{\mathbb R}}
\def\Z{{\mathbb Z}}

\def\d{{\mathtt d}}

\newcommand{\sfi}{\mathsf{i}}

\newcommand{\dCov}{\mathrm{dCov}}
\renewcommand{\Pr}{{ P}}

\newcommand{\card}{\mathrm{card}}
\newcommand{\mbinom}{\binom}
\newcommand{\zahl}[1]{\llbracket #1\rrbracket}
\newcommand\yestag{\addtocounter{equation}{1}\tag{\theequation}}
\newcolumntype{L}[1]{>{\raggedright\let\newline\\\arraybackslash\hspace{0pt}}m{#1}}
\newcolumntype{C}[1]{>{  \centering\let\newline\\\arraybackslash\hspace{0pt}}m{#1}}
\newcolumntype{R}[1]{>{ \raggedleft\let\newline\\\arraybackslash\hspace{0pt}}m{#1}}
\newcolumntype{d}[1]{D{.}{.}{#1}}
\newcolumntype{H}{>{\setbox0=\hbox\bgroup}c<{\egroup}@{}}
\newcolumntype{Z}{>{\setbox0=\hbox\bgroup}c<{\egroup}@{\hspace*{-\tabcolsep}}}

\numberwithin{equation}{section}
\newtheorem{theorem}{Theorem}[section]
\newtheorem{lemma}{Lemma}[section]
\newtheorem{proposition}{Proposition}[section]

\newtheorem{definition}{Definition}[section]
\newtheorem{innercustomgeneric}{\customgenericname}
\providecommand{\customgenericname}{}
\newcommand{\newcustomtheorem}[2]{%
  \newenvironment{#1}[1]
  {%
   \renewcommand\customgenericname{#2}%
   \renewcommand\theinnercustomgeneric{##1}%
   \innercustomgeneric
  }
  {\endinnercustomgeneric}
}
\newcustomtheorem{customtheorem}{Theorem}
\newcustomtheorem{customassumption}{Assumption}
\newcustomtheorem{customlemma}{Lemma}
\newcustomtheorem{customexample}{Example}
\theoremstyle{definition}
\newtheorem{example}{Example}[section]
\newtheorem{remark}{Remark}[section]

\newcommand{\nb}[1]{#1}

\makeatletter
\def\namedlabel#1#2{\begingroup
    #2%
    \def\@currentlabel{#2}%
    \phantomsection\label{#1}\endgroup
}
\makeatother

\renewcommand{\theequation}{\thesection.\arabic{equation}}

\begin{document}

\setlength{\abovedisplayskip}{5pt}
\setlength{\belowdisplayskip}{5pt}
\setlength{\abovedisplayshortskip}{5pt}
\setlength{\belowdisplayshortskip}{5pt}
\hypersetup{colorlinks,breaklinks,urlcolor=blue,linkcolor=blue}

\title{\LARGE Distribution-free consistent independence tests via \nb{center-outward ranks and signs}}

\author{
Hongjian Shi\thanks{Department of Statistics, University of Washington, Seattle, WA 98195, USA; e-mail: {\tt hongshi@uw.edu}},~~
Mathias Drton\thanks{Department of Mathematics, Technical University
  of Munich, 85748 Garching b. M\"unchen, Germany; e-mail: {\tt mathias.drton@tum.de}},~~and~
Fang Han\thanks{Department of Statistics, University of Washington, Seattle, WA 98195, USA; e-mail: {\tt fanghan@uw.edu}}.
}

\date{}

\maketitle

\vspace{-1.55em}

\begin{abstract} 
  This paper investigates the problem of testing independence of two
  random vectors of general dimensions.  For this, we give for the
  first time a distribution-free consistent test.  Our approach
  combines distance covariance with \nb{the center-outward ranks 
  and signs developed in} \citet{hallin2017distribution}.  In technical
  terms, the proposed test is consistent and distribution-free in the family of
  multivariate distributions with nonvanishing (Lebesgue) probability
  densities.  Exploiting the (degenerate) U-statistic structure of the
  distance covariance and the combinatorial nature of Hallin's 
  \nb{center-outward ranks and signs}, 
  we are able to derive the limiting null distribution of our test
  statistic.  The resulting asymptotic approximation is accurate
  already for moderate sample sizes and makes the test implementable
  without requiring permutation.  The limiting distribution is derived
  via a more general result that gives a new type of
  combinatorial non-central limit theorem for double- and
  multiple-indexed permutation statistics.
%  Testing joint independence of two random vectors of any dimension is a
%  fundamental statistical challenge. Recently, several tests based on either
%  distance covariance/correlation \citep{MR2382665,MR2752127} or 
%  Hilbert-Schmidt independence criterion \citep{MR2255909,MR2249882,NIPS2007_3201} 
%  that are consistent against all alternatives have been proposed, 
%  while all of them require permutation based test procedures.
%  We propose a new and consistent test using the so-called center-outward 
%  multivariate rank newly developed by \citet{hallin2017distribution}.
%  Built upon a combinatorial non-central limit theorem, 
%  the asymptotic null distribution of test statistic is shown to be of an explicit form.
%  The proposed test is distribution-free in the family of 
%  multivariate distributions with nonvanishing (Lebesgue) probability densities,
%  implementable without requiring a procedure of permutation, 
%  and has effective empirical performance illustrated by simulation studies.
\end{abstract}


\section{Introduction}
\label{sec:introduction}

Let \nb{$\mX\in\R^p$ and $\mY\in\R^q$} be two real random vectors \nb{defined on the same (otherwise unspecified) probability space}. 
This paper treats the problem of testing the null hypothesis
\begin{align}\label{eq:H0}
H_0: \mX\text{ and }\mY\text{ are independent},
\end{align}
based on $n$ independent copies $(\mX_1,\mY_1),\ldots,(\mX_n,\mY_n)$
of $(\mX,\mY)$.  Testing independence is a fundamental statistical
problem that has received much attention in literature.

For the simplest instance, the bivariate case with $p=q=1$,
\nb{\citet{MR0004426},} \citet{MR0029139}, \citet{MR0125690},
\citet{yanagimoto1970measures}, \citet{feuerverger1993},
\citet{MR3178526}, among many others, have proposed tests that are
consistent against all alternatives from slightly different but rather
general classes of distributions.  The tests are usually
formulated using (univariate) ranks of the
data, %and distribution-freeness of the test could often be expected.
%proposed a test that is consistent against all dependent alternatives in the class of absolutely continuous bivariate distributions. \citet{MR0125690} developed Hoeffding's idea to introduce a test that is consistent against all dependent bivariate alternatives; see also \citet{MR0004426}. \citet{feuerverger1993} raised a consistent bivariate rank test by considering the properties of empirical characteristic functions. \citet{MR3178526} made use of a newly developed rank-based statistic, which coincides with the one raised implicitly by  given marginal continuity, to propose a test of independence that is consistent in the class of bivariate distributions that are discrete, continuous, or a mixture of both. 
although recently more tests were proposed based on alternative
summaries of the data, including (i) binning approaches based on a
partition of the sample space
\citep{MR3068450,MR3491123,MR3941252,zhang2019bet}, (ii)  mutual
information \citep{MR2096503,MR3200177,MR3992389}, and (iii) the maximal information coefficient \citep{Reshef1518,MR3595146,MR3773388}.

Testing independence of $\mX$ and $\mY$ consistently when one or both
of the dimensions $p$ and $q$ are larger than one is substantially
more challenging, as noted in \citet[Sec.~7]{feuerverger1993}. 
Solutions have not been discovered until much 
more recently.  Two tracks were pursued.  First, \citet{MR2382665}
generalized \citeauthor{feuerverger1993}'s statistic to multivariate
cases and proposed a new dependence measure termed ``distance
covariance''.  It has been shown that
%For the multivariate independence, \citet{feuerverger1993} stated the difficulty in extending his test statistic to dimension higher than two. \citet{MR2382665} successfully generalized \citeauthor{feuerverger1993}'s statistic to multivariate case. %by finding a new suitable weight function. A slightly different version has been considered by the same authors previously \citep{MR2298886}. U
under the existence of finite marginal first moments, the distance
covariance is zero if and only if $H_0$ holds.  For further extensions, \citet{MR3127883} generalized distance covariance/correlation to general metric spaces, and \citet{jakobsen2017distance} considered the corresponding test of independence in metric spaces.

The second track to characterize non-linear, non-monotone dependence is based on the maximal correlation introduced in \citet{hirschfeld_1935} and \citet{MR0007220}, reformulated and examined by \citet{MR0111074,MR0115203}. %This concept can be easily generalized to the multivariate case.
\citet{gretton:aistats2005,MR2255909,MR2249882} extended this idea to
examine multivariate cases, resulting in the  Hilbert-Schmidt
independence criterion (HSIC), which is a consistent kernel-based measure of
dependence in multivariate cases. Interestingly, \citet{NIPS2007_3201}
connected HSIC with a Gaussian kernel to the characteristic
function-based statistic raised in \citet{feuerverger1993}, and
\citet{MR3127866} pointed out the equivalence between distance
covariance in general metric spaces and the kernel-based independence criterion. 

A notable feature of both distance- and kernel-based statistics
is that their null distributions depend on the
distributions of $\mX$ and $\mY$ even in the large-sample limit.  This
dependence arises already for $p=q=1$ and is usually difficult to estimate. As a consequence, the
tests are, unlike the rank tests of, e.g., \citet{MR0029139} and
\citet{MR0125690}, no longer distribution-free and permutation
analysis has to be conducted to implement
them. %, which makes the computation heavy.
To remedy this problem,
%The asymptotic null distributions of both distance-based and kernel-based statistics are of similar forms, which depend on distributions of $\mX$ and $\mY$ and are difficult to estimate. Accordingly,
\citet{MR2382665} proposed a nonparametric test based on distance
correlation by applying a universal upper tail probability bound for
all quadratic forms of centered Gaussian random variables that have
their mean equal to one \citep{MR1990053}. However, in practice this
upper bound is usually too conservative for the approach to be a
competitor to the computationally much more expensive permutation test
\citep{MR2752127,NIPS2007_3201}.
%This calls for new test statistics with (asymptotic) null distributions that are of accurate forms or can be easily estimated. 
This triggers the following question:  For general $p,q> 1$, does there exist an asymptotically accurate consistent test of $H_0$ that is distribution-free and hence directly implementable?

Rank-based tests constitute a natural approach to answering the above
question.
% have many appealing properties, and thus have been considered to test multivariate independence before.
Indeed, %in literature
in contrast to \citet{MR2752127}, \citet{MR2752133} claimed that the
methods based on marginal ranks are effective and as powerful as
original ones when the sample size is moderately large and this idea
has been explored in depth in \citet{lin2017copula}.  However,
\citet{MR2298886} \nb{noted that the methods based on marginal ranks
do not enjoy distribution-freeness except in dimension one},
which is also recorded in, e.g., Theorem 2.3.2 in
\citet{lin2017copula}. Using the idea of projection from
\citet{MR2328527}, \citet{MR3737307} generalized Hoeffding's $D$ \citep{MR0029139} to
multivariate cases, and \citet{kim2018robust} proposed the analogues
of Blum--Kiefer--Rosenblatt's $R$ \citep{MR0125690} and Bergsma--Dassios--Yanagimoto's
$\tau^*$ \citep{yanagimoto1970measures,MR3178526,drton2018high}.  \citet{MR3842884} proposed other multivariate extensions of
Hoeffding's $D$, Blum--Kiefer--Rosenblatt's $R$, and
Bergsma--Dassios--Yanagimoto's $\tau^*$, and did numerical studies
comparing them to distance covariance applied to marginal ranks.
Alternatively, \citet{MR3068450} developed a consistent multivariate
test based on ranked distance covariance by transferring the original
problem to testing independence of an aggregated $2\times 2$
contingency table.  However, all the aforementioned tests are not
distribution-free when $p$ or $q$ is larger than $1$, and due to the
difficulty of accounting for the dependence within $\mX$ and $\mY$,
permutation analysis is required for their implementation. 
On the other hand,  \citet{MR2956796} and \citet{NIPS2016_6220} introduced distribution-free graph-based and rank-based tests. However, it is unclear if the former is consistent, and the latter requires choosing two arbitrary reference points. The latter test is almost surely consistent in the sense that the choice of reference points needs to avoid an (unknown) measure zero set.

This paper proposes a solution to the above question by combining
Sz\'ekely, Rizzo, and Bakirov's distance covariance with a recently
defined \nb{concept of} multivariate ranks due to \citet{hallin2017distribution}.
%hereafter called Hallin's (multivariate) rank. %As have been noted before, d
Due to the lack of a canonical ordering on $\R^d$ for $d>1$,
fundamental concepts related to distribution functions in dimension
$d=1$, such as ranks and quantiles, do not admit a simple extension
for $d\ge 2$ that maintains properties such as distribution-freeness.
To overcome this limitation, several types of multivariate ranks have
been introduced; see \citet[Sec.~1.3]{hallin2017distribution} and, more recently, \cite{ghosal2019multivariate} for a
literature review.  None of them, however, is distribution-free except
for pseudo-Mahalanobis ranks \citep{MR1926170,MR2001322}, but these
are restricted to the class of elliptically symmetric distributions
\citep{MR1071174}.  Recently, \citet{MR3611491} introduced the concept
of Monge--Kantorovich ranks \nb{and signs} for all distributions with convex and
compact supports, %(hence finite moments of all orders),
which is the first type of multivariate ranks that enjoys
distribution-freeness for a rich class of distributions.
\citet{hallin2017distribution} generalized this definition %further
by
refraining from moment assumptions and making the solution more
explicit\nb{.  He also adopted the new terminology \emph{center-outward ranks and signs}}.
\nb{\citet{hallin2020distribution} further showed that center-outward ranks
and signs are not only distribution-free, but also essentially maximal ancillary,
which can be interpreted as ``maximal distribution-free'' in view of \citet{MR110115}.}
As shall be seen soon, the explicit nature of the solution
is important as it allows for more delicate manipulations and
ultimately allows us to form a test statistic of $H_0$ whose limiting
null distribution can be determined.  The limiting distribution
furnishes an accurate approximation to the statistic's null
distribution already for moderate sample sizes and allows us to avoid
computationally more involved permutation analysis.

% , and adopted the new terminology center-outward ranks.
%and \citet{del2018smooth} completed Hallin's definitions by constructing smooth cyclically monotone interpolations for empirical center-outward distribution functions.

In detail, our proposed test is based on applying distance covariance
to \nb{center-outward ranks and signs}. We show that the test is consistent
and distribution-free over the class of multivariate distributions
with nonvanishing (Lebesgue) probability densities; see
Section~\ref{sec:prelim} for the precise definition of this class. The
consistency is a consequence of a result of \citet{MR3886582}.  In
light of the prior work of \citet{MR2382665},
\citet{hallin2017distribution}, and \citet{MR3886582}, our major new
discovery is the form of the limiting null distribution of the test
statistic, which is established with all parameters given explicitly.
To this end, we study the weak convergence of U-statistics with a
``degenerate'' kernel and dependent (permutation) inputs, and derive a
general combinatorial non-central limit theorem (non-CLT) for double-
and multiple-indexed permutation statistics.  This theorem
is new and of independent interest beyond our particular application
of asymptotic calibration of the size of the independence test under
$H_0$.
% and accordingly allows us to
% directly implement the test without resorting to a computationally
% expensive permutation procedure.
%We also conduct some preliminary power analysis and evaluate the rate-optimality of our test in the class of Gaussian distributions when $p=q=1$. 

As we were completing this manuscript, we became aware of an
independent work by \citet{deb2019multivariate} who also proposed a
rank-distance-covariance-based independence test.  Their preprint was
posted a few days before ours and presents, in particular, a result
very similar to our Theorem~\ref{thm:null}. The derivations differ
markedly, however. \citeauthor{deb2019multivariate}'s proof uses
techniques based on characteristic functions, whereas we develop a
general combinatorial non-CLT theorem for double- and multiple-indexed
permutation statistics that can be applied to the considered statistic
as well as possible modifications.  There are further differences in
the precise setup of multivariate ranks\nb{: while} we base ourselves directly on recent work by \citet{hallin2017distribution} and by \citet{MR3886582}, \citet{deb2019multivariate} \nb{considered transports to the unit cube rather than the unit ball (see Definition~\ref{def:centerdistr} below) and} present weakened assumptions in the definition of the ranks.

The rest of the paper is organized as
follows. Section~\ref{sec:prelim} introduces \nb{center-outward 
ranks and signs}, and Section~\ref{sec:tests} specifies the proposed
test. Section~\ref{sec:theory} gives the theoretical analysis,
including the combinatorial non-CLT and a study of the proposed test.
Computational aspects are discussed in
Section~\ref{sec:computation}, and numerical studies of the
finite-sample behavior of our test {and an analysis of stock market data} are presented
in Section~\ref{sec:simulation}. All proofs are relegated to a
supplement.

\paragraph{Notation.}

The sets of real and positive integer numbers are denoted
$\R$ and $\Z_+$, respectively.  
For $n\in\Z_+$, we define $\zahl{n}=\{1,2,\ldots,n\}$.
%For $m\in\zahl{n}$, we define $(n)_m=n!/(n-m)!$ and $I^{n}_{m}=\{(i_1,\dots,i_m): i_j\in\zahl{n},~i_j\ne i_k~\text{if}~j\ne k\}$.
%and write $\mathcal{P}_n$ for the set of all $n!$ permutations of $\zahl{n}$. 
%The cardinality of a set $\cS$ is written as $\#\cS$. 
We write $\{x_1,\dots,x_n\}$ and $\{x_i\}_{i=1}^{n}$ for the
multiset consisting of (possibly duplicate) elements $x_1,\dots,x_n$.
%$[\underbrace{x_1,\dots,x_1}_{m_1},\dots,\underbrace{x_k,\dots,x_k}_{m_k}]$,
%$[x_1^{(m_1)},\dots,x_k^{(m_k)}]$, and $[x_i^{(m_i)}]_{i=1}^{k}$ denote the multiset consisting of $m_i$ copies of $x_i$ for all $i\in\zahl{k}$, 
%$[y_1,\dots,y_n]$ and $[y_i]_{i=1}^{n}$ denote the multiset consisting of elements $y_1,\dots,y_n$ (not necessarily distinct).
We use $[x_1,\dots,x_n]$ %, $[x(1),\dots,x(n)]$, 
and $[x_i]_{i=1}^{n}$ to denote sequences. %\fbox{revise it}
%Let $(n)_r$ denote $n!/(n-r)!$, and
%$I_{S}^{r}$ denote the set of all $(\#S)_r$ possible permutations (without repetition) of size $r$ from set $S$.
%$I_{n}^{r}$ denote the set of all $(n)_r$ possible permutations (without repetition) of size $r$ from set $\zahl{n}$.
A permutation of a multiset $\cS=\{x_1,\dots,x_n\}$ is a sequence $[x_{\sigma(i)}]_{i=1}^{n}$, 
%and an $r$-permutation of $\cS$ is a sequence $(x_{\sigma(i)})_{i=1}^{r}$ where $r\le n$,
where $\sigma$ is a bijection from $\zahl{n}$ to itself.
%For a multiset $\cM=[y_1,\dots,y_n]$,
%a permutation of $\cM$ is a sequence $(y_{\sigma(i)})_{i=1}^{n}$.
The family of all distinct permutations of a multiset $\cS$ is denoted $\rP(\cS)$. %\fbox{so all elements of $\rP(\cS)$ are sequences, right?}
%and $\rP_r(\cS)$ denote the family of all distinct $r$-permutations of a set $\cS$. 
The Euclidean norm of $\mv\in\R^d$ is written $\norm{\mv}$. %\fbox{no subscript 2} % and 
%Let $\fM=[M_{jk}]\in \R^{d\times d}$ be a real $d\times d$ matrix. 
%and $I,J$ be two subsets of $\zahl{d}$.  
%Then $\mv_I$ is the sub-vector of $\mv$ with entries
%indexed by $I$, i.e., $\mv_I=(v_{i_1},v_{i_2},\ldots,v_{i_{\# I}})^\top$ with $i_1<i_2<\cdots<i_{\# I}$ and $\{i_1,\ldots,i_{\# I}\}=I$. Both $\fM_{I,J}$ and $\fM[I,J]$ are used to refer to the sub-matrix of $\fM$ with 
%rows indexed by $I$ and columns indexed by $J$. 
%For $0<s<\infty$, we define the  $L_s$-norm of the vector $\mv$ as 
%$\norm{\mv}_s:=(\sum |v_i|^s)^{1/s}$.
%Let $\norm{\fM}_{\max}:=\max_{jk}|M_{jk}|$ be the matrix elementwise maximum norm. 
%The matrix ${\rm diag}(\fM)\in \R^{p\times p}$ is the diagonal matrix
%whose diagonal is the same as that of $\fM$.  
We write $\fI_d$ and $\fJ_d$ for the identity matrix 
and all-ones matrix in $\R^{d\times d}$, respectively. For a sequence of vectors $\mv_1,\ldots,\mv_d$, we use $(\mv_1,\ldots,\mv_d)$ as a shorthand of $(\mv_1^\top,\ldots,\mv_d^\top)^\top$.
For a function $f:\cX\to \R$, we define
$\norm{f}_{\infty}:=\max_{x\in\cX}|f(x)|$.  
The greatest integer less than or equal to $x\in\R$ is denoted $\lfloor x\rfloor$.  
The symbol $\ind(\cdot)$ stands for the indicator function.  
Throughout, $c$ and $C$ refer to positive absolute constants
whose values may differ in different parts of the paper.
For any two real sequences $[a_n]_n$ and $[b_n]_n$, we
write $a_n=O(b_n)$ if there exists $C>0$ such that $|a_n|\leq C|b_n|$
for all $n$ large enough,
%For any two real sequences $\{a_n\}$ and $\{b_n\}$, we
%write $a_n \lesssim b_n$, $a_n=O(b_n)$, or equivalently
%$b_n \gtrsim a_n$, if there exists $C>0$ such that $|a_n|\leq C|b_n|$
%for any large enough $n$. We write $a_n \asymp b_n$ if both
%$a_n \lesssim b_n$ and $a_n\gtrsim b_n$ hold. We 
and $a_n=o(b_n)$ if for any $c>0$, $|a_n|\leq c|b_n|$ holds for all
$n$ large enough.  The symbols $\bS_d$, $\overline\bS_d$, and
$\cS_{d-1}$ stand for the open unit ball, closed unit ball, and
unit sphere in $\R^d$, respectively.
%$X\stackrel{\sf d}{=}Y$ if two random variables $X$ and $Y$ have the
%same distribution, and 
We use  $\stackrel{\sf d}{\longrightarrow}$ and 
%$\stackrel{\sf p}{\longrightarrow}$, $\stackrel{\sf L^{r}}{\longrightarrow}$, 
$\stackrel{\sf a.s.}{\longrightarrow}$ to denote convergence in distribution
% convergence in probability, convergence in the $r$-th mean, 
and almost surely. For any random vector $\mX$, we use $P_{\mX}$ to represent its probability measure.

%\fbox{double check if all notation has been used}

\section{\nb{Center-outward ranks and signs}}\label{sec:prelim}

In this section, we introduce necessary background on \nb{center-outward ranks and signs}.
%the concepts of center-outward distribution functions, ranks, and empirical counterparts that will be used to construct our test statistic. 
As in \citet{hallin2017distribution}, we will be focused on the family
of absolutely continuous distributions on $\R^d$ that have a
nonvanishing (Lebesgue) probability density
(Definition~\ref{def:nonvanish} below).  In what follows it is
understood that the dimension $d$ could be larger than 1 and that all
considered probability measures are fixed, and not to be changed with
the sample size 
$n$ in particular.

\begin{definition}
  \label{def:nonvanish}
  Let $P$ be an absolutely continuous probability measure on $\R^d$
  with (Lebesgue) density $f$.  Such $P$ is said to be a
  \emph{nonvanishing probability measure/distribution} if for all
  $D>0$ there exist constants $\Lambda_{D;f}\ge\lambda_{D;f}>0$ such
  that $\lambda_{D;f}\le f(\mx)\le \Lambda_{D;f}$ for all
  $\lVert\mx\rVert\le D$.  We write $\cP_d$ for the family of all
  nonvanishing probability measures/distributions on $\R^d$.
% Stroock (1999) p.11 / Wikipedia page
\end{definition}

The considered generalization of ranks to higher dimensions rests on the
following concept of a center-outward distribution function, whose
existence and \nb{almost everywhere} uniqueness within the family $\cP_d$ is
guaranteed by the Main Theorem in \citet[p.~310]{MR1369395}.

\begin{definition}[Definition 4.1 in \citealp{hallin2017distribution}]
  \label{def:centerdistr}
  The \emph{center-outward distribution function} $\fF_{\pm}$ of a
  probability measure $P\in\cP_d$ is the \nb{almost everywhere} 
  unique function that (i) is 
  the gradient of a convex function on $\R^d$, (ii) maps $\R^d$ to the
  open unit ball $\bS_d$, and (iii) pushes $P$ forward to $U_d$, where
  $U_d$ is the product of the uniform measure on $[0,1)$ (for the
  radius) and the uniform measure on the unit sphere $\cS_{d-1}$.
  To be explicit, property (iii) requires
  $U_d(B)=P(\fF_{\pm}^{-1}(B))$ for any Borel set $B\subseteq \bS_d$.
%The corresponding center-outward quantile function is $\fQ_{\pm}:=\fF_{\pm}^{-1}$. We also define quantile regions as $\bC(q):=\fQ_{\pm}(q\overline\bS_d)$ and  quantile contours as $\cC(q):=\fQ_{\pm}(q\cS_{d-1})$. The elements of $\fQ_{\pm}(\mathbf{0})$ are called center-outward medians.
\end{definition}

If $\mX\sim P\in\cP_d$ and we further have
$\E\lVert\mX\rVert^2<\infty$, then
% an explicit form of $\fF_{\pm}$ exists. As a matter of fact, in this regard,
the center-outward distribution function $\fF_{\pm}$ of $P$ coincides with the $L_2$-optimal transport  from $P$ to $U_d$ \citep[Theorem 9.4]{MR2459454}, i.e., it is the \nb{almost everywhere} unique solution to the following optimization problem,
\begin{equation}\label{eq:monge}
\inf_{T}\int_{\R^d}\Bigl\Vert T(\mx)-\mx\Bigr\Vert^2 dP~~~\text{subject to }T_{\sharp} P=U_d,
\end{equation}
where $T_{\sharp} P$ denotes the push forward of $P$ under map $T$.
In other words, the optimization is done over all Borel-measurable
maps from $\R^d$ to $\R^d$ pushing $P$ forward to $U_d$.  Assuming
further that the Caffarelli's regularity conditions including
compactness of support \citep[Lemma~2.1]{MR3611491} hold, $\fF_{\pm}$
coincides with the Monge--Kantorovich vector rank transformation $R_P$
proposed in Definition 2.1 in \citet{MR3611491}.  Lastly, it can be
easily checked that when $d=1$, $\fF_{\pm}$ reduces to $2F-1$, where
$F$ is the usual cumulative distribution function.

In dimension $d=1$, the distribution function $F$ determines the
underlying probability distribution $P$.  A natural question is then
whether $\fF_{\pm}$ similarly preserves all information about a
distribution $P\in\cP_d$ when $d>1$.  That this is indeed the case
turns out to be highly nontrivial, and was not resolved until very
recently. The following proposition shows that $\fF_{\pm}$ is
a homeomorphism from $\R^d$ to $\bS_d$ \nb{except for a compact set with Lebesgue
measure zero}, indicating that all the
information about the probability measure $P\in\cP_d$ can be captured
using $\fF_{\pm}$.  This proposition will play a key role in our later
justification of the consistency of our proposed test (Theorem
\ref{thm:consi}).

\begin{proposition}[Theorem 1.1 in \citealp{MR3886582}; Propositions 4.1, 4.2 in \citealp{hallin2017distribution}]\label{prop:Figalli}
  Let $P\in\cP_d$, with center-outward distribution function
  $\fF_{\pm}$. Then,
\begin{enumerate}[label=(\roman*),itemsep=-.5ex]
\item $\fF_{\pm}$ is a probability integral transformation of $\R^d$, that is, $\mX\sim P$ iff $\fF_{\pm}(\mX)\sim U_d$;
\item The set $\fF_{\pm}^{-1}(\bm{0})$ is compact and of Lebesgue
  measure zero.  The restrictions of $\fF_{\pm}$ and $\fF_{\pm}^{-1}$
  to $\bR^d\backslash \fF_{\pm}^{-1}(\bm{0})$ and
  $\bS_d\backslash \{\bm{0}\}$ are homeomorphisms between
  $\R^d\backslash \fF_{\pm}^{-1}(\bm{0})$ and
  $\bS_d\backslash \{\bm{0}\}$.  If $d=1,2$, then the set
  $\fF_{\pm}^{-1}(\bm{0})$ is \nb{a} singleton, and $\fF_{\pm}$ and
  $\fF_{\pm}^{-1}$ are homeomorphisms between $\R^d$ and $\bS_d$.
%\item if $P$ has Lebesgue density $f$, $f(\fx)=c_d^{-1} H_{\varphi^*}(\fx)\lVert\nabla\varphi(\fx)\rVert^{1-d}$, $\fx\in\R^d\backslash K$, where the norming constant $c_d:=2\pi^{d/2}/\Gamma(d/2)$ is the area of the unit sphere $\cS_{d-1}$ and $H_{\varphi^*}$ the Hessian of $\varphi^*$.
%\item for $d=1,2$, $\fF_{\pm}$ and $\fQ_{\pm}$ are homeomorphisms between $\R^d$ and $\bS_d$, respectively, and the center-outward median $\fQ_{\pm}(\mathbf{0})$ is uniquely defined; for $d\ge 3$, the restriction of $\fF_{\pm}$ and $\fQ_{\pm}$ to $\R^d\backslash\fQ_{\pm}(\mathbf{0})$ and $\bS_d\backslash\{\mathbf{0}\}$ are homeomorphisms between $\R^d\backslash\fQ_{\pm}(\mathbf{0})$ and $\bS_d\backslash\{\mathbf{0}\}$, and the center-outward medians form a compact null set;
%\item the quantile regions $\bC(\alpha):=\fF_{\pm}^{-1}(\alpha\overline\bS_d)$, with boundaries $\cC(\alpha):=\fF_{\pm}^{-1}(\alpha\cS_{d-1})$, are connected, compact, and nested as $\alpha\in[0,1)$ increases.
\end{enumerate}
\end{proposition}

We now move on to estimation of $\fF_{\pm}$ based on $n$ independent
copies of $\mX\sim P\in\cP_d$.  The considered estimator mimics the
empirical version of the Monge--Kantorovich problem \eqref{eq:monge},
and the key step is to ``discretize'' the unit ball $\bS_d$ to $n$
grid points.  In the following we sketch Hallin's approach to the
construction of such a grid point set, with a focus on how to form the grid
points when {$d\ge 2$}. To this end,
%We have the finite sample $(\mX_i)_{i=1}^{n}$ which yields empirical measure weakly converging to $P$, and as a counterpart,
%We  
%with corresponding empirical measure converging to $U_d$ 
%as follows: 
let us first factorize $n$ into the following form, whose existence is clear:
\nb{\begin{align}\label{eq:key}
n = n_R n_S+n_0,
~~~~~~n_R,n_S\in\Z_+,~0\le n_0<\min\{n_R,n_S\},
~~~\text{with}~n_R,n_S\to\infty~~~\text{as}~n\to\infty.
\end{align}
Next, consider $n_R n_S$ intersection points between
\begin{itemize}[label=--,itemsep=-.5ex]
\item the $n_R$ hyperspheres centered at $\bm0$ with radii
  $\frac{1}{n_R+1},\dots,\frac{n_R}{n_R+1}$, and
\item $n_S$ distinct unit vectors $\{\mr_1,\dots,\mr_{n_S}\}$.
\end{itemize}
The unit vectors in $\{\mr_1,\dots,\mr_{n_S}\}$ are selected such
that the uniform discrete distribution on this set converges weakly to
the uniform distribution on $\cS_{d-1}$. For $d=2$, we can choose unit
vectors such that the unit circle is divided into $n_S$ equal
arcs. For $d\ge 3$, the requirement is satisfied almost surely when
independently drawing $n_S$ unit vectors from the uniform distribution on $\cS_{d-1}$. %are such that the uniform discrete distribution on this set almost surely converges weakly to the uniform distribution on $\cS_{d-1}$
Moreover, it is not difficult to give a deterministic construction
that serves our purpose; see Section~B in the supplement.
}

\begin{definition}\label{def:aug}
{When $d\ge2$, }
  the augmented grid
  $\cG^{d}_{n_0,n_R,\nb{n_S}}$ is the multiset consisting of $n_0$
  copies of the origin $\bm{0}$ whenever $n_0>0$ and \nb{the intersection points
  $\{\frac{j}{n_R+1}\mr_{k}: j\in\zahl{n_R}, k\in\zahl{n_S}\}$.}
{When $d=1$, letting $n_S=2$, $n_R=\lfloor n/n_S\rfloor$, and $n_0=n-n_Rn_S$, the augmented grid $\cG^{d}_{n_0,n_R,n_S}$ is the multiset consisting of the origin $0$ whenever $n_0>0$ and the points $\{\pm\frac{j}{n_R+1}: j\in\zahl{n_R}\}$.}
\end{definition}

\begin{proposition}\label{prop:unif}
\nb{As long as the uniform discrete distribution on $\{\mr_1,\dots,\mr_{n_S}\}$ converges weakly to the uniform distribution on $\cS_{d-1}$, }
  the uniform discrete distribution on the augmented grid
  $\cG^{d}_{n_0,n_R,n_S}$, which assigns mass $n_0/n$ to the
  origin and mass $1/n$ to every other grid point, 
%   The distribution with probability mass $n_0/n$ at the origin $\bm{0}$ and probability mass $1/n$ at each of points $\{\mt_{j_R,j_1,\ldots,j_{d-1}}; j_R\in\zahl{n_R},j_1\in\zahl{n_1},\ldots,j_{d-1}\in\zahl{n_{d-1}}\}$ with spherical coordinates $(r_{j_R},\varphi_{1,j_1},\dots,\varphi_{d-1,j_{d-1}})^\top$, where $r_{j}=j/(n_R+1)$ for $j\in\zahl{n_R}$ and $\varphi_{m,j}$'s are defined in \eqref{eq:coord}, 
% called the uniform distribution over the augmented grid $\cG^{d}_{n_0,n_R,\bm{n_S}}$, 
weakly converges to $U_d$.
\end{proposition}

%{
%\begin{remark}\label{remark:new1}
%We note that the augmented grid $\cG^{d}_{n_0,n_R,n_S}$ reduces to the one given by \citet[p.~14]{hallin2017distribution} when $d=2$. 
%In addition, we emphasize that the construction of the test does not
%depend on this particular way of choosing a grid. For example, any
%deterministic set of $n$ points in $\bS_d$ in general dimension $d$
%such that the uniform discrete distribution on this set converges
%weakly to $U_d$  would yield a similarly justified test.
%\end{remark}
%}

We are now ready to introduce Hallin's estimator, $\fF_{\pm}^{(n)}$,
of $\fF_{\pm}$.  It is defined via the optimal coupling between the
observed data points and the augmented grid
$\cG^{d}_{n_0,n_R,\nb{n_S}}$.

\begin{definition}[Definition~4.2 in \citealp{hallin2017distribution}]\label{def:empdistr}
Let $\mx_1,\dots,\mx_n$ be data points in $\R^d$.
Let $\cT$ be the collection of all bijective mappings between the multiset $\{\mx_{i}\}_{i=1}^{n}$ and the augmented grid $\cG^{d}_{n_0,n_R,\nb{n_S}}$. The \emph{empirical center-outward distribution function} is defined as
\begin{equation}\label{eq:assignment}
\fF_{\pm}^{(n)}:=\argmin_{T\in\cT}\sum_{i=1}^{n}\Big\lVert\mx_{i}-T(\mx_{i})\Big\rVert^2,
\end{equation}
\nb{the \emph{center-outward rank} of $\mx_i$ is defined 
as $(n_R+1)\lVert\fF^{(n)}_{\pm}(\mx_i)\rVert$,
and the \emph{center-outward sign} of $\mx_i$ is defined 
as $\fF^{(n)}_{\pm}(\mx_i)\rVert/\lVert\fF^{(n)}_{\pm}(\mx_i)\rVert$ 
if $\lVert\fF^{(n)}_{\pm}(\mx_i)\rVert\ne 0$, and $\bm0$ otherwise.} 
%The following concepts are associated with the empirical center-outward distribution function $\fF_{\pm}^{(n)}$ 
%\begin{itemize}[label=--]
%\item center-outward ranks $R_{\pm,i}^{(n)}:=(n_R+1)\lVert\fF_{\pm}^{(n)}(\fZ_{i}^{(n)})\rVert$,
%\item center-outward signs $\fS_{\pm,i}^{(n)}:=\mathbf{0}$ if $\fF_{\pm}^{(n)}(\fZ_{i}^{(n)})=\mathbf{0}$, and $\fF_{\pm}^{(n)}(\fZ_{i}^{(n)})/\lVert\fF_{\pm}^{(n)}(\fZ_{i}^{(n)})\rVert$ otherwise,
%\item center-outward (empirical) quantile regions $\bC_{\pm,\fZ^{(n)}}^{(n)}(\frac{j}{n_R}):=\{\fZ_{i}^{(n)}|R_{\pm,i}^{(n)}\le\frac{j}{n_R+1}\}$,
%\item center-outward (empirical) quantile contours $\cC_{\pm,\fZ^{(n)}}^{(n)}(\frac{j}{n_R}):=\{\fZ_{i}^{(n)}|R_{\pm,i}^{(n)}=\frac{j}{n_R+1}\}$.
%\end{itemize}
\end{definition}

The following two propositions from \citet{hallin2017distribution}
give the Glivenko--Cantelli strong consistency and
distribution-freeness of the empirical center-outward distribution 
function.  Both shall play key roles for the limiting null distribution 
and  asymptotic consistency of the test statistic that will be 
proposed in Section~\ref{sec:tests}.

\begin{proposition}[Glivenko--Cantelli, Proposition~5.1 in \citealp{hallin2017distribution}, Theorem~3.1 in \citealp{del2018smooth}]\label{prop:GC}
Let $\P\in\cP_d$ with center-outward distribution function $\fF_{\pm}$, 
and let $\mX_1,\dots,\mX_n$ be i.i.d. with distribution $\P$ with empirical center-outward distribution function $\fF_{\pm}^{(n)}$. 
Then 
\begin{equation}
\max_{1\le i\le n}\Big\lVert\fF_{\pm}^{(n)}(\mX_{i})-\fF_{\pm}(\mX_{i})\Big\rVert\stackrel{\sf a.s.}{\longrightarrow}0
%~~~\text{and}~~~
%\max_{1\le i\le n}\Big\lVert\fG_{\pm}^{(n)}(\mX_{i})-\fG_{\pm}(\mX_{i})\Big\rVert\stackrel{\sf a.s.}{\longrightarrow}0
\end{equation}
when $n\to\infty$ and \eqref{eq:key} holds.
\end{proposition}

%% Lastly, we state the distribution-freeness property of Hallin's rank. %But before that, let's first recall the definition of permutation of a multiset.

\begin{proposition}[Distribution-freeness, Proposition~6.1(ii) in \citealp{hallin2017distribution}, \nb{Proposition~2.5(ii) in \citealp{hallin2020distribution}}]\label{prop:distrfree}
Let $\mX_1,\dots,\mX_n$ be i.i.d. with distribution $P\in\cP_d$. 
Let $\fF_{\pm}^{(n)}$ be their empirical center-outward distribution
function.
% for $\{\mX_i\}_{i=1}^{n}$. 
%Recall that $n$ is factorized into $n=n_Rn_S+n_0$. 
Then for any decomposition $n_0,n_R,\nb{n_S}$ of $n$, the random
vector $[\fF_{\pm}^{(n)}(\mX_1),\dots,\fF_{\pm}^{(n)}(\mX_n)]$ is
uniformly distributed over $\rP(\cG^{d}_{n_0,n_R,\nb{n_S}})$.  The
latter set is comprised of all permutations of the multiset
$\cG^{d}_{n_0,n_R,\nb{n_S}}$; recall the notation introduced at the
end of Section~\ref{sec:introduction}.
%all the $n!/n_0!$ permutations (with identical objects) of 
%the augmented grid described in Definition~\ref{def:empdistr}, assuming the $n_0$ copies of the origin $\bm{0}$ are indistinguishable.
%the augmented grid consisting of $n_0$ copies of the origin $\bm{0}$ whenever $n_0>0$ and points $\{\mt_{j_R,j_1,\ldots,j_{d-1}}; j_R\in\zahl{n_R},j_1\in\zahl{n_1},\ldots,j_{d-1}\in\zahl{n_{d-1}}\}$ with spherical coordinates $(r_{j_R},\varphi_{1,j_1},\dots,\varphi_{d-1,j_{d-1}})^\top$.
%the augmented grid $\cG^{d}_{n_0,n_R,\bm{n_S}}$ defined in Definition~\ref{def:aug}.
\end{proposition}

%The construction of smooth cyclically monotone interpolations for empirical center-outward distribution functions can be found in \citet{del2018smooth}. The detail is omitted here while we illustrate non-smooth cyclically monotone interpolations in the paper to see their idea. 

%\begin{figure}[!htbp]
%\centering
%\includegraphics[scale=0.5,trim={1cm 1cm 0 1cm},clip]{./fig/plotseed2}
%\caption{A non-smooth cyclically monotone interpolation of empirical center-outward distribution function.}\label{fig:nonsmooth}
%\end{figure}

\section{A distribution-free test of independence}\label{sec:tests}

This section introduces the proposed distribution-free test of $H_0$ in~(\ref{eq:H0}) built on \nb{center-outward ranks and signs}. 
%We now turn to tests of the joint independence hypothesis . 
The main new methodological idea is simple: We propose to plug the
calculated \nb{center-outward ranks and signs}, instead of the original data,
into the consistent test statistics presented in the introduction
(Section~\ref{sec:introduction}).  The distribution theory
for the proposed test statistic, however, is non-trivial and requires
new technical developments, which shall be detailed in Section~\ref{sec:theory}.

To illustrate our idea, we will focus on one particular consistent
test statistic in the sequel, namely, the distance covariance of
\citet{MR2382665}. Other choices including HSIC and more recent
proposals like the ball covariance proposed in \citet{MR4078465}
shall be discussed in Section~\ref{sec:theory} following the
presentation of our general combinatorial
non-CLT.% and a separate discussion section.

%One popular choice  is distance covariance since this measure is consistent against all dependent alternatives, i.e, is zero if and only if $H_0$ holds, under the existence of finite first moments of both $\lVert\mX\rVert$ and $\lVert\mY\rVert$ \citep[Theorem 3.11]{MR3127883}.

We begin with details on the distance covariance that are necessary to
convey the main idea. We first introduce a representation of the
associated measure of dependence.

\begin{definition}[Distance covariance measure of dependence, \citet{MR2382665}]\label{def:dcov}
Let $\mX\in\R^p$ and $\mY\in\R^q$ be two random vectors with
$\E(\lVert\mX\rVert+\lVert\mY\rVert)<\infty$, and let $(\mX',\mY')$ be an independent copy of $(\mX,\mY)$.
% and let $d_{\cX},d_{\cY}$ denote the $L_2$ metrics in spaces $\R^p$ and $\R^q$, respectively. 
The distance covariance of $(\mX,\mY)$  is defined as 
\begin{equation}\label{eq:dcov}
\dCov^2(\mX,\mY):=\E(d_{\mX}(\mX,\mX')d_{\mY}(\mY,\mY')),
\end{equation}
which is finite and uses the kernel function 
% $d_{\mX}(\cdot,\cdot)$ defined as 
\begin{align}\label{eq:fproj}
d_{\mX}(\mx,\mx')=d_{P_{\mX}}(\mx,\mx') :=
\lVert\mx-\mx'\rVert - E\lVert\mx-\mX_2\rVert 
- E\lVert\mX_1-\mx'\rVert
+ E\lVert\mX_1-\mX_2\rVert, 
\end{align}
%\begin{align*}
%d_{\mu}(\mx,\mx')&:=d(\mx,\mx')-a_{\mu}(\mx)-a_{\mu}(\mx')+D(\mu),\\
%a_{\mu}(\mx)&:=\int d(\mx,\mx') d\mu(\mx')~~~\text{and}~~~D(\mu):=\int d(\mx,\mx') d\mu(\mx)d\mu(\mx'),
%\end{align*}
and its analogue $d_{\mY}(\my,\my')$. Here $\mX_1$ and $\mX_2$ are independent and both follow the distribution $P_{\mX}$.
\end{definition}

The finiteness of $\dCov^2(\mX,\mY)$ in \eqref{eq:dcov} was proved by \citet[Proposition~2.3]{MR3127883}.
It can be shown that under the same conditions as in Definition~\ref{def:dcov}, 
\[\dCov^2(\mX,\mY)=\frac14\E(s(\mX_{1},\mX_{2},\mX_{3},\mX_{4})s(\mY_{1},\mY_{2},\mY_{3},\mY_{4})),\]
where $(\mX_{1},\mY_{1}),\dots,(\mX_{4},\mY_{4})$ are independent copies of $(\mX,\mY)$ and
\[
s(\mt_{1},\mt_{2},\mt_{3},\mt_{4}):=\lVert\mt_{1}-\mt_{2}\rVert+\lVert \mt_{3}-\mt_{4}\rVert-\lVert \mt_{1}-\mt_{3}\rVert-\lVert \mt_{2}-\mt_{4}\rVert;
%\yestag\label{eq:function-s}
\]
see also \citet[Sec.~3.4]{MR3178526}.
Accordingly, we have an unbiased estimator of the distance covariance between $\mX$ and $\mY$ as follows.
\begin{definition}[Sample distance covariance, \citet{MR3053543}]\label{def:sdcov}
Let $(\mX_1,\mY_1),\dots,(\mX_n,\mY_n)$ be independent copies of $(\mX,\mY)$ with $\mX\in\R^p$, $\mY\in\R^q$, $\E(\lVert\mX\rVert+\lVert\mY\rVert)<\infty$. The sample distance covariance is defined as
\begin{equation}\label{eq:defnew}
\dCov^2_n\Big([\mX_{i}]_{i=1}^{n},[\mY_{i}]_{i=1}^{n}\Big)=\mbinom{n}{4}^{-1}
\sum_{1\le i_1<\dots<i_4\le n}K\Big((\mX_{i_1},\mY_{i_1}),\dots,(\mX_{i_4},\mY_{i_4})\Big),
\end{equation}
where
\begin{align}\label{eq:kernel}
K\Big((\mx_{1},\my_{1}),\dots,(\mx_{4},\my_{4})\Big)
&:=\frac{1}{4\cdot4!}\sum_{[i_1,\dots,i_4]\in\rP(\zahl{4})} s(\mx_{i_1},\mx_{i_2},\mx_{i_3},\mx_{i_4})
s(\my_{i_1},\my_{i_2},\my_{i_3},\my_{i_4}),\\
\text{and recall}~~~s(\mt_{1},\mt_{2},\mt_{3},\mt_{4})&:=\lVert\mt_{1}-\mt_{2}\rVert+\lVert \mt_{3}-\mt_{4}\rVert-\lVert \mt_{1}-\mt_{3}\rVert-\lVert \mt_{2}-\mt_{4}\rVert.\notag
\end{align}
\end{definition}

{
The following is a direct consequence of Lemma~1 in \citet{MR3798874}.

\begin{proposition}\label{prop:equiv}
Definition~1 in \citet{MR3053543}, Equation~(3.2) in \citet{MR3269983}, 
Definition~5.3 (U-statistic) in \citet{jakobsen2017distance}, 
and Definition \ref{def:sdcov} above are equivalent. 
\end{proposition}
}

%However, tests based on sample distance covariance have some disadvantages. 
%First, although showed that the consistency holds if both $\lVert\mX\rVert$ and $\lVert\mY\rVert$ have only finite first moments, 
%First, moment assumptions seem inappropriate for the independence testing problem. Secondly, these tests involve permutation testing procedures, which are computationally costly. 
%Alternatively, we consider applying distance covariance to Hallin's multivariate distribution function.

We are now ready to describe our distribution-free test of
independence, which combines distance covariance with 
\nb{center-outward ranks and signs}.

\begin{definition}[The proposed distribution-free test statistic]\label{def:feasible}
Let $(\mX_1,\mY_1),\dots,(\mX_n,\mY_n)$ be independent copies of $(\mX,\mY)$ 
with $P_{\mX}\in\cP_p$ and $P_{\mY}\in\cP_q$. 
Let $\fF_{\mX,\pm}^{(n)}$ and $\fF_{\mY,\pm}^{(n)}$ be the empirical center-outward distribution functions for $\{\mX_i\}_{i=1}^{n}$ and $\{\mY_i\}_{i=1}^{n}$.
We define the test statistic
\begin{equation}\label{eq:feasible}
\widehat M_n:=n\cdot \dCov^2_n\Big([\fF_{\mX,\pm}^{(n)}(\mX_i)]_{i=1}^{n},[\fF_{\mY,\pm}^{(n)}(\mY_i)]_{i=1}^{n}\Big).
\end{equation}
\end{definition}

By Proposition~\ref{prop:distrfree}, the statistic $\widehat M_n$ is
distribution-free under the independence hypothesis $H_0$ in
\eqref{eq:H0}.  Hence, an exact critical value for rejection of $H_0$
can be approximated via Monte Carlo simulation.  Numerically less
demanding, one could instead adopt the critical value 
{based on the limiting null distribution of $\widehat M_n$ derived from the following theorem. }
%that comes from
%the asymptotic form of the infeasible test statistic $\widetilde M_n$
%shown in Corollary \ref{coro:jakobsen}, i.e.,  $H_0$ is rejected if
%\[
%\widehat M_n>Q_{1-\alpha},
%\]
%where $Q_{1-\alpha}$ is defined as in \eqref{eq:testinfea}. 

{
\begin{theorem}[Limiting null distribution]\label{thm:null}
Let $(\mX_1,\mY_1),\dots,(\mX_n,\mY_n)$ be independent copies of $(\mX,\mY)$ 
with $P_{\mX}\in\cP_p$ and $P_{\mY}\in\cP_q$, 
and $\mX$ and $\mY$ are independent.
Then we have
\begin{equation}
\widehat M_n \stackrel{\sf d}{\longrightarrow} \sum_{k=1}^{\infty}\lambda_k(\xi_k^2-1),
\end{equation}
as $n\to\infty$ and \eqref{eq:key} holds, 
where $\lambda_k$, $k\in\Z_+$, are the non-zero eigenvalues of the integral equation
\begin{equation}
\E\Big(d_{\mU}(\bmu,\mU)d_{\mV}(\mv,\mV) \phi(\mU,\mV)\Big)=\lambda \phi(\bm{u},\mv),
\end{equation}
in which $d_{\mU}(\bmu,\bmu')$ and $d_{\mV}(\mv,\mv')$ are defined as in \eqref{eq:fproj}, $\mU \sim U_p$ and $\mV \sim U_q$ are independent, and $[\xi_k]_{k=1}^{\infty}$ is a sequence of independent standard Gaussian random variables. 
\end{theorem}
}

{
\begin{remark}
In Section~\ref{sec:theory} we will prove Theorem \ref{thm:null} rigorously. 
Intuitively, it is helpful to first consider the following ``oracle'' test
statistic $\widetilde M_n$:
\[
\widetilde M_n:=n\cdot \dCov^2_n\Big([\fF_{\mX,\pm}(\mX_i)]_{i=1}^{n},[\fF_{\mY,\pm}(\mY_i)]_{i=1}^{n}\Big),
\]
where $\fF_{\mX,\pm}$ and $\fF_{\mY,\pm}$ denote 
the center-outward distribution functions of $\P_{\mX}$ and $\P_{\mY}$, respectively. 
The infeasibility stems from the use of the (population) center-outward
distribution functions. 
One can easily verify using the asymptotic  theory of degenerate U-statistics that under the null
\[
\widetilde M_n \stackrel{\sf d}{\longrightarrow} \sum_{k=1}^{\infty}\lambda_k(\xi_k^2-1),
\]
where $[\lambda_k]_{k=1}^{\infty}$ and $[\xi_k]_{k=1}^{\infty}$ are defined as in Theorem \ref{thm:null}. 
%It is reasonable to guess that $\widehat M_n$ and $\widetilde M_n$ are asymptotically equivalent under the null in view of Proposition~\ref{prop:GC}. 
Somewhat surprising to us, the limiting null distribution of
$\widehat M_n$ is exactly the same as that of $\widetilde M_n$. 
\end{remark}
}

Therefore, for any pre-specified significance level $\alpha\in(0,1)$, our proposed test is hence
\begin{equation}\label{eq:test}
\mathsf{T}_{\alpha}:= \ind\Big(\widehat M_n>Q_{1-\alpha}\Big),~~~Q_{1-\alpha}:=\inf\Big\{x\in\R:\P\Big(\sum_{k=1}^{\infty}\lambda_k(\xi_k^2-1)\le x\Big)\ge 1-\alpha \Big\}.
\end{equation}
%For any pre-specified significance level $\alpha\in(0,1)$, our proposed test is hence 
%\begin{equation}\label{eq:test}
%\mathsf{T}_{\alpha}:= \ind\Big(\widehat M_n>Q_{1-\alpha}\Big).
%\end{equation}
Consequently, by Theorem \ref{thm:null}, 
\begin{equation}\label{eq:asymp-null}
\Pr(\mathsf{T}_{\alpha}=1\given H_0)=\alpha+o(1).
\end{equation}

It should be highlighted that\nb{, thanks to distribution-freeness,
given any fixed dimensions $p$ and $q$, 
the asymptotically small term in \eqref{eq:asymp-null}
is independent of the underlying distributions, 
and converges to zero uniformly over all the underlying distributions with $P_{\mX}\in\cP_p$, $P_{\mY}\in\cP_q$, and $\mX$ independent of $\mY$. 
T}he values of $\lambda_k$'s, 
and hence also the critical value $Q_{1-\alpha}$ itself, are
\nb{distribution-free} and only depend on the dimensions $p$ and
$q$.  
The critical value may thus be calculated using numerical methods for
each pair of $p$ and $q$.  Details will be described in Section \ref{subsec:eigv}. Table~\nb{C.1} in the supplement further records the critical values at significance levels $\alpha=0.1,0.05,0.01$ 
for $(p,q)=(1,1),(1,2),\ldots,(10,10)$ with accuracy $5\cdot 10^{-3}$.

Due to (i) the near-homeomorphism property of the center-outward distribution function shown in
Proposition \ref{prop:Figalli}; (ii) the strong Glivenko-Cantelli
consistency of \nb{empirical center-outward distribution functions} 
shown in Proposition \ref{prop:GC}; and
(iii) the fact that the distance covariance measure of dependence is zero if and only if $H_0$ holds under finiteness of marginal first
% Lyons (2013) p.3297 the condition there that X and Y have finite second moments is thus seen to be superfluous
 moments \citep[Theorem 3.11]{MR3127883}, it holds that $\widehat M_n$
 is asymptotically consistent \nb{and the corresponding test $\mathsf{T}_{\alpha}$ is consistent}.  This fact is summarized in the following theorem.

\begin{theorem}[\nb{Consistency}]\label{thm:consi}
  Let $(\mX_1,\mY_1),\dots,(\mX_n,\mY_n)$ be independent copies of $(\mX,\mY)$, 
  where $P_{\mX}\in\cP_p$ with center-outward distribution function $\fF_{\mX,\pm}$, 
  and $P_{\mY}\in\cP_q$ with center-outward distribution function $\fF_{\mY,\pm}$.
%Let $\fF_{\mX,\pm}^{(n)}(\cdot)$ and $\fF_{\mY,\pm}^{(n)}(\cdot)$ denote the empirical center-outward distribution functions for $(\mX_i)_{i=1}^{n}$ and $(\mY_i)_{i=1}^{n}$, respectively.
\nb{We then have, as long as $n\to\infty$ and \eqref{eq:key} holds,
\begin{equation}\label{eq:strconsi}
%\dCov^2_n\Big([\fF_{\mX,\pm}^{(n)}(\mX_i)]_{i=1}^{n},[\fF_{\mY,\pm}^{(n)}(\mY_i)]_{i=1}^{n}\Big)
\widehat M_n/n\;
\stackrel{\sf a.s.}{\longrightarrow}\;\dCov^2\Big(\fF_{\mX,\pm}(\mX),\fF_{\mY,\pm}(\mY)\Big),
\end{equation}
where $\dCov^2(\fF_{\mX,\pm}(\mX),\fF_{\mY,\pm}(\mY))\ge 0$ with equality if and only if $\mX$ and $\mY$ are independent.
%and equality holds if and only if $\mX$ and $\mY$ are independent.
%Moreover, if and only if $\mX$ and $\mY$ are independent, it holds that
%\begin{equation}
%\limsup_{n\to\infty}\E \Big[\dCov^2_n\Big([\fF_{\mX,\pm}^{(n)}(\mX_i)]_{i=1}^{n},[\fF_{\mY,\pm}^{(n)}(\mY_i)]_{i=1}^{n}\Big)\Big] = 0.
%\end{equation}
In addition, under any fixed alternative $H_1$, we obtain $\widehat M_n\stackrel{\sf a.s.}{\longrightarrow}\infty$ if $n\to\infty$ and \eqref{eq:key} holds, and thus}
\begin{equation}\label{eq:asymp-alter}
\nb{\Pr(\mathsf{T}_{\alpha}=1\given H_1)=1-o(1).}
\end{equation}
%This shows the consistency of the proposed test against any fixed alternative.}
\end{theorem}

%We remark that combining Theorem \ref{thm:null} and Theorem \ref{thm:consi} shows the consistency of the proposed test against any fixed alternative. 

{
We conclude this section with one more remark that discusses an interesting connection between the proposed test and a famous dependence measure, Blum--Kiefer--Rosenblatt's $R$ dependence measure  \citep{MR0125690}, when $p=q=1$. 

\begin{remark}\label{remark:new2}
In the univariate case ($p = q = 1$), the statistic $\widehat M_n/n$ is actually (up to a constant) a consistent estimator of Blum--Kiefer--Rosenblatt's $R$ measure of dependence  \citep{MR0125690}. In detail, Theorem~\ref{thm:consi} has shown that $\widehat M_n/n\stackrel{\sf a.s.}{\longrightarrow}
\dCov^2(\fF_{\mX,\pm}(\mX),\fF_{\mY,\pm}(\mY))$. 
When $X$ and $Y$ are both absolutely continuous, \citet[Lemma~10]{bergsma2006new} 
showed that
\[\frac14\dCov^2(X,Y)=\int \{F_{(X,Y)}(x,y)-F_{X}(x)F_{Y}(y)\}^2 dxdy,\]
where $F_{Z}(\cdot)$ denotes the cumulative distribution function of $Z$. This implies that
\[\frac1{16}\dCov^2(\fF_{X,\pm}(X),\fF_{Y,\pm}(Y))
=\int \{F_{(X,Y)}(x,y)-F_{X}(x)F_{Y}(y)\}^2 d F_{X}(x)d F_{Y}(y).\]
The right-hand side is Blum--Kiefer--Rosenblatt's $R$ and $\widehat M_n/(16n)$ converges to it almost surely. 
\end{remark}
}

\section{Theoretical analysis}\label{sec:theory}

This section provides the theoretical justification for the test
in~(\ref{eq:test}).  By Proposition~\ref{prop:distrfree}, both $[\fF_{\mX,\pm}^{(n)}(\mX_i)]_{i=1}^{n}$ and $[\fF_{\mY,\pm}^{(n)}(\mY_i)]_{i=1}^{n}$ are generated from uniform permutation measures. In view of Definition~\ref{def:feasible}, it is hence clear that under $H_0$ the test statistic $\widehat M_n$ is a summation over the product space of two uniform permutation measures, which belongs to the family of permutation statistics.

The study of permutation statistics %is at the center of probability
% theory, and
can be traced back at least to \citet{MR0011424}, who proved an
asymptotic normality result for single-indexed permutation statistics of the form $\sum_{i=1}^{n}x^{(n)}_{i}y^{(n)}_{\pi_{i}}$. Here $\mx^{(n)}$ and $\my^{(n)}$ are vectors that are possibly varying with $n$, and $\mpi$ is uniformly distributed on $\rP(\zahl{n})$. Later, \citet{MR0031670}, \citet{MR0044058}, \citet{MR0089560}, and \citet{MR130707}, among many others, generalized \citeauthor{MR0011424}'s results in different ways, and \citet{MR751577} gave a sharp Berry--Esseen bound for such permutation statistics using Stein's method.

Double-indexed permutation statistics, of the form $\sum_{i\ne j}A^{(n)}_{ij}B^{(n)}_{\pi_{i}\pi_{j}}$ with $\fA^{(n)}$ and $\fB^{(n)}$ as matrices possibly varying with $n$, are more difficult to tackle. They were first investigated by \citet{MR0010941},
who gave sufficient conditions for asymptotic normality.
% to hold for such double-indexed ones.
Later, various weakened conditions were introduced in, e.g., \citet[Chap.~4.1]{MR0175265}, \citet{MR230426}, \citet{MR236965}, \citet[Chap.~2.4]{Cliff1973Sa}, \citet{MR532242}, \citet{MR857081}, \citet{MR1032590}, and the Berry--Esseen bound was established in \citet{MR1474091}, \citet{MR2205339}, and \citet{MR2573554}.
%for the double-indexed permutation statistics of general form 
%which implies more general conditions such that the asymptotic normality holds. 

Despite this vast literature, there is a notable absence of results on
permutation statistics which, as its degenerate U-statistics
``cousins'', may weakly converge to a non-normal
distribution. %As a matter of fact, to our knowledge very few results exist for combinatorial non-CLTs \fbox{literature?}.
% The reason for such an absence could, of course, be well justified by
% a lack of motivation. However, quite interestingly and also clear from
% reading the context before,
Our analysis of $\widehat M_n$, however, hinges on such a
combinatorial non-CLT.  In the following, we present two general
theorems that fill the gap.

Before stating the two theorems, we introduce some notions needed.
For each $i=1,2$, let $\mZ_i$ be a random vector taking values in
$\Omega_i$, a compact subset of $\R^{p_i}$.  We consider triangular
arrays $\{\mz^{(n)}_{i;j},n\in\Z_+,j\in\zahl{n}\}$, for $i=1,2$, such
that the random variables with uniform discrete distributions on the respective
multisets $\{\mz^{(n)}_{i;j},j\in\zahl{n}\}$, denoted by
$\mZ^{(n)}_i$, weakly converge to $\mZ_i$ as $n\to\infty$.  We further
introduce an independent copy of $\mZ_i$, denoted $\mZ'_{i}$, and
independent copies of the $\mZ^{(n)}_{i}$, denoted ${\mZ^{(n)}_{i}}'$.
Finally, for $i=1,2$ and $n\in\Z_+$, let
$g^{(n)}_{i},g_{i}: \Omega_i\times \Omega_i\to\R$ be real-valued
functions, the former of which may change with $n$.

Our first theorem is then focused on double-indexed
permutation-statistics of the form
\begin{equation}
\widehat D^{(n)}=\mbinom{n}{2}^{-1}\sum_{1\le j_1< j_2\le n}
  g^{(n)}_1\Big(\mz^{(n)}_{1;j_1},\mz^{(n)}_{1;j_2}\Big)
  g^{(n)}_2\Big(\mz^{(n)}_{2;\pi_{j_1}},\mz^{(n)}_{2;\pi_{j_2}}\Big),
\end{equation}
where $\mpi$ is uniformly distributed on $\rP(\zahl{n})$.

%In contrast, our results are the first ones for non-normal limiting distributions.

\begin{theorem}\label{thm:general} 
  Assume that for each $i=1,2$, the functions
  $g^{(n)}_i,~n\in\Z_+$, and $g_i$ satisfy the following conditions:
\begin{enumerate}[label=(\roman*),itemsep=-.5ex]
\item\label{asm:sym2} each $g^{(n)}_{i}$ is symmetric, i.e., $g^{(n)}_{i}(\mz,\mz')=g^{(n)}_{i}(\mz',\mz)$ for all $\mz,\mz'\in\Omega_i$;
\item\label{asm:cnt2} the family $g^{(n)}_{i}$, $n\in\Z_+$, is equicontinuous; %which means for every $\epsilon>0$ there exists $\delta>0$ such that $\lVert(\mz-\mz'',\mz'-\mz''')\rVert<\delta$ implies $\lvert g^{(n)}_{i}(\mz,\mz')-g^{(n)}_{i}(\mz'',\mz''')\rvert<\epsilon$ for all $n\in\Z_+$;
\item\label{asm:nnd2} each $g^{(n)}_{i}$ is non-negative definite, i.e.,
\[\sum_{j_1,j_2=1}^{\ell}c_{j_1}c_{j_2}g^{(n)}_{i}(\mz_{j_1},\mz_{j_2})\ge 0\]
for all $c_1,\dots,c_{\ell}\in\R$, $\mz_1,\dots,\mz_{\ell}\in\Omega_i$, ${\ell}\in\Z_+$; 
\item\label{asm:dgn2}  each $g^{(n)}_{i}$ has $\E(g^{(n)}_{i}(\mz,\mZ^{(n)}_{i}))=0$; 
\item\label{asm:fnt2} each $g^{(n)}_{i}$ has $\E(g^{(n)}_{i}(\mZ^{(n)}_{i},{\mZ^{(n)}_{i}}')^2)\in(0,+\infty)$; %, recalling that ${\mZ^{(n)}_{i}}'$ is an independent copy of $\mZ_i$; 
\item\label{asm:unf} as $n\to\infty$, the functions $g^{(n)}_{i}$
  converge uniformly on $\Omega_i$ to $g_i$, with
  $\E(g_i(\mZ_{i},\mZ'_{i})^2)\in(0,+\infty)$.
  %, where $\mZ'_{i}$ is an independent copy of $\mZ_{i}$.
\end{enumerate}
 It then holds that
 \[n\widehat D^{(n)}\stackrel{\sf d}{\longrightarrow}
   \sum_{k_1,k_2=1}^{\infty}\lambda_{1,k_1}\lambda_{2,k_2}(\xi_{k_1,k_2}^2-1)\]
 as $n\to\infty$, where $\xi_{k_1,k_2},~k_1,k_2\in\Z_+$, are
 i.i.d.~standard Gaussian, and the $\lambda_{i,k}\geq 0,~k\in\Z_+$,
 are eigenvalues of the Hilbert-Schmidt integral operator given by
 $g_i$, i.e., for each $i$ the $\lambda_{i,k}$'s solve the integral
 equations
\[
\E(g_i(\mz_{i},\mZ_{i})e_{i,k}(\mZ_{i}))=\lambda_{i,k}e_{i,k}(\mz_{i})
\]
for a system of orthonormal eigenfunctions $e_{i,k}$.
\end{theorem}

Theorem \ref{thm:general} provides the essential component of our
analysis for $\widehat M_n$. However, $\widehat M_n$ is a permutation
statistic that is not double- but quadruple-indexed.  To cover this
case, we have to extend Theorem \ref{thm:general} to multiple-indexed
permutation statistics, the study of which is much more sparse (see,
for example, \citet{raic15multivariate} for some recent progresses).
Further notation is needed.

For all $j\in\Z_+$, let $\mw_j=(\mz_{1;j},\mz_{2;j})$ be a
vector with $\mz_{i;j}\in\Omega_i$, for $i=1,2$. 
Let $h: (\Omega_1\times \Omega_2)^m\to\R$
be a \emph{symmetric} kernel of order $m$, i.e., 
$h(\mw_1,\ldots,\mw_m)=h(\mw_{\sigma(1)},\ldots,\mw_{\sigma(m)})$
for all permutations $\sigma\in\rP(\zahl{m})$ and $\mw_1,\ldots,\mw_m\in\Omega_1\times \Omega_2$.
For any integer $\ell\in\zahl{m}$, and any measure $\Pr_{\mW}$, we let
\[
h_{\ell}(\mw_1\ldots,\mw_{\ell}; \Pr_{\mW}):=\E(h(\mw_1\ldots,\mw_{\ell},\mW_{\ell+1},\ldots,\mW_m)),
\]
where $\mW_1,\ldots,\mW_m$ are $m$ independent random vectors with
distribution $\Pr_{\mW}$.

The next theorem treats a multiple-indexed permutation-statistic of
order $m$ defined as
\begin{equation}\label{eq:mips}
\widehat \Pi^{(n)}=\mbinom{n}{m}^{-1}\sum_{1\le j_1 <\cdots< j_m\le n}
  h\Big(
  (\mz^{(n)}_{1;j_1},\mz^{(n)}_{2;\pi_{j_1}}),\ldots,
  (\mz^{(n)}_{1;j_m},\mz^{(n)}_{2;\pi_{j_m}})\Big),
\end{equation}
where $\mpi$ is uniformly distributed on $\rP(\zahl{n})$, and the
triangular arrays $\{\mz^{(n)}_{i;j},n\in\Z_+,j\in\zahl{n}\},~i=1,2$
are as introduced before the statement of Theorem~\ref{thm:general}.

\begin{theorem}\label{cor:general} 
%Recall the definitions of $\mZ^{(n)}_i,{\mZ^{(n)}_{i}}',\mZ_i,\mZ'_i$. 
Let $\mZ_i$ and $\mZ^{(n)}_i$, $i=1,2$, be defined as for Theorem~\ref{thm:general}. 
%Let $\mW$ denote random vector $(\mZ_{1},\mZ_{2})$ with $\mZ_{i}\in\Omega_i$, $i=1,2$.
%and $\mZ_1\indep \mZ_2$.
%where $\Omega_i$ is a compact subset of $\R^{p_i}$.
%Let the kenrel $h$ be \emph{symmetric}, i.e., 
%$h(\mw_1,\ldots,\mw_m)=h(\mw_{\sigma(1)},\ldots,\mw_{\sigma(m)})$, 
%for all permutations $\sigma\in\rP(\zahl{m})$ and $\mw_1,\ldots,\mw_m\in\Omega_1\times \Omega_2$.
Assume the kernel $h$  has the following three properties:
\begin{enumerate}[label=(\Roman*),itemsep=-.5ex]
\item\label{property:1} $h$ is continuous with $\lVert h\rVert_\infty <\infty$;
\item\label{property:2} $h_1\Big(\mw_1;\Pr_{\mZ^{(n)}_1}\times \Pr_{\mZ^{(n)}_2}\Big)=0$;
\item\label{property:3} one has
\begin{align*}
\mbinom{m}{2}\cdot h_2\Big(\mw_1,\mw_2;\Pr_{\mZ^{(n)}_1}\times \Pr_{\mZ^{(n)}_2}\Big)
&=g^{(n)}_1(\mz_{1;1},\mz_{1;2})g^{(n)}_2(\mz_{2;1},\mz_{2;2}),\\
\text{and}~~~
\mbinom{m}{2}\cdot h_2\Big(\mw_1,\mw_2;\Pr_{\mZ_1}\times \Pr_{\mZ_2}\Big)
&=g_1(\mz_{1;1},\mz_{1;2})g_2(\mz_{2;1},\mz_{2;2}),
\end{align*}
%where for each $i=1,2$, $g^{(n)}_{i},~n\in\Z_+$ are symmetric, equicontinuous, non-negative definite, $\E(g^{(n)}_i(\mz,\mZ^{(n)}_{i}))=0$ with $\E(g^{(n)}_i(\mZ^{(n)}_{i},{\mZ^{(n)}_{i}}')^2)\in(0,+\infty)$, and $g^{(n)}_{i}$ converges uniformly to $g_{i}$ which is continuous and satisfies
where for each $i=1,2$, $g^{(n)}_{i},~n\in\Z_+$, and $g_{i}$ 
satisfy Assumptions~\ref{asm:sym2}--\ref{asm:unf} from Theorem~\ref{thm:general}.
\end{enumerate}
We then have 
\[n\widehat \Pi^{(n)}\stackrel{\sf d}{\longrightarrow}
\sum_{k_1,k_2=1}^{\infty}\lambda_{1,k_1}\lambda_{2,k_2}(\xi_{k_1,k_2}^2-1)\]
as $n\to\infty$, where $\lambda_{i,k}$ and $\xi_{k_1,k_2}$ are defined as in Theorem~\ref{thm:general}.
\end{theorem}

With the aid of Theorem \ref{cor:general}, we are now ready to {prove Theorem \ref{thm:null}, which presents} the limiting null distribution of $\widehat M_n$.  In our context,
$p_1=p$, $p_2=q$, $m=4$, and $h$ is the kernel $K$ defined in
\eqref{eq:kernel}.  The multisets $\{\mz^{(n)}_{1;j},j\in\zahl{n}\}$
and $\{\mz^{(n)}_{2;j},j\in\zahl{n}\}$ are taken to be
$\{\bmu^{(n)}_{j},j\in\zahl{n}\}:=\cG^{p}_{n_0,n_R,\nb{n_S}}$ and
$\{\mv^{(n)}_{j},j\in\zahl{n}\}:=\cG^{q}_{n_0,n_R,\nb{n_S}}$,
respectively.  Accordingly, $\mZ^{(n)}_{1}$ follows the uniform discrete 
distribution over $\cG^{p}_{n_0,n_R,\nb{n_S}}$, denoted by
$\mU^{(n)}$, and $\mZ^{(n)}_{2}$ has a uniform discrete distribution over
$\cG^{q}_{n_0,n_R,\nb{n_S}}$, denoted by $\mV^{(n)}$.  The functions
$g^{(n)}_1$, $g_1$, $g^{(n)}_2$, and $g_2$ %, as shown below,
can be chosen as $-d_{\mU^{(n)}}$, $-d_{\mU}$, $-d_{\mV^{(n)}}$, and
$-d_{\mV}$, defined in the manner of \eqref{eq:fproj}, respectively.

We now verify properties~\ref{property:1}--\ref{property:3}. 
Write $\mw=(\bmu,\mv)$ and $\mw'=(\bmu',\mv')$.
Notice that the kernel $K$ is symmetric and continuous on $\overline\bS_p\times\overline\bS_q$.
We have
\begin{align*}
K_{1}\Big(\mw;\P_{\mU^{(n)}}\times \P_{\mV^{(n)}}\Big)
= 0,
~~~~~~
&6K_{2}\Big(\mw,\mw';\P_{\mU^{(n)}}\times \P_{\mV^{(n)}}\Big)
= \Big(-d_{\mU^{(n)}}(\bmu,\bmu')\Big)\Big(-d_{\mV^{(n)}}(\mv,\mv')\Big),\\
%K_{1}\Big(\mw;\P_{\mU}\times \P_{\mV}\Big)
%= 0
~~~\text{and}~~~~~~
&6K_{2}\Big(\mw,\mw';\P_{\mU}\times \P_{\mV}\Big)
= \Big(-d_{\mU}(\bmu,\bmu')\Big)\Big(-d_{\mV}(\mv,\mv')\Big),
\end{align*}
by \citet[Sec.~1.1]{MR3798874supp}.
Moreover, the $-d_{\mU^{(n)}}(\bmu,\bmu')$ is symmetric, non-negative
definite \citep[p.~3291]{MR3127883}, and equicontinuous since 
%\begin{align*}
%\lvert d_{\mU^{(n)}_{i}}(\bmu_{i},\bmu'_{i})-d_{\mU^{(n)}_{i}}(\bmu_{i},\bmu''_{i})\rvert
%&\;=\Big\lvert \lVert\bmu_i-\bmu'_i\rVert-\lVert\bmu_i-\bmu''_i\rVert - 
%\E\Big[\lVert\mU^{(n)}_i-\bmu'_i\rVert-\lVert\mU^{(n)}_i-\bmu''_i\rVert\Big]\Big\rvert\\
%&\;\le 2\lVert\bmu'_i-\bmu''_i\rVert,
%\end{align*}
%and thus
\[\lvert -d_{\mU^{(n)}}(\bmu,\bmu')-(-d_{\mU^{(n)}}(\bmu''',\bmu''))\rvert\le 2\lVert\bmu-\bmu'''\rVert+2\lVert\bmu'-\bmu''\rVert.\]
%It remains to prove that $-d_{\mU^{(n)}_{i}}(\bmu_{i},\bmu'_{i})$ converges uniformly to $-d_{\mU_{i}}(\bmu_{i},\bmu'_{i})$. It suffices to show $\E \lVert \bmu_i -\mU^{(n)}_{i}\rVert$ converges uniformly to $\E \lVert \bmu_i -\mU_{i}\rVert$, which follows from  the facts that $\E \lVert \bmu_i -\mU^{(n)}_{i}\rVert$ converges to $\E \lVert \bmu_i -\mU_{i}\rVert$ pointwisely (by the Portmanteau Lemma) and that
%\[\Big\lvert\,\Big\lvert\E \lVert \bmu_i -\mU^{(n)}_{i}\rVert-\E \lVert \bmu_i -\mU_{i}\rVert\Big\rvert - \Big\lvert\E \lVert \bmu'_i -\mU^{(n)}_{i}\rVert-\E \lVert \bmu'_i -\mU_{i}\rVert\Big\rvert\,\Big\rvert\le 2\lVert\bmu_i-\bmu'_i\rVert.\] 
One can verify that $\E[-d_{\mU^{(n)}}(\bmu,\mU^{(n)})]=0$, and 
$-d_{\mU^{(n)}}(\bmu,\bmu')$ converges uniformly to $-d_{\mU}(\bmu,\bmu')$ by combining the pointwise convergence using the Portmanteau Lemma \citep[Lemma~2.2]{MR1652247} and the equicontinuity of $-d_{\mU^{(n)}}(\bmu,\bmu')$ \citep[Exercise~7.16]{MR0385023}.
The similar results hold for $-d_{\mV^{(n)}}(\mv,\mv')$ and $-d_{\mV}(\mv,\mv')$. 
Lastly, under $H_0$, $[\fF^{(n)}_{\mX,\pm}(\mX_i)]_{i=1}^{n}$ and $[\fF^{(n)}_{\mY,\pm}(\mY_i)]_{i=1}^{n}$ are independent with margins uniformly distributed on $\rP(\cG^{p}_{n_0,n_R,\nb{n_S}})$ and $\rP(\cG^{q}_{n_0,n_R,\nb{n_S}})$, respectively. Hence our statistic is distributed of the form \eqref{eq:mips}.

{In summary, Theorem \ref{cor:general} can be applied to the statistic
$\widehat M_n$ and we have accordingly proven Theorem \ref{thm:null} rigorously. }
%the following corollary summarizes the result from
%the above derivations.
Furthermore, although our focus is on the combination of 
\nb{center-outward ranks and signs} 
with the distance covariance statistic, the general form of our
combinatorial non-CLTs (Theorems \ref{thm:general} and
\ref{cor:general}) also yields the limiting null distributions for
test statistics based on plugging \nb{center-outward ranks and signs} into HSIC-type or
ball-covariance statistics
\citep{gretton:aistats2005,MR2255909,MR2249882,MR4078465}.
% can be
% similarly derived based on Theorems \ref{thm:general} and
% \ref{cor:general}.
We omit the details for these analogies.

\section{Computational aspects}\label{sec:computation}

In this section, we describe the practical implementation of our
test. To perform the proposed test, for any given $n$, we fix a
factorization such that
\[
\nb{
n = n_R n_S+n_0,
~~~~~~n_R,n_S\in\Z_+,~0\le n_0<\min\{n_R,n_S\},
~~~\text{with}~n_R,n_S\to\infty~~~\text{as}~n\to\infty.
}
\]
%where $n_R,n_S$ are chosen to be as close to each other as possible.

First, we need to compute $[\fF_{\mX,\pm}^{(n)}(\mX_i)]_{i=1}^{n}$ and
$[\fF_{\mY,\pm}^{(n)}(\mY_i)]_{i=1}^{n}$ as defined in
\eqref{eq:assignment}. This is an assignment problem and will be discussed in Section
\ref{subsec:assign}. After obtaining
$[\fF_{\mX,\pm}^{(n)}(\mX_i)]_{i=1}^{n}$ and
$[\fF_{\mY,\pm}^{(n)}(\mY_i)]_{i=1}^{n}$, the test statistic
$\widehat{M}_n$ in \eqref{eq:feasible} can be computed using Equation
(3.3) in \citet{MR3556612} in $O(n^2)$ time.  Second, we have to
calculate the critical value $Q_{1-\alpha}$ defined in
\eqref{eq:test}. This value can be estimated numerically, as detailed in Section \ref{subsec:eigv}. {We have also provided the critical values at significance levels $\alpha=0.1,0.05,0.01$ for $(p,q)=(1,1),(1,2),\ldots,(10,10)$ with accuracy $5\cdot 10^{-3}$ in Table~\nb{C.1} in the supplement. }

{As shall be shown soon, the total computation complexity of our
  proposed test is $O(n^{5/2}\log(n))$ in  various cases. To contrast,
  to implement the distance covariance based test for instance, one
  has a time complexity $O(Rn^2)$, with $R$ representing the number of
  permutations. For many choices of $R$, our test will have a clear computational advantage. }

\subsection{Assignment problems}\label{subsec:assign}
Problem \eqref{eq:assignment} amounts to a linear sum assignment
problem (LSAP), a fundamental problem in linear programming and
combinatorial optimization.
%We first describe LSAP in terms of matrix formulation informally: we are given a nonnegative $n\times n$ cost matrix $C=[c_{ij}]$, where the element in the $i$-th row and $j$-th column represents the cost of assigning the $i$-th row to the $j$-th column. The problem is to assign each row to a distinct column such that the total cost is minimized. An example of $n=4$ is given in Figure~\ref{fig:matrix}, where the first, second, third, fourth rows are assigned to the third,  first, fourth, second columns, respectively. 
%Alternatively, 
We define LSAP through graph theory. Consider a weighted (complete)
bipartite graph $(S,T;E)$ with $S:=\{\ms_i\}_{i=1}^{n}$,
$T:=\{\mt_j\}_{j=1}^n$, $\ms_i,\mt_j\in\R^d$, where in Problem
\eqref{eq:assignment}, $S=\{\mx_i\}_{i=1}^{n}$ and
$T=\cG^{d}_{n_0,n_R,\nb{n_S}}$.  The edge between $\ms_i$ and $\mt_j$, denoted by $(\ms_i,\mt_j)$, has a nonnegative weight $c_{ij}:=\lVert\ms_i-\mt_j\rVert^2$,~ $i,j\in\zahl{n}$. We want to find an \emph{optimal matching}, i.e., a subset of edges such that each vertex is an endpoint of exactly one edge in this subset with a minimum sum of weights of its edges; see Figure~\ref{fig:bipartite} for an illustration of $n=3$, where edges in the optimal matching are marked in red. 

\begin{figure}[t]
\centering
%\begin{subfigure}[t]{0.45\textwidth}
%\includegraphics{./fig/assignmentmatrix}
%\caption{Matrix formulation of LSAP}\label{fig:matrix}
%\end{subfigure}
%~~~~~~
%\begin{subfigure}[t]{0.45\textwidth}
\includegraphics[scale=0.7]{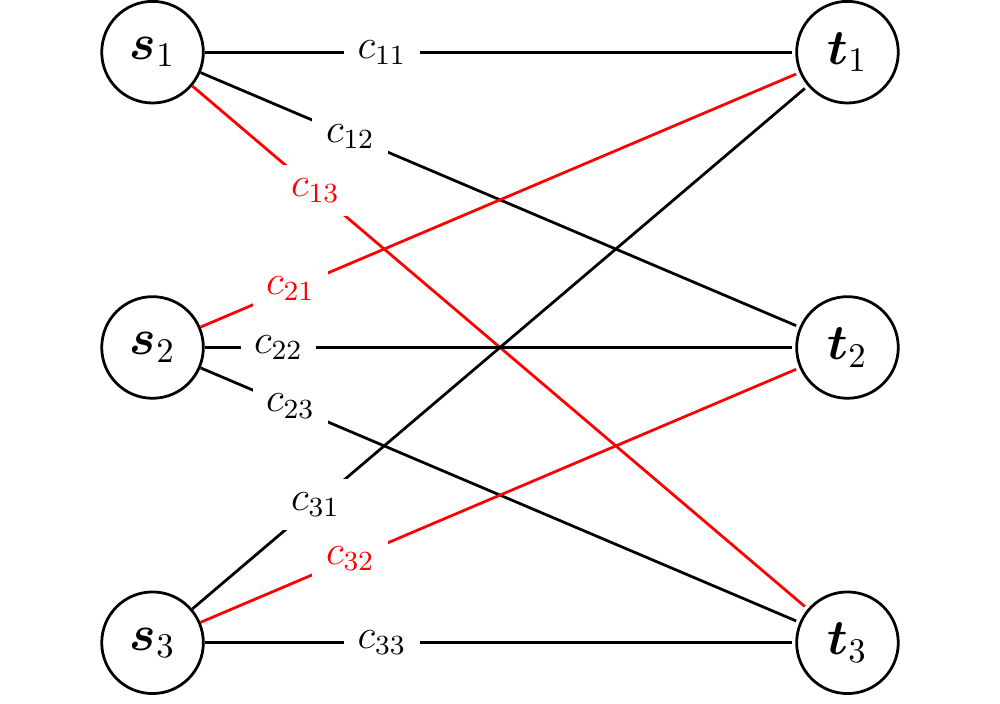}
\caption{Bipartite graph formulation of a linear sum assignment problem (LSAP)}\label{fig:bipartite}
%\end{subfigure}
%\caption{Two ways to formulate the linear sum assignment problem (LSAP).}
\end{figure}

%Introducing the dummy variables $x_{ij}$ defined as
%\[
%x_{ij}=\begin{cases} 1, & \text{if edge $(\ms_i,\mt_j)$ is in the optimal matching,}\\
%0, & \text{otherwise,}
%\end{cases}
%\]
%LSAP can be formulated as a linear program:
%\begin{align*}
%&\min_{x_{ij}}     &&\mkern-20mu\sum_{i,j}c_{ij}x_{ij}\\
%&\text{subject to} &&\mkern-20mu\sum_{j=1}^{n} x_{ij}=1,~\text{for}~i\in\zahl{n};
%~\sum_{i=1}^{n} x_{ij}=1,~\text{for}~j\in\zahl{n};
%~x_{ij}\in\{0,1\},~\text{for}~i,j\in\zahl{n}.
%\end{align*}
%Then an edge $(\ms_i,\mt_j)$ is in the optimal matching if and only if $x_{ij}=1$. The dual linear program is
%\begin{align*}
%&\max_{\alpha_i,\beta_j} &&\mkern-120mu\sum_{i}\alpha_{i}+\sum_{j}\beta_{j}\\
%&\text{subject to}       &&\mkern-120mu\alpha_{i}+\beta_{j}\le c_{ij},~\text{for}~i, j\in\zahl{n};
%~\alpha_{i},\beta_{j}~\text{unconstrained}.
%\end{align*}
%The sufficient and necessary condition for an optimal solution is
%%\[x_{ij}(c_{ij}-\alpha_{i}-\beta_{j})=0,~\text{for}~i,j\in\zahl{n}.\]
%\begin{align*}
%&\alpha_{i}+\beta_{j}\le c_{ij},  & &\mkern-150mu\text{for}~i,j\in\zahl{n},  \\
%&\alpha_{i}+\beta_{j}= c_{ij},    & &\mkern-150mu\text{for}~x_{ij}=1.
%\end{align*}
We introduce some terms to state the theorem below. A \emph{perfect matching} is a subset of edges such that each vertex is incident to exactly one edge. 
The \emph{total weight} of a perfect matching is the sum of weights of the edges in this matching.
A perfect matching is called \emph{$(1+\epsilon)$-approximate} for
$\epsilon>0$ if its total weight is no larger than $(1+\epsilon)$ times the total weight of the optimal matching.

\begin{theorem}[\citet{MR1015271}, \citet{MR3205219}, \citet{MR3238983}]\label{thm:matching}
Assume that points $\ms_i,\mt_j\in \R^d, ~i,j\in\zahl{n}$, have bounded integer coordinates, and
% the squared Euclidean distances between all pairs $(\ms_i,\mt_j)$
that the squared distances $\lVert \ms_i-\mt_j\rVert^2,
~i,j\in\zahl{n}$ are all bounded  by some integer $N$. Then there exists an algorithm to find the optimal matching in $O(n^{5/2}\log(nN))$ time. Furthermore,
\begin{itemize}
\item[(i)] if $d=2$, there exists an exact algorithm for computing the optimal matching in $O(n^{3/2+\delta}\log(N))$ time for any arbitrarily small constant $\delta>0$;
\item[(ii)] if $d\ge3$, there is an algorithm to compute a $(1+\epsilon)$-approximate perfect matching in \\$O(\epsilon^{-1}n^{3/2}\tau(n,\epsilon)\log^{4}(n/\epsilon)\log(\max c_{ij}/\min c_{ij}))$ time, 
where $\tau(n,\epsilon)$  depending on $n,\epsilon$ is small.
\end{itemize}
\end{theorem}

{In the supplement we will  describe the algorithm developed by
  \citet{MR1015271} under the basic settings. It is essentially the
  combination of the Hungarian method
  \citep{MR0075510,MR0091857,MR0093429} and the algorithm of \citet{MR0337699}.
We will ignore the details of the faster exact algorithm for $d=2$ by \citet{MR3205219} 
and the approximate algorithm for $d\ge3$ by \citet{MR3238983}; both
algorithms improve the Gabow--Tarjan algorithm by exploiting the geometric structure of the weight matrix. }

\subsection{Eigenvalues and quadratic forms in normal variables}\label{subsec:eigv}

In {Theorem \ref{thm:null}}, $\lambda_k,~k\in\Z_+$, are non-zero eigenvalues (counted with multiplicity) of the integral equation
\[
\E(d_{\mU}(\bmu,\mU)d_{\mV}(\mv,\mV) \phi(\mU,\mV))=\lambda \phi(\bm{u},\mv).
\]
Under the independence hypothesis $H_0$, the eigenvalues 
 $\lambda_k,~k\in\Z_+$, are given by all the products $\lambda_{1,j_1}\lambda_{2,j_2},~j_1,j_2\in\Z_+$, where $\lambda_{1,j},~j\in\Z_+$, and $\lambda_{2,j},~j\in\Z_+$, are the  non-zero eigenvalues of the integral equations
\[
\E(d_{\mU}(\bmu,\mU)\phi_1(\mU))=\lambda_{1}\phi_1(\bmu)
~~~\text{and}~~~
\E(d_{\mV}(\mv,\mV)\phi_2(\mV))=\lambda_{2}\phi_2(\mv),
\]
respectively \citep[Lemma 4.2]{MR3541972}. The  non-zero eigenvalues of integral equation $\E(d_{\mU}(\bmu,\mU)\phi_1(\mU))=\lambda_{1}\phi_1(\bmu)$ with $\mU\sim U_p$ are given by 
\[
-4/(\pi^2j^2),~~~{\rm for~all}~~j\in\Z_+~~ {\rm when}~~ p=1.
\] 
We are not aware of any closed form formulas for the eigenvalues when
$p\ge 2$.  However, in practice, the non-zero eigenvalues
$\{\lambda_{1,j}\}_{j=1}^{\infty}$ can be numerically estimated by the
non-zero eigenvalues of the matrix
\nb{
\[
(\fI_{M}-\fJ_{M}/M)\fD^{(M)}(\fI_{M}-\fJ_{M}/M)/M, 
\]
denoted by $\lambda_{1,j}^{(M)},~j\in\zahl{M-1}$, where $M:=M_R M_S$, 
$\fD^{(M)}=[D_{jj'}^{(M)}]$, $D_{jj'}^{(M)}=\lVert\bmu_j^{(M)}-\bmu_{j'}^{(M)}\rVert$ and 
$\bmu_j^{(M)},~j\in\zahl{M}$, are points in the grid $\cG^{p}_{0,M_R,M_S}$.
%for each $j\in\zahl{N}$, $\bmu_j^{(N)}$ is drawn independently from distribution $U_p$.
%e.g., exactly the $n$ points in the augmented grid $\cG^{p}_{n_0,n_R,\bm{n_S}}$ defined in Lemma~\ref{lem:unif}. 
Here %both $\lambda_{1,j},~j\in\Z_+$ and 
$\lambda_{1,j}^{(M)},~j\in\zahl{M-1}$ are all negative \citep[p.~3291]{MR3127883}. %and both sequences are sorted into ascending order. Then for any fixed $j\in\Z_+$, we have $\lambda_{1,j}^{(N)}\to\lambda_{1,j}$ as $n\to\infty$. 
For $p=1$, we take $\lambda_{1,j}^{(M)}=-4/(\pi^2j^2)$.
%with the convergence rate
%\[|\lambda_{1,j}^{(N)}-\lambda_{1,j}|\le C\max_{k:\lambda_{1,k}=\lambda_{1,j}}\Big\lVert \lambda_{1,j}\phi_{1,k}(\bmu)-\E d_{\mU^{(N)}}(\bmu,\mU^{(N)})\phi_{1,k}(\mU^{(N)})\Big\rVert_{\infty}\]
%for some absolute constant $C>0$ and all sufficiently large $n$, where $\phi_{1,k}(\bmu)$ is the eigenfunction associated with eigenvalue $\lambda_{1,k}$, 
%$\mU^{(n)}$ represents the uniform distribution over the augmented grid $\{\bmu_j^{(n)}\}_{j=1}^{n}$,
%and $d_{\mU^{(n)}}(\bmu,\bmu')$ is defined similarly as \eqref{eq:fproj}. 
We can obtain eigenvalues $\lambda_{2,j}^{(M)},~j\in\zahl{M-1}$ based on the grid $\cG^{q}_{0,M_R,M_S}$ similarly. Then we sort the positive products
$\lambda^{(M)}_{1,j_1}\lambda^{(M)}_{2,j_2},~j_1,j_2\in\zahl{M-1}$
into a descendingly ordered sequence $[\lambda^{(M)}_{k}]_{k=1}^{(M-1)^2}$,} and have the following theorem.

%\begin{theorem}
%Let $[\lambda_k]_{k=1}^{\infty}$ and $(\lambda^{(N)}_{k})_{k=1}^{N^2}$ be those defined in Proposition~\ref{prop:jakobsen} and defined above, respectively,
%and $\{\xi_k\}_{k=1}^{\infty}$ be independent standard Gaussian random variables.
%Then it holds that
%\[\sum_{k=1}^{N^2}\lambda_k^{(N)}(\xi_k^2-1) \stackrel{\sf d}{\longrightarrow} \sum_{k=1}^{\infty}\lambda_k(\xi_k^2-1).\]
%\end{theorem}

\begin{theorem}\label{thm:eigvcomp}
  Let $[\lambda_k]_{k=1}^{\infty}$ and
  $[\lambda^{(M)}_{k}]_{k=1}^{(M-1)^2}$ be eigenvalues as defined
  in {Theorem \ref{thm:null}} and above, respectively.  Let
  $[\xi_k]_{k=1}^{\infty}$ be a sequence of independent standard
  Gaussian random variables.  Then it holds for any pre-specified
  significance level $\alpha\in(0,1)$ that
\[Q_{1-\alpha}^{(M)} \to Q_{1-\alpha}\]
as \nb{$M_R\to\infty$ and $M_S\to\infty$}, where $Q_{1-\alpha}^{(M)}$ and $Q_{1-\alpha}$ are the $(1-\alpha)$ quantiles of
\[\sum_{k=1}^{(M-1)^2}\lambda_k^{(M)}(\xi_k^2-1)~~~\text{and}~~~\sum_{k=1}^{\infty}\lambda_k(\xi_k^2-1),\]
respectively.
\end{theorem}

%Therefore, we can approximate the $(1-\alpha)$ quantile of quadratic form in normal random variables $\sum_{k=1}^{\infty}\lambda_k(\xi_k^2-1)$ in any accuracy.
%For a sufficiently large $N$,
%the tail probability of quadratic form $\sum_{k=1}^{N^2}\lambda_k^{(N)}(\xi_k^2-1)$
%can be numerically evaluated using 
%approaches summarized in \citet{MR2580921}, namely 
%\citeauthor{farebrother1984}'s (\citeyear{farebrother1984}) algorithm or \citeauthor{MR0137199}'s (\citeyear{MR0137199}) method. %Accordingly, we can numerically compute quantile $Q_{1-\alpha}$ for any significance level $\alpha$.

Consequently, we can approximate the $(1-\alpha)$ quantile of quadratic form  $\sum_{k=1}^{\infty}\lambda_k(\xi_k^2-1)$ 
by estimating that of quadratic form $\sum_{k=1}^{(M-1)^2}\lambda_k^{(M)}(\xi_k^2-1)$
for a sufficiently large $M$.
The latter is done  by solving the inverse of
the cumulative distribution function of quadratic form $\sum_{k=1}^{(M-1)^2}\lambda_k^{(M)}(\xi_k^2-1)$, 
which can be numerically evaluated using 
\citeauthor{farebrother1984}'s (\citeyear{farebrother1984}) algorithm or \citeauthor{MR0137199}'s (\citeyear{MR0137199}) method.

\section{Numerical studies}\label{sec:simulation}

This section compares the performances of our tests using 
(i) \nb{the} theoretical rejection threshold {$Q_{1-\alpha}$ defined in \eqref{eq:test} and computed using the approximation in Section~\ref{subsec:eigv},} and 
(ii) a Monte Carlo simulation-based rejection threshold to the existing tests of
independence \nb{that use} (iii) distance covariance with marginal ranks \citep{lin2017copula}, and (iv) distance covariance \citep{MR3053543}.

{The test
via distance covariance with marginal ranks proceeds as follows. 
Write $\mx_i = (x_{i,1},...,x_{i,p})$ for $i\in\zahl{n}$.
Let $r_{i,k}$ be the rank of $x_{i,k}$ among $x_{1,k},x_{2,k},\dots,x_{n,k}$ for each $k\in\zahl{p}$. 
The marginal rank (vector) of $\mx_i$ is defined as $(r_{i,1},...,r_{i,p})$. 
The marginal rank (vector) of $\my_i$ is defined similarly. 
Then we run the permutation-based distance covariance test on the marginal ranks instead of the original data. }

\subsection{Simulation results}

We first conduct Monte Carlo simulation experiments on the finite-sample
performance of the proposed test from Section~\ref{sec:tests}. We evaluate the empirical sizes and powers of the four competing tests
stated above for both Gaussian and non-Gaussian distributions.
The values reported below are based on $1,000$ simulations at the
nominal significance level of $0.05$, with sample size
$n\in\{216,432,864,1728\}$, dimensions $p=q\in\{2,3,5,{7}\}$, and correlation
$\rho\in\{0,0.005,0.01,\dots,0.15\}$. {More simulation studies on even higher dimensions of $p=q=10$ and 30 are presented in the supplement, Section~\nb{C}.} 
For tests (iii) and (iv), we resample $n$ times in the permutation procedure.
%All data sets are generated as an i.i.d.~sample from a distribution as
%specified below.

\begin{example}\label{eg:sim-power1}
{The data $(X_1,\ldots,X_n)$ are independently drawn from $(\mX,\mY)\in\R^{p+q}$, which follows a multivariate normal distribution with mean zero and covariance matrix 
$\fI_{p+q}+\tau\fL_{p+q;1,2}+\rho\fL_{p+q;1,p+1}$
(where
$\fL_{d;i,j}:=\me_{d;i}\me_{d;j}^{\top}+\me_{d;j}\me_{d;i}^{\top}$ and
$\me_{d;i}\in\R^{d}$ is the $i$-th standard basis vector in
$d$-dimensional space, i.e., all entries are zero except for the one at the $i$-th position)} with (a) $\tau=0$; (b) $\tau=0.5${; and (c) $\tau=0.9$.}
\end{example}

%\begin{example}\label{eg:sim-power2}
%The data $(\mX,\mY)$ are given by $X_i=Q_{t(3)}(\Phi(X_i^*))$, $i\in\zahl{p}$ and $Y_j=Q_{t(3)}(\Phi(Y_j^*))$, $j\in\zahl{q}$, where $Q_{t(3)}$ stands for the quantile function for Student's $t$-distribution with $3$ degrees of freedom, and $(\mX^*,\mY^*)$ are generated as in Example~\ref{eg:sim-power1}.
%\end{example}
%\begin{example}\label{eg:sim-power3}
%The data $(\mX,\mY)$ are given by $X_i=Q_{t(2)}(\Phi(X_i^*))$, $i\in\zahl{p}$ and $Y_j=Q_{t(2)}(\Phi(Y_j^*))$, $j\in\zahl{q}$, where $Q_{t(2)}$ stands for the quantile function for Student's $t$-distribution with $2$ degrees of freedom, and $(\mX^*,\mY^*)$ are generated as in Example~\ref{eg:sim-power1}.
%\end{example}

\begin{example}\label{eg:sim-power4}
The data $(X_1,\ldots,X_n)$ are independently drawn from $(\mX,\mY)$, which is given by $X_i=Q_{t(1)}(\Phi(X_i^*))$, $i\in\zahl{p}$ and $Y_j=Q_{t(1)}(\Phi(Y_j^*))$, $j\in\zahl{q}$, where $Q_{t(1)}$ stands for the quantile function for Student's $t$-distribution with $1$ degree of freedom (Cauchy distribution), and $(\mX^*,\mY^*)$ are generated as in Example~\ref{eg:sim-power1}.
\end{example}

{
\renewcommand{\tabcolsep}{4pt}
\renewcommand{\arraystretch}{1.10}
\begin{table}[t]
\centering
\caption{{Empirical sizes of the proposed test using theoretical (noted as Hallin(t)) and simulation-based (noted as Hallin(s)) rejection threshold, test via distance covariance with marginal ranks (noted as rdCov), and test via distance covariance (noted as dCov) in Example \ref{eg:sim-power1}(a).}}
\label{tab:size-1}{
{
\small
\begin{tabular}{ccC{.55in}C{.55in}C{.55in}C{.55in}}
$(p,q)$ & $n$ & Hallin(t) & Hallin(s) &  rdCov  &   dCov  \\
\hline
$(2,2)$ &  $216$  &  0.040  &  0.043  &  0.043  &  0.045  \\
$(2,2)$ &  $432$  &  0.037  &  0.047  &  0.048  &  0.050  \\
$(2,2)$ &  $864$  &  0.045  &  0.045  &  0.050  &  0.048  \\
$(2,2)$ & $1728$  &  0.054  &  0.054  &  0.061  &  0.057  \\
$(3,3)$ &  $216$  &  0.047  &  0.047  &  0.058  &  0.053  \\
$(3,3)$ &  $432$  &  0.047  &  0.053  &  0.045  &  0.043  \\
$(3,3)$ &  $864$  &  0.040  &  0.047  &  0.053  &  0.048  \\
$(3,3)$ & $1728$  &  0.049  &  0.049  &  0.043  &  0.050  \\
$(5,5)$ &  $216$  &  0.040  &  0.043  &  0.040  &  0.048  \\
$(5,5)$ &  $432$  &  0.033  &  0.043  &  0.048  &  0.043  \\
$(5,5)$ &  $864$  &  0.047  &  0.050  &  0.040  &  0.048  \\
$(5,5)$ & $1728$  &  0.059  &  0.059  &  0.053  &  0.039  \\
$(7,7)$ &  $216$  &  0.068  &  0.048  &  0.053  &  0.056  \\
$(7,7)$ &  $432$  &  0.064  &  0.050  &  0.054  &  0.053  \\
$(7,7)$ &  $864$  &  0.056  &  0.051  &  0.048  &  0.046  \\
$(7,7)$ & $1728$  &  0.052  &  0.054  &  0.048  &  0.052  \\
\end{tabular}}}
%Results are averaged over $5000$ simulated data sets.
\end{table}
}

In these two examples, the independence hypothesis holds when $\rho=0$.
{We first report the empirical sizes of all four considered tests, 
%the proposed tests using theoretical and simulation-based 
%rejection thresholds, which are 
presented in Table~\ref{tab:size-1}. It can be observed that the proposed tests with either rejection threshold as well as their two competitors control the size effectively. }

The empirical powers for Examples~\ref{eg:sim-power1}--\ref{eg:sim-power4} are
summarized in Figures~\ref{fig:power1a}--\ref{fig:power4c}.
For the proposed test, we present results only for the theoretical 
rejection threshold as the results for the simulation-based threshold are similar and hence omitted. % due to the observation above.
%The powers of tests involving HSIC are always lower than the counterparts
%involving distance covariance and hence omitted.

Several facts are noteworthy. {First, when the 
sample size is large and the dimension is relatively small, throughout
all settings the performance of the proposed test is not much worse
than the two competing ones. It should be highlighted that our method
achieves this performance with smaller computational time, as shown in Figure~\ref{fig:time} and also confirmed in our theoretical analysis of computational cost.}  {Second, the proposed test beats the other two when the within-group correlation is high, i.e., as $\tau$ becomes larger from the setting (a) to (c), even when the dimension is high. Third,} for heavy-tailed distributions, the tests via 
distance covariance with \nb{center-outward ranks and signs} and marginal ranks perform better than 
the original distance covariance test.
%{Third, the proposed test performs much better relative to the other two competing tests  as within-group correlation becomes higher, i.e., as $\tau$ becomes larger from setting (a) to (c), even when the dimension is high.}
Lastly, compared to its competitors, the proposed test appears to be more sensitive to dimension. This is as expected.
%{Lastly, one may wonder how these tests compare with the corresponding parametric tests, for example, the likelihood ratio test. 
%\citet{MR2382665} observed that, based on multiple simulations, 
%when the dependence structure is nonlinear, 
%distance covariance test is more powerful than the parametric likelihood ratio test, 
%while in the multivariate normal case, the power of distance covariance test
%is quite close to that of the likelihood ratio test. }

\begin{figure}[!htbp]
\centering
\includegraphics[width=\textwidth]{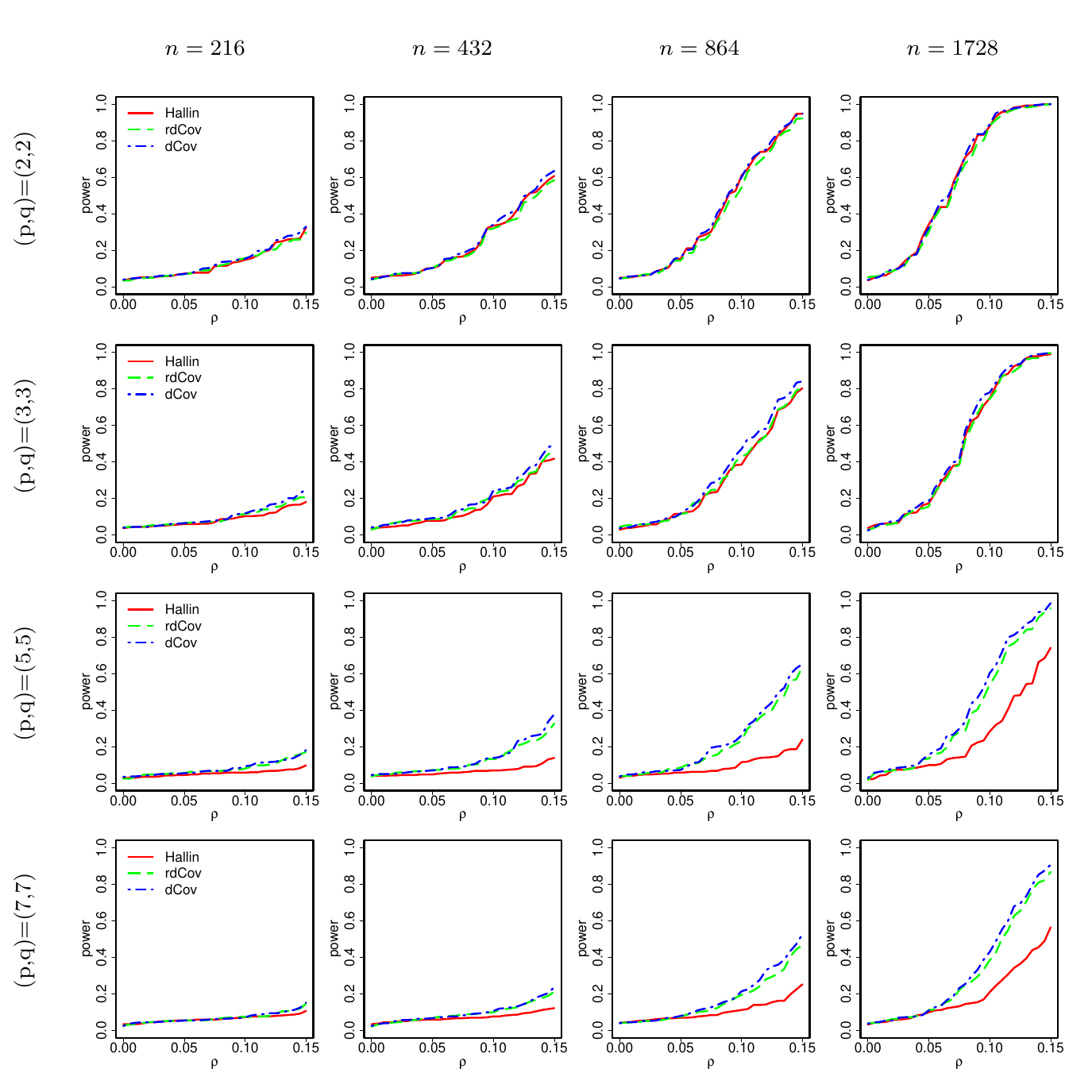}
\caption{Empirical powers of the three competing tests 
in Example~\ref{eg:sim-power1}(a). 
The $y$-axis represents the power based on 1,000 replicates and 
the $x$-axis represents the level of a desired signal.}\label{fig:power1a}
\end{figure}

\begin{figure}[!htbp]
\centering
\includegraphics[width=\textwidth]{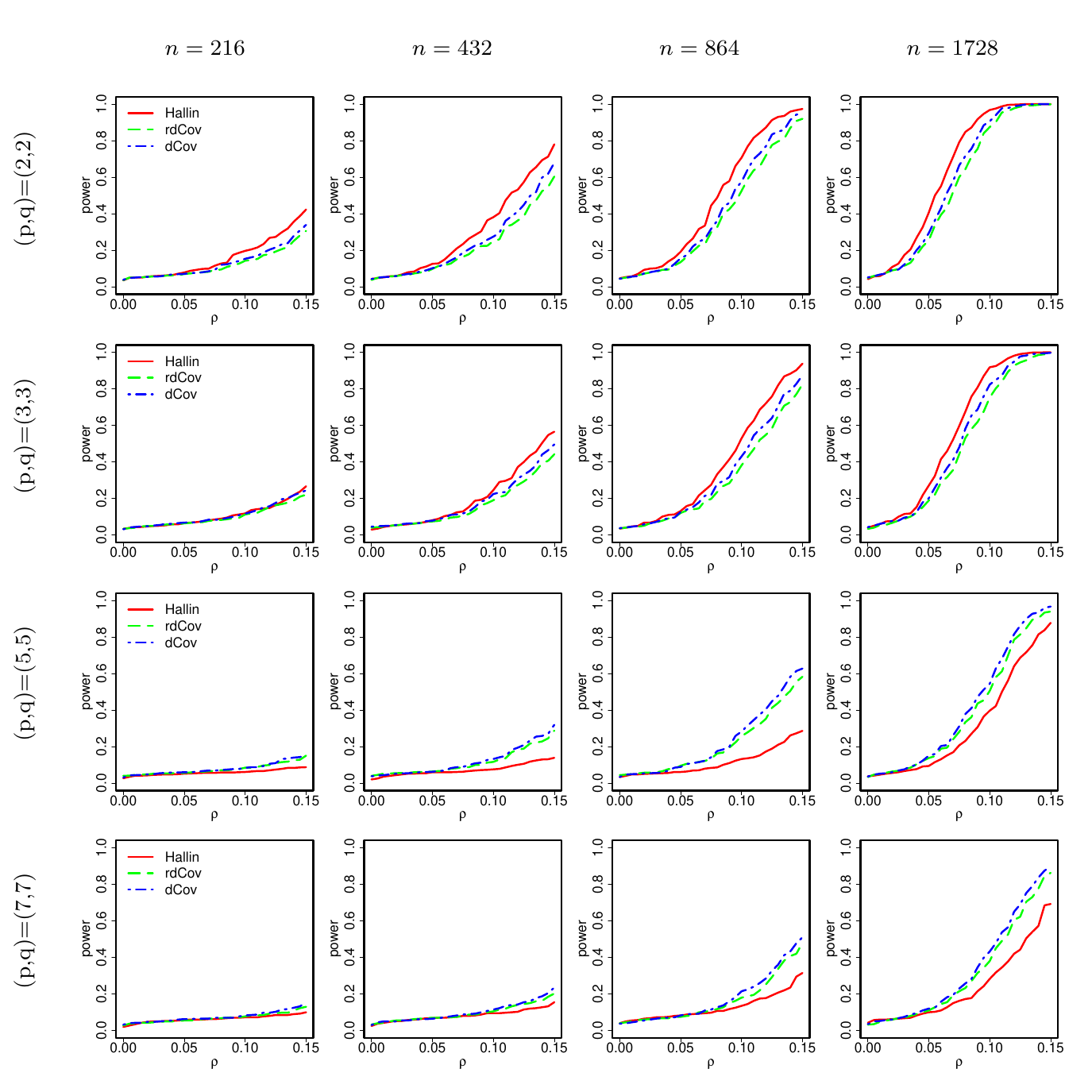}
\caption{Empirical powers of the three competing tests 
in Example~\ref{eg:sim-power1}(b). 
The $y$-axis represents the power based on 1,000 replicates and 
the $x$-axis represents the level of a desired signal.}\label{fig:power1b}
\end{figure}

\begin{figure}[!htbp]
\centering
\includegraphics[width=\textwidth]{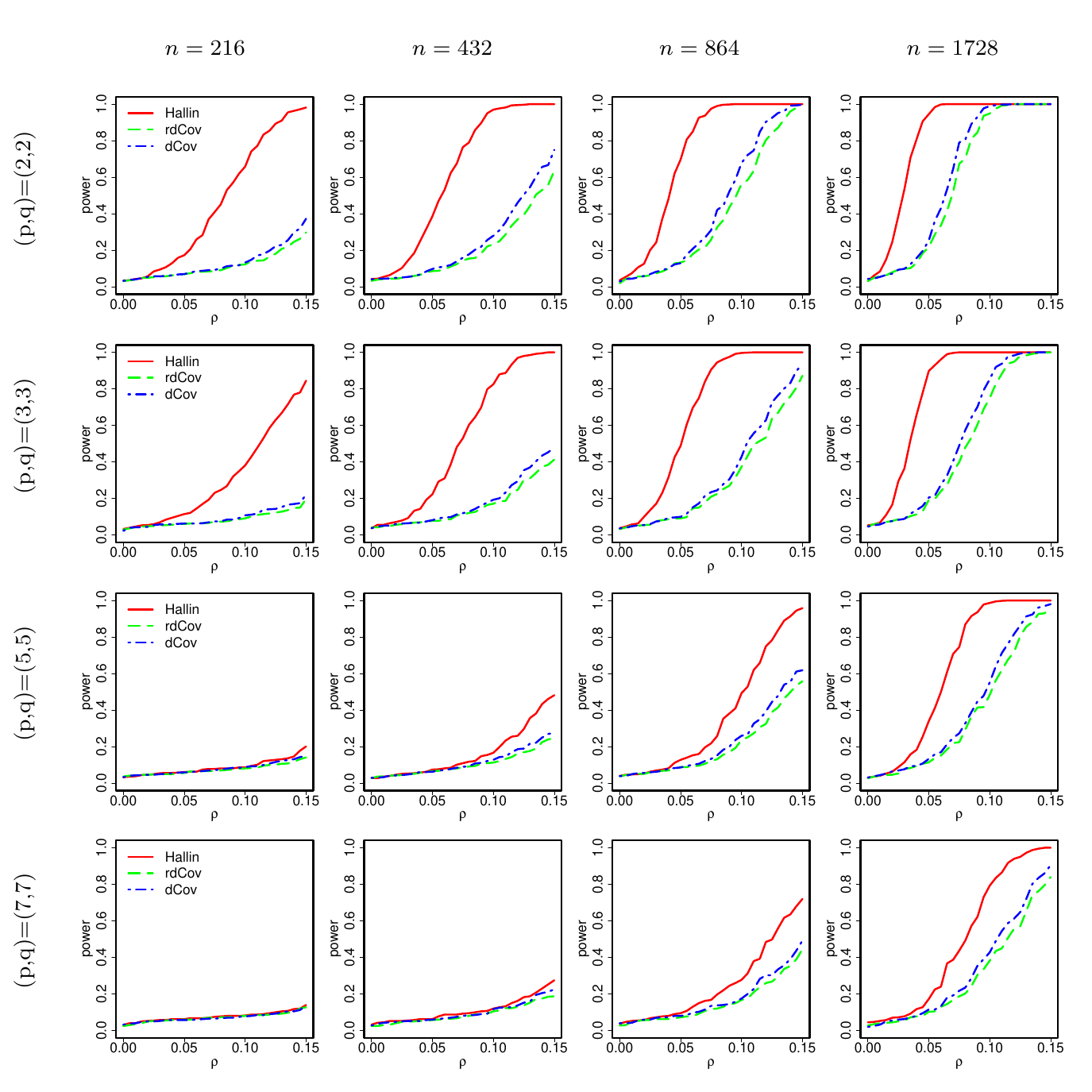}
\caption{Empirical powers of the three competing tests 
in Example~\ref{eg:sim-power1}(c). 
The $y$-axis represents the power based on 1,000 replicates and 
the $x$-axis represents the level of a desired signal.}\label{fig:power1c}
\end{figure}

\begin{figure}[!htbp]
\centering
\includegraphics[width=\textwidth]{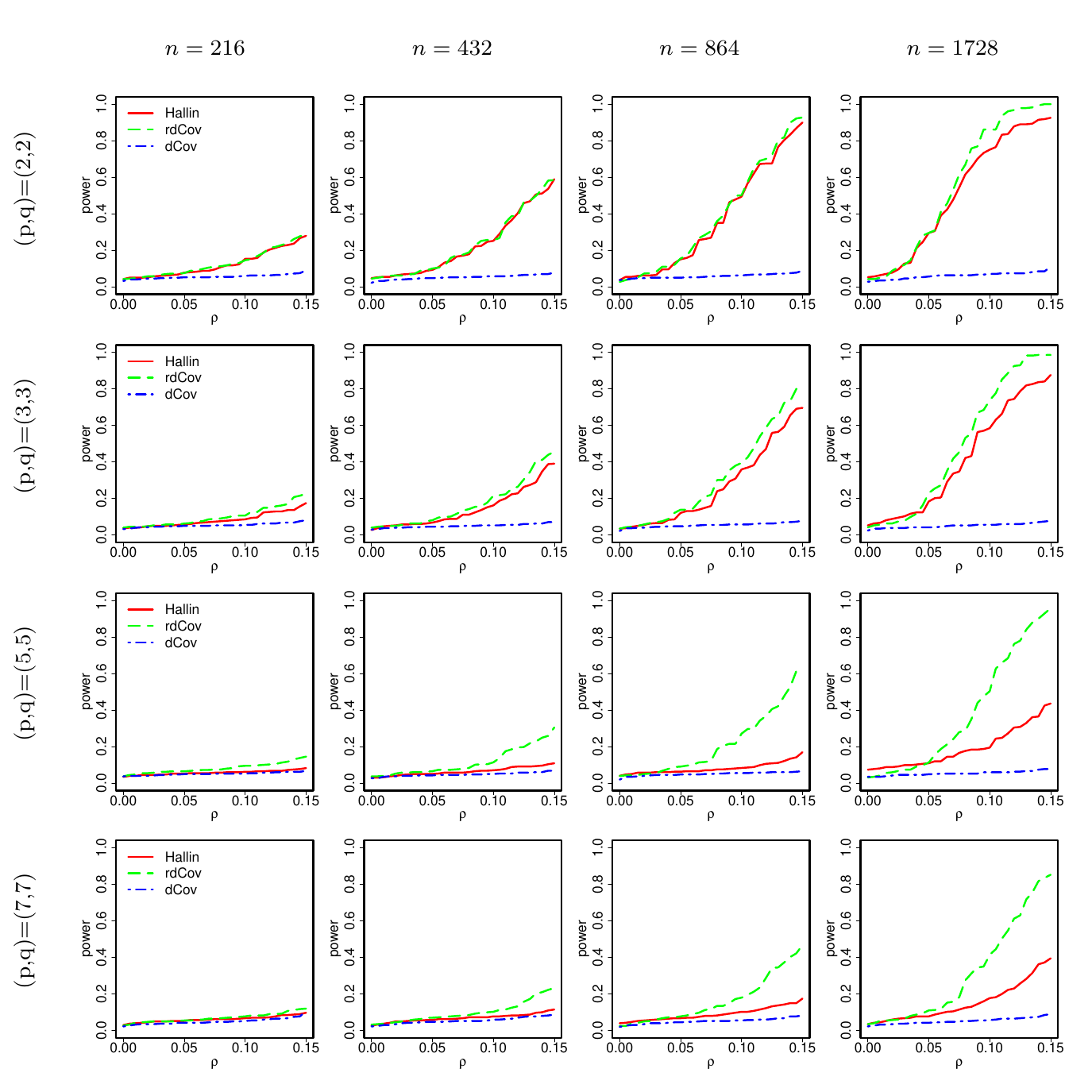}
\caption{Empirical powers of the three competing tests 
in Example~\ref{eg:sim-power4}(a). 
The $y$-axis represents the power based on 1,000 replicates and 
the $x$-axis represents the level of a desired signal.}\label{fig:power4a}
\end{figure}

\begin{figure}[!htbp]
\centering
\includegraphics[width=\textwidth]{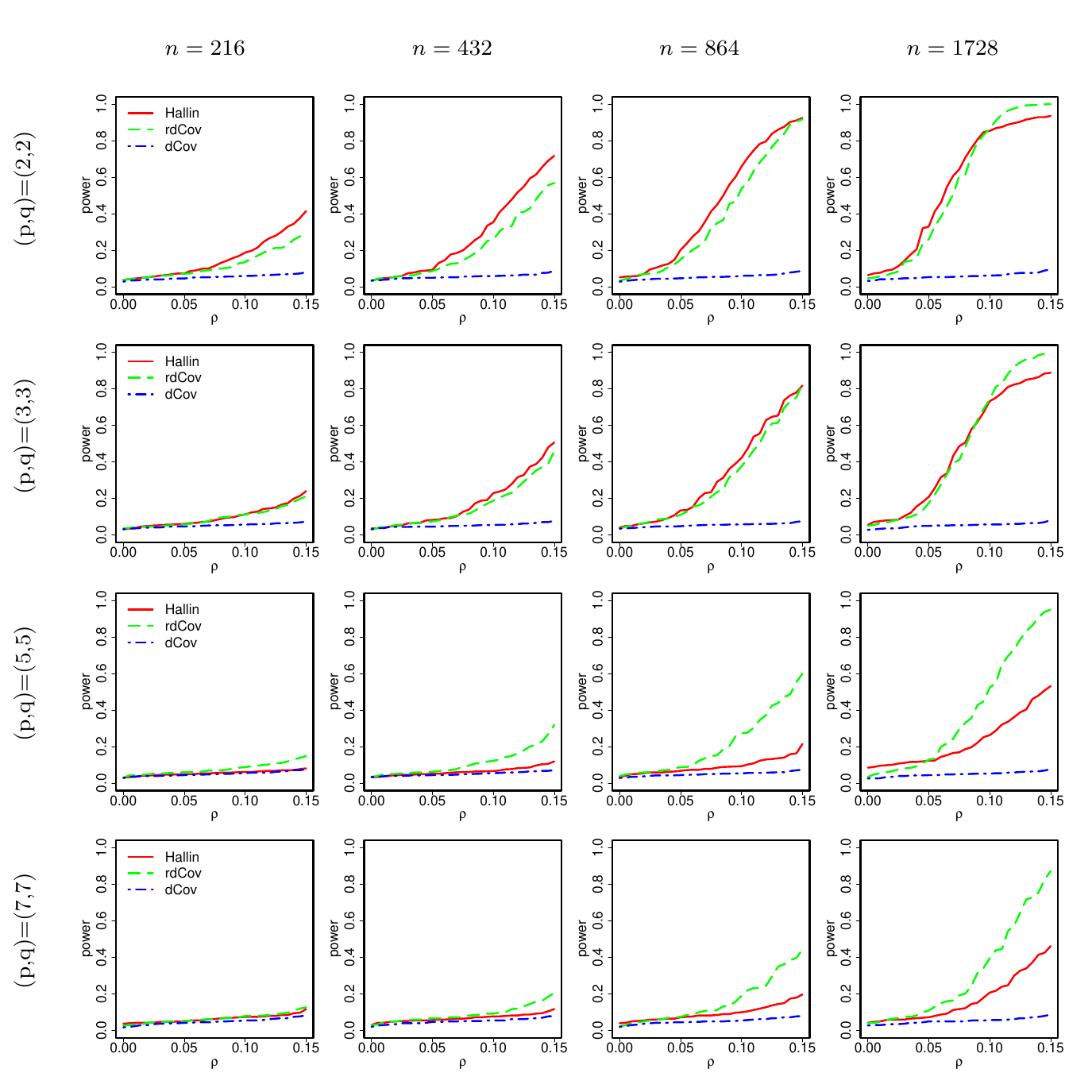}
\caption{Empirical powers of the three competing tests 
in Example~\ref{eg:sim-power4}(b). 
The $y$-axis represents the power based on 1,000 replicates and 
the $x$-axis represents the level of a desired signal.}\label{fig:power4b}
\end{figure}

\begin{figure}[!htbp]
\centering
\includegraphics[width=\textwidth]{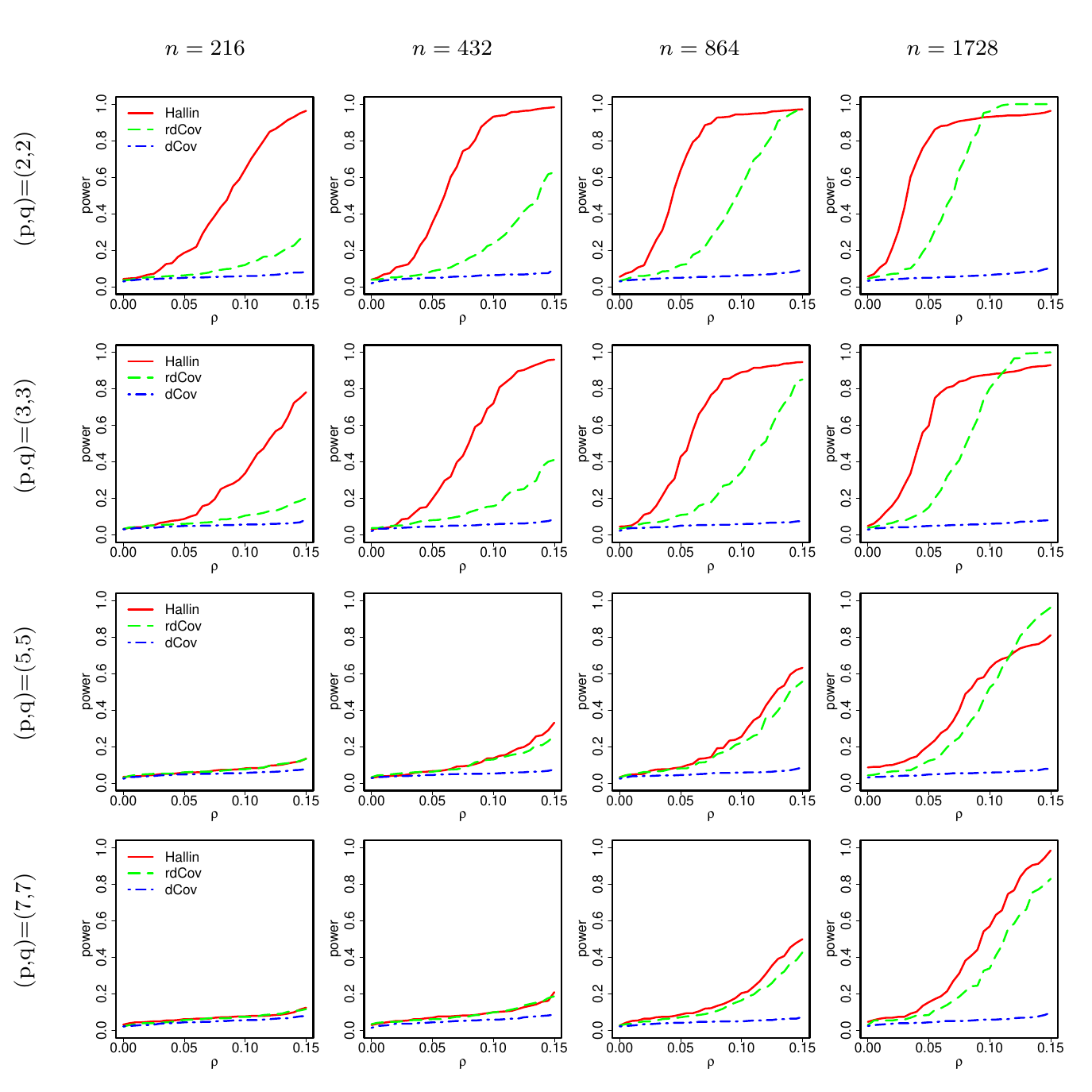}
\caption{Empirical powers of the three competing tests 
in Example~\ref{eg:sim-power4}(c). 
The $y$-axis represents the power based on 1,000 replicates and 
the $x$-axis represents the level of a desired signal.}\label{fig:power4c}
\end{figure}

\begin{figure}[!htbp]
\centering
\includegraphics[width=.75\textwidth]{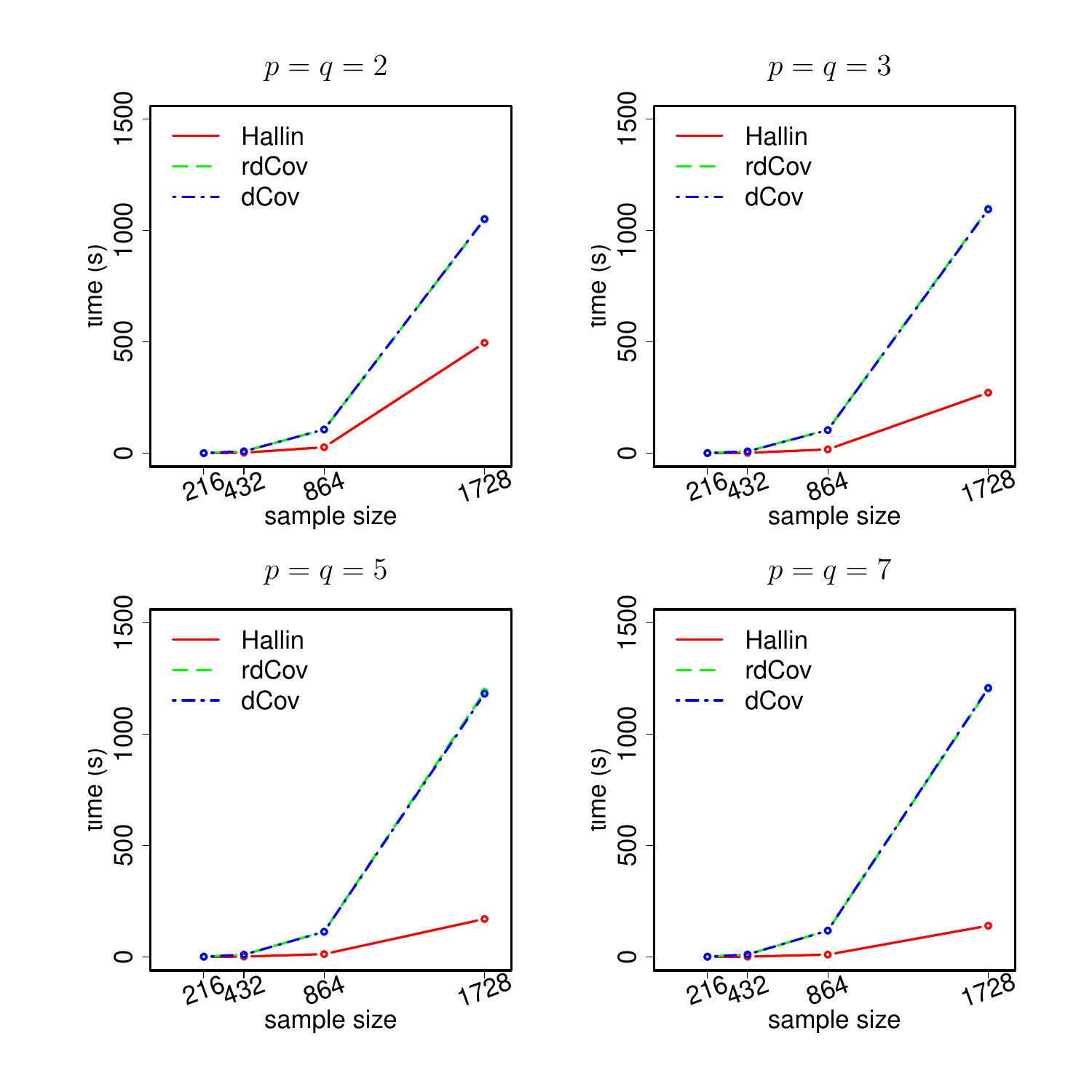}
\caption{A comparison of computation time 
in Example \ref{eg:sim-power1}(a)
for the three tests. 
The $y$-axis represents the averaged computation elapsed time 
(in seconds) of 1,000 replicates of a single experiment and 
the $x$-axis represents the sample size. To compute the optimal matching, we used the algorithm in \citet{MR1015271}.}\label{fig:time}
\end{figure}

{
\subsection{Real stock market data analysis} \label{sec:realdata}

We analyze the monthly log returns of daily closing prices for stocks
that are constantly in the Standard \& Poor 100 (S\&P 100) index
during the time period 2003 to 2012. The data are from Yahoo! Finance
({\sf finance.yahoo.com}), and the stocks are classified into 10
sectors by Global Industry Classification Standard (GICS). Stock
market data tend to be heavy-tailed with many outliers, and monthly
log returns may reasonably be modeled as independent and identically
distributed random variables.
% and empirically also observed to be (nearly) so.
The time period we analyzed includes some well known turbulent stretches like the 2007-08 financial crisis, which, however, could be either explained using heavy-tailed (e.g., elliptical or stable) distribution models or captured as outliers. 

In this section we limit our scope and focus on detecting
between-group dependence between two sectors in S\& P 100 that contain
a rather small number of stocks: (1) Telecommunication, including
stocks ``AT\&T Inc [T]'' and ``Verizon Communications [VZ]'';  and (2)
Materials, including stocks ``Du Pont (E.I.) [DD]'', ``Dow Chemical
[DOW]'', ``Freeport-McMoran Cp \& Gld [FCX]'', and ``Monsanto
Co. [MON]''.  We then consider detection of possible dependence between the Telecommunication sector and any two stocks in the Materials sector. 

To this end, we apply the three considered tests to the monthly log returns of (T,VZ) coupled with either (DD,DOW), or (DD,FCX), or (DD,MON), or (DOW,FCX), or (DOW,MON), or (FCX,MON).  The p-values for these three tests are reported in Table~\ref{tab:pvalue-real}. There, one observes that using the proposed test yields uniformly the strongest evidence to conclude the existence of dependence between (T,VZ) and any two stocks in the Materials sector. }

{
\renewcommand{\tabcolsep}{1pt}
\renewcommand{\arraystretch}{1.10}
\begin{table}[!htb]
\centering
\caption{{P-values based on the proposed test as well as two competing tests for the dataset of US stock closing prices between 2003 and 2012.}}
\label{tab:pvalue-real}{
{
\small
\begin{tabular}{C{.75in}C{.75in}C{.75in}C{.75in}C{.75in}C{.85in}C{.85in}C{.85in}}
 &  & (DD,DOW) & (DD,FCX) & (DD,MON) & (DOW,FCX) & (DOW,MON) & (FCX,MON) \\
\hline
  Hallin & (T, VZ) & 0.001 & 0.005 & 0.002 & 0.004 & 0.001 & 0.065 \\
  rdCov &(T, VZ) & 0.002 & 0.013 & 0.005 & 0.009 & 0.002 & 0.070 \\
  dCov & (T, VZ) & 0.002 & 0.018 & 0.003 & 0.012 & 0.002 & 0.101 \\
\end{tabular}}}
%Results are averaged over $5000$ simulated data sets.
\end{table}
}

\nb{
\section*{Acknowledgments}
We thank the co-editor Hongyu Zhao, the anonymous associate editor,
and two anonymous referees for their very detailed and constructive
comments and suggestions, which have helped greatly to improve the quality of the paper.
}

\appendix

\section{Proofs}

%\fbox{double check if all notation has been introduced}

%Let $\mX_1,\dots,\mX_n$ be i.i.d.~copies of $\mX$, with
%$\mX_i=(X_{i,1},\dots,X_{i,p})^\top$.  Let
%$j\ne k\in[p]$, and let $h: (\R^{2})^m\to\R$ be a kernel of order
%$m$.  

Further concepts concerning U-statistics are needed.
For  any symmetric kernel $h$, any integer $\ell\in\zahl{m}$, 
and any probability measure $\Pr_{\mX}$, we remind the definition of
\begin{align*}
h_{\ell}(\mx_1\ldots,\mx_{\ell}; \Pr_{\mX})&:=\E h(\mx_1\ldots,\mx_{\ell},\mX_{\ell+1},\ldots,\mX_m),
\end{align*}
and write
\begin{align*}
\widetilde h_{\ell}(\mx_1,\ldots,\mx_{\ell}; \Pr_{\mX}) 
&:=h_{\ell}(\mx_1,\ldots,\mx_{\ell};\Pr_{\mX})
  -\E h-\sum_{k=1}^{\ell-1}\sum_{1\leq i_1<\cdots<i_k\leq\ell}
  \widetilde h_{k}(\mx_{i_1},\ldots,\mx_{i_k};\Pr_{\mX}),\yestag\label{eq:han-hl}
\end{align*}
where $\mX_1,\ldots,\mX_m$ are $m$ independent random variables with
law $\Pr_{\mX}$ and $\E h:=\E h(\mX_1,\ldots,\mX_m)$. 
%The kernel as well as the corresponding U-statistic are said to be \emph{degenerate} under $\Pr_{\mX}$ if $h_1(\cdot)$ has variance zero.  
%We use the term \emph{completely degenerate} to indicate that the
%variances of $h_1(\cdot),\ldots,h_{m-1}(\cdot)$ are all zero.
%We let $H_n^{(\ell)}(\cdot;\Pr_{\mZ})$ be the U-statistic based on the
%completely degenerate kernel $h^{(\ell)}(\cdot;\Pr_{\mZ})$ from~(\ref{eq:han-hl}):
%\begin{equation}
%H_n^{(\ell)}(\cdot;\Pr_{\mZ}):=\mbinom{n}{\ell}^{-1}\sum_{1\le i_1< i_2<\cdots< i_\ell\le n} h^{(\ell)}\Big(\mZ_{i_1},\ldots,\mZ_{i_\ell};\Pr_{\mZ}\Big).
%\end{equation}
%We have the following Hoeffding decomposition
%\[\widehat U_n = \E h+\sum_{\ell=1}^{m} \mbinom{m}{\ell}H_n^{(\ell)}(\cdot;\Pr_{\mZ}).\]
%If we further define $\delta_{\ell}(\cdot;\Pr_{\mZ})=\binom{m}{\ell}\widetilde h_{\ell}(\cdot;\Pr_{\mZ})$, 
%and let $ D_n^{(\ell)}(\cdot;\Pr_{\mZ})$ be the corresponding U-statistic 
%based on the completely degenerate kernel $\delta^{(\ell)}(\cdot;\Pr_{\mZ})$:
%\begin{equation}
%D_{\ell}(\cdot;\Pr_{\mZ}):=\mbinom{n}{\ell}^{-1}\sum_{1\le i_1< i_2<\cdots< i_\ell\le n} \delta_{\ell}\Big(\mZ_{i_1},\ldots,\mZ_{i_\ell};\Pr_{\mZ}\Big),
%\end{equation}
We also have
\begin{align*}
   &\mbinom{n}{m}^{-1}\sum_{1\le i_1< i_2<\cdots< i_m\le n} h\Big(\mX'_{i_1},\ldots,\mX'_{i_m}\Big) \\
=\;&\E h + \sum_{\ell=1}^{m}\mbinom{m}{\ell} \mbinom{n}{\ell}^{-1}\sum_{1\le i_1< i_2<\cdots< i_\ell\le n} \widetilde h_{\ell}\Big(\mX'_{i_1},\ldots,\mX'_{i_\ell};\Pr_{\mX}\Big),\yestag
\end{align*}
for any (possibly dependent) random variables $\mX'_1,\dots, \mX'_n$. This is the Hoeffding decomposition with respect to $\Pr_{\mX}$.

\paragraph{Additional notation.}
Let $(n)_r$ denote $n!/(n-r)!$.
The cardinality of a set $\mathcal{S}$ is written $\card(\mathcal{S})$.
For a multiset $\cM=\{x_1,\dots,x_n\}$ and $r\in\zahl{n}$,
an $r$-permutation of $\cM$ is a sequence $[x_{\sigma(i)}]_{i=1}^{r}$,
where $\sigma$ is a bijection from $\zahl{n}$ to itself.
For $r\in\zahl{n}$, let $I^{n}_{r}$ denote the family of all $(n)_r$ possible $r$-permutations of set $\zahl{n}$.
For $x\in\R$, let
$x_{+}=\max\{x,0\}$ denote the positive part of $x$.
Let $\mx\circ\my$ and $\mx\cdot\my$ denote the Hadamard product and
dot product of two vectors $\mx,\my\in\R^d$, respectively.
We use $\stackrel{\sf p}{\longrightarrow}$ 
to denote convergence in probability.  We use $\sfi$ to represent the imaginary unit. 
%We first introduce more notation. For any vector $\mv\in\mathbb{R}^p$, we denote $\|\mv\|$ as its Euclidean norm.
%We define the $L^{\infty}$ norm of a random variable as $\lVert X\rVert_{\infty}=\inf\{t\ge 0:|X|\le t\text{ a.s.}\}$, the $\psi_2$ (sub-gaussian) norm as $\lVert X\rVert_{\psi_2}=\inf\{t>0:\E\exp(X^2/t^2)\le 2\}$, and the $\psi_1$ (sub-exponential) norm as $\lVert X\rVert_{\psi_1}=\inf\{t>0:\E\exp(|X|/t)\le 2\}$.
%For any measure $\Pr_{Z}$ and kernel $h$, we let
%$H_n^{(\ell)}(\cdot;\Pr_{Z})$ be the U-statistic based on the
%completely degenerate kernel $h^{(\ell)}(\cdot;\Pr_{Z})$ from~(\ref{eq:han-hl}):
%\begin{equation}
%H_n^{(\ell)}(\cdot;\Pr_{Z}):=\mbinom{n}{\ell}^{-1}\sum_{1\le i_1< i_2<\cdots< i_\ell\le n} h^{(\ell)}\Big(Z_{i_1},\ldots,Z_{i_\ell};\Pr_{Z}\Big).
%\end{equation}

\subsection{Proofs for Section \ref{sec:prelim} of the main paper}

\subsubsection{Proof of Proposition \ref{prop:unif}}

\begin{proof}[Proof of Proposition \ref{prop:unif}]
We first prove the case $n_0=0$ and then generalize to $n_0>0$. 
For simpler presentation, 
let $\lambda_{n_0,n_R,{n_S}}$ denote the uniform measure (distribution) on the augmented grid $\cG^{d}_{n_0,n_R,{n_S}}$, 
let $\mu_{n_R}$ denote the uniform measure on the points $\{\frac{j}{n_R+1}: j\in\zahl{n_R}\}$, 
and let $\nu_{{n_S}}$ denote the uniform measure on the points 
$\{\mr_{k}: k\in\zahl{n_S}\}$.  Furthermore,
let $\mu$ denote the uniform measure on $[0,1)$,
and let $\nu$ denote the uniform measure over the unit sphere $\cS_{d-1}$.

If $n_0=0$, then $\lambda_{n_0,n_R,{n_S}}$ is the product measure of 
$\mu_{n_R}$ (for the radius) and  $\nu_{\bm{n_S}}$ (for the unit sphere). 
By assumption, $\nu_{{n_S}}$ weakly converges to $\nu$ as $n_S\to\infty$. 
Moreover, $\mu_{n_R}$ weakly converges to $\mu$ as $n_R\to\infty$ by the following argument: 
\[\mu_{n_R}\Big((0,x]\Big)=\frac{\lfloor n_R x\rfloor}{n_R}\to x =\mu\Big((0,x]\Big),
~~~\text{for}~x\in(0,1),\]
as $n_R\to\infty$. Combining these facts, 
and applying Theorem~2.8 in \citet{MR1700749} to the separable space $\bS_d$, 
%we deduce that for any measurable sets $A\subseteq [0,1)$ and $B\subseteq \cS_{d-1}$,
%\begin{align*}
%\lambda_{n_0,n_R,\bm{n_S}}(A\times B)
%=(\mu_{n_R}\times \nu_{\bm{n_S}})(A\times B)
%=\ &\mu_{n_R}(A)\nu_{\bm{n_S}}(B)\\
%\to\,&\mu(A)\nu(B)
%=(\mu\times \nu)(A\times B)
%=U_d (A\times B)
%\end{align*}
we deduce that $\lambda_{n_0,n_R,{n_S}}$, 
%the uniform measure on the augmented grid $\cG^{d}_{n_0,n_R,\bm{n_S}}$, 
the product measure of $\mu_{n_R}$  and  $\nu_{{n_S}}$, 
weakly converges to $\mu\times\nu=U_d$ as $n_R,n_S\to\infty$. 

If $n_0>0$, we compare the uniform measure on the augmented grid $\cG^{d}_{0,n_R,{n_S}}$
(denoted by $\lambda_{0,n_R,{n_S}}$) 
and that on $\cG^{d}_{n_0,n_R,{n_S}}$. 
For any $U_d$-continuity set $D\subseteq\bS_d$, we obtain
\[\lambda_{0,n_R,{n_S}}(D)=\frac{\card (D\cap\cG^{d}_{0,n_R,{n_S}})}{n-n_0}
~~~\text{and}~~~
\lambda_{n_0,n_R,{n_S}}(D)=\frac{\card (D\cap\cG^{d}_{0,n_R,{n_S}})+n_0\ind(\bm{0}\in D)}{n}.
\]
Therefore,
\begin{align*}
\lvert\lambda_{0,n_R,{n_S}}(D)-\lambda_{n_0,n_R,{n_S}}(D)\rvert
%=\Big\lvert\frac{\card (D\cap\cG^{d}_{0,n_R,{n_S}})}{n-n_0}-\frac{\card (D\cap\cG^{d}_{0,n_R,{n_S}})+n_0\ind(\bm{0}\in D)}{n}\Big\rvert\\
\le\;&\Big(\frac{1}{n-n_0}-\frac{1}{n}\Big)\card (D\cap\cG^{d}_{0,n_R,{n_S}})+\frac{n_0}{n}\\
\le\;&\Big(\frac{1}{n-n_0}-\frac{1}{n}\Big)(n-n_0)+\frac{n_0}{n}=\frac{2n_0}{n}\to 0,\yestag\label{eq:measurediff}
\end{align*}
where the last step follows by noticing 
\[\frac{n_0}{n}<\frac{\min\{n_R,n_S\}}{n}\le\frac{n_S}{n_Rn_S+n_0}\le \frac{1}{n_R}\to0\]
as $n_R\to\infty$. 
We have proven in the case $n_0=0$ that $\lambda_{0,n_R,{n_S}}$
weakly converges to $U_d$ 
and then $\lambda_{0,n_R,{n_S}}(D)\to U_d(D)$. 
This, together with \eqref{eq:measurediff}, proves that
$\lambda_{n_0,n_R,{n_S}}(D)\to U_d(D)$ for any $U_d$-continuity Borel set  $D\subseteq\bS_d$, 
and equivalently, 
$\lambda_{n_0,n_R,{n_S}}$ weakly converges to $U_d$ 
as $n_R,n_S\to\infty$.
\end{proof}

\subsection{Proofs for Section \ref{sec:tests} of the main paper}

\subsubsection{Proof of Proposition \ref{prop:equiv}}

\begin{proof}
The equivalence of these three versions of the sample distance
covariance is well known; we include a proof for completeness but claim no originality here. 

The sample distance covariance defined in \citet{MR3053543} 
and \citet{MR3269983} can be described as follows. 
First define 
\begin{align*}
a_{i,j}&:=\lVert \mX_i-\mX_j\rVert,
~~~a_{i,+}:=\sum_{\ell=1}^{n}a_{i,\ell},
~~~a_{+,j}:=\sum_{k=1}^{n}a_{k,j},
~~~a_{+,+}:=\sum_{k,\ell=1}^{n}a_{k,\ell},\\
%b_{i,j}&:=\lVert \mY_i-\mY_j\rVert,
%~~~b_{i,+}:=\sum_{\ell=1}^{n}b_{i,\ell},
%~~~b_{+,j}:=\sum_{k=1}^{n}b_{k,j},
%~~~b_{+,+}:=\sum_{k,\ell=1}^{n}b_{k,\ell},\\
A^*_{i,j}&:=\begin{cases}a_{i,j}-\frac{1}{n-1}a_{i,+}-\frac{1}{n-1}a_{+,j}+\frac{1}{n(n-1)}a_{+,+}, & \text{if }i\ne j,\\
\frac{1}{n-1}a_{i,+}-\frac{1}{n(n-1)}a_{+,+}, & \text{if }i=j,\end{cases}\\
%B^*_{i,j}&:=\begin{cases}\frac{n}{n-1}(\frac{n-1}{n}b_{i,j}-\frac{1}{n}b_{i,+}-\frac{1}{n}b_{+,j}+\frac{1}{n^2}b_{+,+}), & \text{if }i\ne j,\\
%\frac{n}{n-1}(\frac{1}{n}b_{i,+}-\frac{1}{n^2}b_{+,+}), & \text{if }i=j,\end{cases}\\
\widetilde A_{i,j}&:=\begin{cases}a_{i,j}-\frac{1}{n-2}a_{i,+}-\frac{1}{n-2}a_{+,j}+\frac{1}{(n-1)(n-2)}a_{+,+}, & \text{if }i\ne j,\\
0, & \text{if }i=j,\end{cases}
%\widetilde B_{i,j}&:=\begin{cases}b_{i,j}-\frac{1}{n-2}b_{i,+}-\frac{1}{n-2}b_{+,j}+\frac{1}{(n-1)(n-2)}b_{+,+}, & \text{if }i\ne j,\\
%0, & \text{if }i=j.\end{cases}
\end{align*}
Similarly, we introduce the distances $b_{i,j}:=\lVert
\mY_i-\mY_j\rVert$, and define the 
sums $b_{i,+},b_{+,j},b_{+,+}$, and corresponding $B^*_{i,j}$,
$\widetilde B_{i,j}$ in analogy to the quantities for the $\mX_i$. 
Then the sample distance covariance from Definition~1 in \citet{MR3053543} is
\begin{equation}\label{eq:def2013}
\dCov^2_n\Big([\mX_{i}]_{i=1}^{n},[\mY_{i}]_{i=1}^{n}\Big):=
\frac{1}{n(n-3)}\Big\{\sum_{i,j=1}^{n}A^*_{i,j}B^*_{i,j}-\frac{n}{n-2}\sum_{i=1}^{n}A^*_{i,i}B^*_{i,i}\Big\},
\end{equation}
and the sample distance covariance from Equation~(3.2) in \citet{MR3269983} is
\begin{equation}\label{eq:def2014}
\dCov^2_n\Big([\mX_{i}]_{i=1}^{n},[\mY_{i}]_{i=1}^{n}\Big):=
\frac{1}{n(n-3)}\sum_{i\ne j}\widetilde A_{i,j}\widetilde B_{i,j}.
\end{equation}
%The sample distance covariance defined by Definition~3.2 in our paper is given by
%\begin{equation}\label{eq:def2019}
%\dCov^2_n\Big([\mX_{i}]_{i=1}^{n},[\mY_{i}]_{i=1}^{n}\Big):=
%\binom{n}{4}^{-1}
%\sum_{1\le i_1<\dots<i_4\le n}K\Big((\mX_{i_1},\mY_{i_1}),\dots,(\mX_{i_4},\mY_{i_4})\Big),
%\end{equation}
%where
%\begin{align*}
%K\Big((\mx_{1},\my_{1}),\dots,(\mx_{4},\my_{4})\Big)
%&:=\frac{1}{4\cdot4!}\sum_{[i_1,\dots,i_4]\in\rP(\zahl{4})} s(\mx_{i_1},\mx_{i_2},\mx_{i_3},\mx_{i_4})s(\my_{i_1},\my_{i_2},\my_{i_3},\my_{i_4}),\\
%\intertext{where the summation is over all permutations of $[1,2,3,4]$ and}
%s(\mt_{1},\mt_{2},\mt_{3},\mt_{4})&:=\lVert\mt_{1}-\mt_{2}\rVert+\lVert \mt_{3}-\mt_{4}\rVert-\lVert \mt_{1}-\mt_{3}\rVert-\lVert \mt_{2}-\mt_{4}\rVert.
%\end{align*}
We first prove the equivalence between \eqref{eq:def2013} and \eqref{eq:def2014}. 
Lemma 3.1 in \citet{MR3556612} gives that the right-hand side of \eqref{eq:def2014} equals to
\begin{equation}\label{eq:huo2016}
\frac{1}{n(n-3)}\sum_{i\ne j}a_{i,j}b_{i,j}
-\frac{2}{n(n-2)(n-3)}\sum_{i=1}^{n}a_{i,+}b_{i,+}
+\frac{a_{+,+}b_{+,+}}{n(n-1)(n-2)(n-3)}.
\end{equation}
It remains to prove that the right-hand side of \eqref{eq:def2013} equals to \eqref{eq:huo2016} as well, which can be established by straightforward calculation following the proof of Lemma 3.1 in \citet{MR3556612}. 
First, one can verify the following equalities:
\begin{align}
 &a_{i,j}=a_{j,i},~~~a_{i,i}=0,~~~a_{i,+}=a_{+,i},~~~
&&b_{i,j}=b_{j,i},~~~b_{i,i}=0,~~~b_{i,+}=b_{+,i},\label{eq:huoeq1}\\
 &\sum_{i\ne j}a_{i,j}=a_{+,+},~~~
&&\sum_{i\ne j}b_{i,j}=b_{+,+},\label{eq:huoeq1.5}\\
 &\sum_{i\ne j}a_{i,+}=\sum_{i\ne j}a_{+,j}=(n-1)a_{+,+},~~~
&&\sum_{i\ne j}b_{i,+}=\sum_{i\ne j}b_{+,j}=(n-1)b_{+,+},\label{eq:huoeq2}\\
 &\sum_{i\ne j}a_{i,j}b_{i,+}=
  \sum_{i\ne j}a_{i,j}b_{+,j}=\sum_{i=1}^{n}a_{i,+}b_{i,+},~~~
&&\sum_{i\ne j}a_{i,+}b_{i,j}=
  \sum_{i\ne j}a_{+,j}b_{i,j}=\sum_{i=1}^{n}a_{i,+}b_{i,+},\label{eq:huoeq3}\\
% &\sum_{i\ne j}a_{i,j}b_{+,j}=\sum_{i=1}^{n}a_{i,+}b_{i,+},~~~
%&&\sum_{i\ne j}a_{+,j}b_{i,j}=\sum_{i=1}^{n}a_{i,+}b_{i,+},\label{eq:huoeq4}\\
 &\sum_{i\ne j}a_{i,+}b_{+,j}=a_{+,+}b_{+,+}-\sum_{i=1}^{n}a_{i,+}b_{i,+},~~~
&&\sum_{i\ne j}a_{+,j}b_{i,+}=a_{+,+}b_{+,+}-\sum_{i=1}^{n}a_{i,+}b_{i,+}\label{eq:huoeq5}.
\end{align}
Next, we may simplify the right-hand side of \eqref{eq:def2013}. We have
\[
\frac{1}{n(n-3)}\Big\{\sum_{i,j=1}^{n}A^*_{i,j}B^*_{i,j}-\frac{n}{n-2}\sum_{i=1}^{n}A^*_{i,i}B^*_{i,i}\Big\}
=\frac{1}{n(n-3)}\Big\{\sum_{i\ne j}A^*_{i,j}B^*_{i,j}-\frac{2}{n-2}\sum_{i=1}^{n}A^*_{i,i}B^*_{i,i}\Big\},
\]
where 
\begin{align*}
\sum_{i\ne j}A^*_{i,j}B^*_{i,j}
\;&=\sum_{i\ne j}
\Big(a_{i,j}-\frac{a_{i,+}}{n-1}-\frac{a_{+,j}}{n-1}+\frac{a_{+,+}}{n(n-1)}\Big)
\Big(b_{i,j}-\frac{b_{i,+}}{n-1}-\frac{b_{+,j}}{n-1}+\frac{b_{+,+}}{n(n-1)}\Big)\\
\;&=\sum_{i\ne j}
\Big(a_{i,j}b_{i,j}
-\frac{a_{i,j}(b_{i,+}+b_{+,j})+(a_{i,+}+a_{+,j})b_{i,j}}{n-1}
+\frac{(a_{i,+}+a_{+,j})(b_{i,+}+b_{+,j})}{(n-1)^2}\\
\;&\quad+\frac{a_{i,j}b_{+,+}+a_{+,+}b_{i,j}}{n(n-1)}
-\frac{(a_{i,+}+a_{+,j})b_{+,+}+a_{+,+}(b_{i,+}+b_{+,j})}{n(n-1)^2}
+\frac{a_{+,+}b_{+,+}}{n^2(n-1)^2}\Big),\!\!\!\yestag\label{eq:summand1}\\
\sum_{i=1}^{n}A^*_{i,i}B^*_{i,i}
\;&=\sum_{i=1}^{n}
\Big(\frac{a_{i,+}}{n-1}-\frac{a_{+,+}}{n(n-1)}\Big)
\Big(\frac{b_{i,+}}{n-1}-\frac{b_{+,+}}{n(n-1)}\Big)\\
&=\frac{1}{(n-1)^2}\sum_{i=1}^{n}
\Big(a_{i,+}b_{i,+}-\frac{a_{i,+}b_{+,+}+a_{+,+}b_{i,+}}{n}+\frac{a_{+,+}b_{+,+}}{n^2}\Big)\yestag\label{eq:summand2}.
\end{align*}
Furthermore, we have
\begin{align*}
&\sum_{i\ne j}\frac{a_{i,j}(b_{i,+}+b_{+,j})+(a_{i,+}+a_{+,j})b_{i,j}}{n-1}
 \xlongequal{\eqref{eq:huoeq3}}
 \frac{4}{n-1}\sum_{i=1}^{n}a_{i,+}b_{i,+},\\
&\sum_{i\ne j}\frac{(a_{i,+}+a_{+,j})(b_{i,+}+b_{+,j})}{(n-1)^2}
 \xlongequal{\eqref{eq:huoeq5}}
 \frac{1}{(n-1)^2}\Big\{2(n-2)\sum_{i=1}^{n}a_{i,+}b_{i,+}+2a_{+,+}b_{+,+}\Big\},\\
&\sum_{i\ne j}\frac{a_{i,j}b_{+,+}+a_{+,+}b_{i,j}}{n(n-1)}
 \xlongequal{\eqref{eq:huoeq1.5}}
 \frac{2a_{+,+}b_{+,+}}{n(n-1)},\\
&\sum_{i\ne j}\frac{(a_{i,+}+a_{+,j})b_{+,+}+a_{+,+}(b_{i,+}+b_{+,j})}{n(n-1)^2}
 \xlongequal{\eqref{eq:huoeq2}}
 \frac{4a_{+,+}b_{+,+}}{n(n-1)},\\
&\sum_{i\ne j}\frac{a_{+,+}b_{+,+}}{n^2(n-1)^2}=\frac{a_{+,+}b_{+,+}}{n(n-1)},\\
\text{and}~~~&\sum_{i=1}^{n}\frac{a_{i,+}b_{+,+}+a_{+,+}b_{i,+}}{n}=\frac{2a_{+,+}b_{+,+}}{n},~~~
 \sum_{i=1}^{n}\frac{a_{+,+}b_{+,+}}{n^2}=\frac{a_{+,+}b_{+,+}}{n}.
\end{align*}
Plugging all these equalities above into \eqref{eq:summand1} and \eqref{eq:summand2} completes the proof. 
 
The equivalence between \eqref{eq:def2014} and \eqref{eq:defnew} is an immediate consequence of Lemma~1 in \citet{MR3798874}, which shows that \eqref{eq:def2014} is equivalent to
\begin{equation}\label{eq:def2018}
\dCov^2_n\Big([\mX_{i}]_{i=1}^{n},[\mY_{i}]_{i=1}^{n}\Big):=
\binom{n}{4}^{-1}
\sum_{1\le i_1<\dots<i_4\le n}K'\Big((\mX_{i_1},\mY_{i_1}),\dots,(\mX_{i_4},\mY_{i_4})\Big),
\end{equation}
where
\begin{align*}
&K'\Big((\mx_{1},\my_{1}),\dots,(\mx_{4},\my_{4})\Big)\\
:=\;&\frac{1}{4!}\sum_{[i_1,\dots,i_4]\in\rP(\zahl{4})} \lVert\mx_{i_1}-\mx_{i_2}\rVert \Big(\lVert\my_{i_1}-\my_{i_2}\rVert+\lVert\my_{i_3}-\my_{i_4}\rVert-2\lVert\my_{i_1}-\my_{i_3}\rVert\Big).
\end{align*}
By expanding the above summation, one obtains that \eqref{eq:def2018} is equivalent to 
\begin{equation}\label{eq:def2018sym}
\dCov^2_n\Big([\mX_{i}]_{i=1}^{n},[\mY_{i}]_{i=1}^{n}\Big):=
\binom{n}{4}^{-1}
\sum_{1\le i_1<\dots<i_4\le n}K''\Big((\mX_{i_1},\mY_{i_1}),\dots,(\mX_{i_4},\mY_{i_4})\Big),
\end{equation}
where
\begin{align*}
&K''\Big((\mx_{1},\my_{1}),\dots,(\mx_{4},\my_{4})\Big)\\
:=\;&\frac{1}{4!}\sum_{[i_1,\dots,i_4]\in\rP(\zahl{4})} \lVert\mx_{i_1}-\mx_{i_2}\rVert \Big(\lVert\my_{i_1}-\my_{i_2}\rVert+\lVert\my_{i_3}-\my_{i_4}\rVert-\lVert\my_{i_1}-\my_{i_3}\rVert-\lVert\my_{i_2}-\my_{i_4}\rVert\Big).
\end{align*}
Next, by expanding the summation again, we have \eqref{eq:def2018sym} is equivalent to \eqref{eq:defnew}.

Definition~5.3 (U-statistic) in \citet{jakobsen2017distance} can be written as
\begin{equation}\label{eq:def2017}
\dCov^2_n\Big([\mX_{i}]_{i=1}^{n},[\mY_{i}]_{i=1}^{n}\Big):=\mbinom{n}{6}^{-1}
\sum_{1\le i_1<\dots<i_6\le n}K^*\Big((\mX_{i_1},\mY_{i_1}),\dots,(\mX_{i_6},\mY_{i_6})\Big),
\end{equation}
where
\begin{align*}
K^*\Big((\mx_{1},\my_{1}),\dots,(\mx_{6},\my_{6})\Big)
&:=\frac{1}{6!}\sum_{[i_1,\dots,i_6]\in\rP(\zahl{6})} s(\mx_{i_1},\mx_{i_2},\mx_{i_3},\mx_{i_4})s(\my_{i_1},\my_{i_2},\my_{i_5},\my_{i_6}),\\
\text{and recall}~~~
s(\mt_{1},\mt_{2},\mt_{3},\mt_{4})&:=\lVert\mt_{1}-\mt_{2}\rVert+\lVert \mt_{3}-\mt_{4}\rVert-\lVert \mt_{1}-\mt_{3}\rVert-\lVert \mt_{2}-\mt_{4}\rVert.
\end{align*}
The equivalence between \eqref{eq:def2017} and \eqref{eq:defnew} can be verified by expanding the summation as well. 
\end{proof}

\subsubsection{Proof of Theorem \ref{thm:null}}

\begin{proof}[Proof of Theorem \ref{thm:null}]
This theorem is a corollary of Theorem \ref{cor:general}, 
which we prove in Section~\ref{subsec:general}. 
In our context,
$p_1=p$, $p_2=q$, $m=4$, and $h$ is the kernel $K$ defined in
\eqref{eq:kernel}.  The multisets $\{\mz^{(n)}_{1;j},j\in\zahl{n}\}$
and $\{\mz^{(n)}_{2;j},j\in\zahl{n}\}$ are taken to be
$\{\bmu^{(n)}_{j},j\in\zahl{n}\}:=\cG^{p}_{n_0,n_R,\bm{n_S}}$ and
$\{\mv^{(n)}_{j},j\in\zahl{n}\}:=\cG^{q}_{n_0,n_R,\bm{n_S}}$,
respectively.  Accordingly, $\mZ^{(n)}_{1}$ follows the uniform discrete
distribution over $\cG^{p}_{n_0,n_R,\bm{n_S}}$, denoted by
$\mU^{(n)}$, and $\mZ^{(n)}_{2}$ has a uniform discrete distribution over
$\cG^{q}_{n_0,n_R,\bm{n_S}}$, denoted by $\mV^{(n)}$.  The functions
$g^{(n)}_1$, $g_1$, $g^{(n)}_2$, and $g_2$ %, as shown below,
can be chosen as $-d_{\mU^{(n)}}$, $-d_{\mU}$, $-d_{\mV^{(n)}}$, and
$-d_{\mV}$, defined in the manner of \eqref{eq:fproj}, respectively.
Recall that
\begin{align*}
d_{\mU^{(n)}}(\bmu,\bmu')&:=
\lVert\bmu-\bmu'\rVert - \E\lVert\bmu-\mU^{(n)}_{2}\rVert
- \E\lVert\mU^{(n)}_{1}-\bmu'\rVert + \E\lVert\mU^{(n)}_{1}-\mU^{(n)}_{2}\rVert,\\
\text{and}~~~
d_{\mU}(\bmu,\bmu')&:=
\lVert\bmu-\bmu'\rVert - \E\lVert\bmu-\mU_{2}\rVert
- \E\lVert\mU_{1}-\bmu'\rVert + \E\lVert\mU_{1}-\mU_{2}\rVert,\yestag\label{eq:equidef}
\end{align*}
with their analogues $d_{\mV^{(n)}}(\mv,\mv')$ and $d_{\mV}(\mv,\mv')$. Here $\mU^{(n)}_{1}$ and $\mU^{(n)}_{2}$ are independent with law $\P_{\mU^{(n)}}$, 
and $\mU_{1}$ and $\mU_{2}$ are independent with law $\P_{\mU}$. 

We verify the conditions in Theorem~\ref{cor:general} as follows.
Proposition \ref{prop:unif} shows that 
$\mU^{(n)}$ and $\mV^{(n)}$ converge in distribution to $\mU$ and $\mV$, respectively. 
We also have that
(I) the kernel $K$ is symmetric and continuous on $\overline\bS_p\times\overline\bS_q$, and thus $\lVert K\rVert_{\infty}<\infty$; 
(II) $K_{1}(\mw;\P_{\mU^{(n)}}\times \P_{\mV^{(n)}}) = 0$; 
(III)
\begin{align*}
6K_{2}\Big(\mw,\mw';\P_{\mU^{(n)}}\times \P_{\mV^{(n)}}\Big)
&= \Big(-d_{\mU^{(n)}}(\bmu,\bmu')\Big)\Big(-d_{\mV^{(n)}}(\mv,\mv')\Big),\\
%K_{1}\Big(\mw;\P_{\mU}\times \P_{\mV}\Big)
%= 0
\text{and}~~~
6K_{2}\Big(\mw,\mw';\P_{\mU}\times \P_{\mV}\Big)
&= \Big(-d_{\mU}(\bmu,\bmu')\Big)\Big(-d_{\mV}(\mv,\mv')\Big),
\end{align*}
by \citet[Sec.~1.1]{MR3798874supp}.

Next we verify Assumptions~\ref{asm:sym2}--\ref{asm:unf} for $-d_{\mU^{(n)}}(\bmu,\bmu')$ and $-d_{\mU}(\bmu,\bmu')$.
It can be easily seen that
$-d_{\mU^{(n)}}(\bmu,\bmu')$ is symmetric (Assumption~\ref{asm:sym2}),
and has $\E[-d_{\mU^{(n)}}(\bmu,\mU^{(n)})]=0$ (Assumption~\ref{asm:dgn2}) 
and $\E[\{d_{\mU^{(n)}}(\mU^{(n)},\mU^{(n)}_{*})\}^2]\in(0,+\infty)$ (Assumption~\ref{asm:fnt2}) by \citet[Theorem~4(i)]{MR2382665}.
\citet[p.~3291]{MR3127883} has proved that
functions $-d_{\mU^{(n)}}(\bmu,\bmu')$
are non-negative definite (Assumption~\ref{asm:nnd2}). 
We have $-d_{\mU^{(n)}}(\bmu,\bmu')$ is equicontinuous (Assumption~\ref{asm:cnt2}) since
\begin{align*}
\lvert -d_{\mU^{(n)}}(\bmu,\bmu')-(-d_{\mU^{(n)}}(\bmu,\bmu''))\rvert
&\;=\Big\lvert \lVert\bmu-\bmu'\rVert-\lVert\bmu-\bmu''\rVert - 
\E\Big[\lVert\mU^{(n)}-\bmu'\rVert-\lVert\mU^{(n)}-\bmu''\rVert\Big]\Big\rvert\\
&\;\le 2\lVert\bmu'-\bmu''\rVert,
\end{align*}
and moreover, 
$\lvert -d_{\mU^{(n)}}(\bmu,\bmu')-(-d_{\mU^{(n)}}(\bmu''',\bmu''))\rvert\le 2\lVert\bmu-\bmu'''\rVert+2\lVert\bmu'-\bmu''\rVert.$
It remains to prove that $-d_{\mU^{(n)}}(\bmu,\bmu')$ converges uniformly to $-d_{\mU}(\bmu,\bmu')$ (Assumption~\ref{asm:unf}). 
Using the portmanteau Lemma \citep[Lemma~2.2]{MR1652247} 
and Proposition~\ref{prop:unif}, we have for all $\bmu,\bmu'\in\overline\bS_p$,
\begin{align*}
\E\lVert\bmu-\mU^{(n)}_{*}\rVert       \to  \E\lVert\bmu-\mU_{*}\rVert,
~~~
\E\lVert\mU^{(n)}-\bmu'\rVert          \to  \E\lVert\mU-\bmu'\rVert,\\
~~~\text{and}~~~
\E\lVert\mU^{(n)}-\mU^{(n)}_{*}\rVert  \to  \E\lVert\mU-\mU_{*}\rVert,
\end{align*}
and thus $-d_{\mU^{(n)}}(\bmu,\bmu')$ converges pointwisely to $-d_{\mU}(\bmu,\bmu')$.
Then the uniform convergence follows from
the equicontinuity of $-d_{\mU^{(n)}}(\bmu,\bmu')$ \citep[Exercise~7.16]{MR0385023}.
Assumptions~\ref{asm:sym2}--\ref{asm:unf} can be similarly verified for $-d_{\mV^{(n)}}(\mv,\mv')$ and $-d_{\mV}(\mv,\mv')$ as well.

Lastly, using Proposition~\ref{prop:distrfree}, 
$[\fF^{(n)}_{\mX,\pm}(\mX_i)]_{i=1}^{n}$ and $[\fF^{(n)}_{\mY,\pm}(\mY_i)]_{i=1}^{n}$ are 
uniformly distributed on $\rP(\cG^{p}_{n_0,n_R,\bm{n_S}})$ and $\rP(\cG^{q}_{n_0,n_R,\bm{n_S}})$, respectively. 
In addition, under $H_0$, $[\fF^{(n)}_{\mX,\pm}(\mX_i)]_{i=1}^{n}$ and $[\fF^{(n)}_{\mY,\pm}(\mY_i)]_{i=1}^{n}$ are independent. 
Hence our statistic is distributed as
\[
\widehat M_n = n\cdot\mbinom{n}{4}^{-1}\sum_{1\le j_1<\cdots< j_4\le n}
  K\Big(
  (\bmu^{(n)}_{\pi'_{j_1}},\mv^{(n)}_{\pi''_{j_1}}),\ldots,
  (\bmu^{(n)}_{\pi'_{j_4}},\mv^{(n)}_{\pi''_{j_4}})\Big),
\]
where $\mpi'$ and $\mpi''$ are uniformly distributed on $\rP(\zahl{n})$ and independent,
and thus the same as the form \eqref{eq:mips}
by defining permutation $\mpi$ for which $\pi_i=j$ subject to $\pi'_k=i$ and $\pi''_k=j$ for some $k$.
\end{proof}

\subsubsection{Proof of Theorem \ref{thm:consi}}

\begin{proof}[Proof of Theorem \ref{thm:consi}]
We begin by proving the first claim \eqref{eq:strconsi}.
Let $\mU_{i}$, $\mV_{i}$, $\mU^{(n)}_{i}$, $\mV^{(n)}_{i}$ denote 
$\fF_{\mX,\pm}(\mX_i)$, 
$\fF_{\mY,\pm}(\mY_i)$, 
$\fF_{\mX,\pm}^{(n)}(\mX_i)$, 
$\fF_{\mY,\pm}^{(n)}(\mY_i)$, 
respectively. 
Write $\mW_i:=(\mU_{i},\mV_{i})$, 
$\mW^{(n)}_{i}:=(\mU^{(n)}_{i},\mV^{(n)}_{i})$, 
and $\mw_i:=(\bmu_{i},\mv_{i})$.
The main idea here is to bound
\[\Big|\dCov^2_n([\mU^{(n)}_{i}]_{i=1}^{n},[\mV^{(n)}_{i}]_{i=1}^{n})-\dCov^2_n([\mU_{i}]_{i=1}^{n},[\mV_{i}]_{i=1}^{n})\Big|.\]
Recall that
\begin{align*}
\dCov^2_n\Big([\mU^{(n)}_{i}]_{i=1}^{n},[\mV^{(n)}_{i}]_{i=1}^{n}\Big)
&=\mbinom{n}{4}^{-1}
\sum_{1\le i_1<\dots<i_4\le n}K(\mW^{(n)}_{i_1},\mW^{(n)}_{i_2},\mW^{(n)}_{i_3},\mW^{(n)}_{i_4}),\\
\dCov^2_n\Big([\mU_{i}]_{i=1}^{n},[\mV_{i}]_{i=1}^{n}\Big)
&=\mbinom{n}{4}^{-1}
\sum_{1\le i_1<\dots<i_4\le n}K(\mW_{i_1},\mW_{i_2},\mW_{i_3},\mW_{i_4}),
\end{align*}
where %\fbox{in the following equation, where are $i_1,\ldots,i_4$ inside the sum?}
\[
K(\mw_{1},\dots,\mw_{4})
:=\frac{1}{4\cdot4!}\sum_{[i_1,\dots,i_4]\in\rP(\zahl{4})} s(\bmu_{i_1},\bmu_{i_2},\bmu_{i_3},\bmu_{i_4})s(\mv_{i_1},\mv_{i_2},\mv_{i_3},\mv_{i_4}),\yestag\label{eq:kuigang}
\]
and $s(\mt_{1},\mt_{2},\mt_{3},\mt_{4})
:=\lVert\mt_{1}-\mt_{2}\rVert+\lVert \mt_{3}-\mt_{4}\rVert-\lVert \mt_{1}-\mt_{3}\rVert-\lVert \mt_{2}-\mt_{4}\rVert.$
Using the inequality 
\begin{align*}
   \;& \Big|\Big\lVert \mU^{(n)}_{i_1}-\mU^{(n)}_{i_2}\Big\rVert\cdot\Big\lVert \mV^{(n)}_{i_3}-\mV^{(n)}_{i_4}\Big\rVert
      -\Big\lVert \mU_{i_1}-\mU_{i_2}\Big\rVert\cdot\Big\lVert \mV_{i_3}-\mV_{i_4}\Big\rVert\Big|\\
\le\;& \Big|\Big\lVert \mU^{(n)}_{i_1}-\mU^{(n)}_{i_2}\Big\rVert-\Big\lVert \mU_{i_1}-\mU_{i_2}\Big\rVert\Big|\cdot\Big\lVert \mV^{(n)}_{i_3}-\mV^{(n)}_{i_4}\Big\rVert 
      +\Big|\Big\lVert \mV^{(n)}_{i_3}-\mV^{(n)}_{i_4}\Big\rVert-\Big\lVert \mV_{i_3}-\mV_{i_4}\Big\rVert\Big|\cdot\Big\lVert \mU_{i_1}-\mU_{i_2}\Big\rVert\\
\le\;& \Big(\Big\lVert \mU^{(n)}_{i_1}-\mU_{i_1}\Big\rVert+\Big\lVert \mU^{(n)}_{i_2}-\mU_{i_2}\Big\rVert\Big)\cdot 2 
     + \Big(\Big\lVert \mV^{(n)}_{i_3}-\mV_{i_3}\Big\rVert+\Big\lVert \mV^{(n)}_{i_4}-\mV_{i_4}\Big\rVert\Big)\cdot 2\\
\le\;& 4\sup_{1\le i\le n}\Big\lVert \mU^{(n)}_{i}-\mU_{i}\Big\rVert
      +4\sup_{1\le i\le n}\Big\lVert \mV^{(n)}_{i}-\mV_{i}\Big\rVert,
\end{align*}
where $i_1,i_2,i_3,i_4$ could be duplicate, we deduce from \eqref{eq:kuigang} that
\begin{equation}\label{eq:16sup}
|K(\mW^{(n)}_{i_1},\dots,\mW^{(n)}_{i_4})-K(\mW_{i_1},\dots,\mW_{i_4})|\le 
16\Big(\sup_{1\le i\le n}\Big\lVert \mU^{(n)}_{i}-\mU_{i}\Big\rVert
      +\sup_{1\le i\le n}\Big\lVert \mV^{(n)}_{i}-\mV_{i}\Big\rVert\Big).
\end{equation}
This implies %\fbox{should we use $[U_i]_{i=1}^{n}$ instead of $()$ inside $\dCov()$? happening at multiple places}
\begin{align*}
   \;& \Big|\dCov^2_n\Big([\mU^{(n)}_{i}]_{i=1}^{n},[\mV^{(n)}_{i}]_{i=1}^{n}\Big)-\dCov^2_n\Big([\mU_{i}]_{i=1}^{n},[\mV_{i}]_{i=1}^{n}\Big)\Big|\\
\le\;& 16\Big(\sup_{1\le i\le n}\Big\lVert \mU^{(n)}_{i}-\mU_{i}\Big\rVert
      +\sup_{1\le i\le n}\Big\lVert \mV^{(n)}_{i}-\mV_{i}\Big\rVert\Big).\yestag\label{eq:ineconsi}
\end{align*}
Applying Proposition \ref{prop:GC} (Glivenko--Cantelli) to \eqref{eq:ineconsi} yields that
\begin{equation}
\Big|\dCov^2_n\Big([\mU^{(n)}_{i}]_{i=1}^{n},[\mV^{(n)}_{i}]_{i=1}^{n}\Big)
    -\dCov^2_n\Big([\mU_{i}]_{i=1}^{n},[\mV_{i}]_{i=1}^{n}\Big)\Big|
\stackrel{\sf a.s.}{\longrightarrow} 0.
\end{equation}
This together with 
\[
\dCov^2_n\Big([\mU_{i}]_{i=1}^{n},[\mV_{i}]_{i=1}^{n}\Big)
\stackrel{\sf a.s.}{\longrightarrow}
\dCov^2\Big(\fF_{\mX,\pm}(\mX),\fF_{\mY,\pm}(\mY)\Big),
\]
the strong consistency of $\dCov^2_n([\mU_{i}]_{i=1}^{n},[\mV_{i}]_{i=1}^{n})$ 
\citep[Theorem~5.5]{jakobsen2017distance}, yields %\eqref{eq:strconsi}.
\[
\widehat M_n/n=
\dCov^2_n\Big([\mU^{(n)}_{i}]_{i=1}^{n},[\mV^{(n)}_{i}]_{i=1}^{n}\Big)
\stackrel{\sf a.s.}{\longrightarrow}
\dCov^2\Big(\fF_{\mX,\pm}(\mX),\fF_{\mY,\pm}(\mY)\Big).
\]

Next we prove the second claim. It has been proved by \citet[Theorem~3(i)]{MR2382665}~that 
$\dCov^2(\fF_{\mX,\pm}(\mX),\fF_{\mY,\pm}(\mY))\ge 0$ and equality holds if and only if $\fF_{\mX,\pm}(\mX)$ and $\fF_{\mY,\pm}(\mY)$ are independent. It remains to show that (a) the independence of $\fF_{\mX,\pm}(\mX)$ and $\fF_{\mY,\pm}(\mY)$, is equivalent to (b) the independence of $\mX$ and $\mY$. 
It is obvious that (b) implies (a). Then we prove (a)~implies~(b).
For any Borel sets $B_1\subseteq\R^p$ and $B_2\subseteq\R^q$, using Proposition~\ref{prop:Figalli}(ii) and Definition~\ref{def:nonvanish}, we deduce
\begin{align*}
\Pr(\mX\in B_1, \mY\in B_2)=\;& \Pr(\mX\in B_1, \mY\in B_2)-\Pr(\mX\in\fF_{\mX,\pm}^{-1}(\bm{0}))-\Pr(\mY\in\fF_{\mY,\pm}^{-1}(\bm{0}))\\
\le\;& \Pr(\mX\in B_1\backslash \fF_{\mX,\pm}^{-1}(\bm{0}), \mY\in B_2\backslash \fF_{\mY,\pm}^{-1}(\bm{0}))
\le  \Pr(\mX\in B_1, \mY\in B_2),
\end{align*}
and thus 
\begin{equation}\label{eq:bigone}
\Pr(\mX\in B_1, \mY\in B_2)=\Pr(\mX\in B_1\backslash \fF_{\mX,\pm}^{-1}(\bm{0}), \mY\in B_2\backslash \fF_{\mY,\pm}^{-1}(\bm{0})).
\end{equation}
We can similarly obtain 
\begin{equation}\label{eq:smalltwo}
\Pr(\mX\in B_1)=\Pr(\mX\in B_1\backslash \fF_{\mX,\pm}^{-1}(\bm{0}))
~~~\text{and}~~~ 
\Pr(\mY\in B_2)=\Pr(\mY\in B_2\backslash \fF_{\mY,\pm}^{-1}(\bm{0})).
\end{equation}
It follows that 
\begin{align*}
\;&\Pr(\mX\in B_1, \mY\in B_2)
\xlongequal{\eqref{eq:bigone}} \Pr(\mX\in B_1\backslash \fF_{\mX,\pm}^{-1}(\bm{0}), \mY\in B_2\backslash \fF_{\mY,\pm}^{-1}(\bm{0}))\\[-.3em]
\xlongequal{\text{Prop.~\ref{prop:Figalli}(ii)}} \;& \Pr\{\fF_{\mX,\pm}(\mX)\in \fF_{\mX,\pm}(B_1\backslash \fF_{\mX,\pm}^{-1}(\bm{0})), \fF_{\mY,\pm}(\mY)\in \fF_{\mX,\pm}(B_2\backslash \fF_{\mY,\pm}^{-1}(\bm{0}))\}\\[-.3em]
\xlongequal{\fF_{\mX,\pm}(\mX) {{\rotatebox[origin=c]{90}{\footnotesize$\models$}}}\fF_{\mY,\pm}(\mY)} \;& \Pr\{\fF_{\mX,\pm}(\mX)\in \fF_{\mX,\pm}(B_1\backslash \fF_{\mX,\pm}^{-1}(\bm{0}))\}\cdot\Pr\{\fF_{\mY,\pm}(\mY)\in \fF_{\mX,\pm}(B_2\backslash \fF_{\mY,\pm}^{-1}(\bm{0}))\}\\[-.3em]
\xlongequal{\text{Prop.~\ref{prop:Figalli}(ii)}} \;& \Pr(\mX\in B_1\backslash \fF_{\mX,\pm}^{-1}(\bm{0}))\cdot\Pr(\mY\in B_2\backslash \fF_{\mY,\pm}^{-1}(\bm{0}))
\xlongequal{\eqref{eq:smalltwo}} \Pr(\mX\in B_1)\cdot\Pr(\mY\in B_2).
\end{align*}

Finally, under any fixed alternative $H_1$, combining the above two
claims yields that
\[\widehat M_n/n\stackrel{\sf a.s.}{\longrightarrow}\dCov^2\Big(\fF_{\mX,\pm}(\mX),\fF_{\mY,\pm}(\mY)\Big)>0\] 
as $n\to\infty$ and \eqref{eq:key} holds.  Thus, $\widehat
M_n\stackrel{\sf a.s.}{\longrightarrow}\infty$ and
\eqref{eq:asymp-alter} follows by noticing that $Q_{1-\alpha}$ is a constant with respect to $n$, and depends only on $p$ and $q$. 
\end{proof}

\subsection{Proofs for Section \ref{sec:theory} of the main paper}

\subsubsection{Proof of Theorem \ref{thm:general}}

%
%\fbox{As one assumption is used in the proof,  state explicitly at that place "by Assumption XXX" or sth similar}
%
%\fbox{Use $Z$ but not $U$}
%
We first state the following properties of the limiting functions:
\begin{lemma}\label{lem:property}
The limiting functions $g_i$, $i=1,2$, satisfy:
%\begin{enumerate}[label=(\roman*'),itemsep=-.5ex]
%%\item\label{asm:cmp1} 
%\item\label{asm:sym1} $g_i$ is symmetric, i.e., $g_i(\mz,\mz')=g_i(\mz',\mz)$ for all $\mz,\mz'\in\Omega_i$;
%\item\label{asm:cnt1} $g_i$ is continuous;
%\item\label{asm:nnd1} $g_i$ is non-negative definite, that is,
%\[\sum_{j_1,j_2=1}^{n}c_{j_1}c_{j_2}g_i(\mz_{j_1},\mz_{j_2})\ge 0\]
%for all $c_1,\dots,c_n\in\R$, $\mz_1,\dots,\mz_{n}\in\Omega_i$, $n\in\Z_+$.
%\item\label{asm:dgn1} $\E(g_i(\mz,\mZ'_{i}))=0$
%\item\label{asm:fnt1} $\E(g_i(\mZ_{i},\mZ'_{i})^2)\in(0,+\infty)$.
%\end{enumerate}
\begin{description}[labelwidth=.3in,itemsep=-.5ex,font=\normalfont]
\item[\namedlabel{asm:sym1}{(i')}]   $g_i$ is symmetric, i.e., $g_i(\mz,\mz')=g_i(\mz',\mz)$ for all $\mz,\mz'\in\Omega_i$;
\item[\namedlabel{asm:cnt1}{(ii')}]  $g_i$ is continuous;
\item[\namedlabel{asm:nnd1}{(iii')}] $g_i$ is non-negative definite;
\item[\namedlabel{asm:dgn1}{(iv')}]  $\E(g_i(\mz,\mZ_{i}))=0$;
\item[\namedlabel{asm:fnt1}{(v')}]   $\E(g_i(\mZ_{i},\mZ'_{i})^2)\in(0,+\infty)$.
\end{description}
\end{lemma}

\begin{proof}[Proof of Lemmma \ref{lem:property}]
% equi2 + pntwse  =  unf
%  sym2 + unf     =  sym1  (easy)
%  cnt2 + unf     =  cnt1  (Rudin 1976, Theorem 7.12)
%  nnd2 + unf     =  nnd1  (easy)
%  dgn2 + unf     =  dgn1  (see Lemma A.1)
%  fnt2 + unf     =  fnt1  (Portmanteau + unif)
Given Assumption~\ref{asm:unf}, 
Properties~\ref{asm:sym1} and \ref{asm:nnd1} readily follow from Assumptions~\ref{asm:sym2} and \ref{asm:nnd2}, respectively.
Property~\ref{asm:cnt1} follows from Assumptions~\ref{asm:cnt2} and \ref{asm:unf} by Theorem~7.12 in \citet{MR0385023}.
Property~\ref{asm:dgn1} holds by noticing 
$\E(g_i(\mz,\mZ^{(n)}_{i}))\to\E(g_i(\mz,\mZ_{i}))$ 
by Property~\ref{asm:cnt1} and the portmanteau lemma \citep[Lemma~2.2]{MR1652247}, and 
\begin{align*}
  \lvert \E g_{i}(\mz_{i},\mZ^{(n)}_{i})\rvert
&=\lvert \E g^{(n)}_{i}(\mz_{i},\mZ^{(n)}_{i})- \E g_{i}(\mz_{i},\mZ^{(n)}_{i})\rvert
\\
&\le\E \lvert g^{(n)}_{i}(\mz_{i},\mZ^{(n)}_{i})- g_{i}(\mz_{i},\mZ^{(n)}_{i})\rvert
\le \lVert g^{(n)}_{i}-g_{i}\rVert_\infty\to0,
\end{align*}
where the first step is by Assumption~\ref{asm:dgn2},  
and the last step is due to Assumption~\ref{asm:unf}.
For Property~\ref{asm:fnt1}, $\E(g_i(\mZ_{i},\mZ'_{i})^2)>0$ has been assumed in Property~\ref{asm:unf}, 
and $\E(g_i(\mZ_{i},\mZ'_{i})^2)<\infty$ since $\Omega_i$ is compact and Property~\ref{asm:cnt1}.
\end{proof}

\begin{proof}[Proof of Theorem \ref{thm:general}]
%For $i=1,2$, let $\mZ^{(n)}_{i}$ denote the random variable uniformly distributed on  
%$\{\mz^{(n)}_{i;j},j\in\zahl{n}\}$.
The proof is divided into two steps. 
The first step consists of 
defining a ``truncated'' version $\widehat D^{(n)}_{K}$ of $\widehat D^{(n)}$
and finding the limiting distribution of $\widehat D^{(n)}_{K}$.
The second step is to 
bound the difference between $\widehat D^{(n)}_{K}$ and $\widehat D^{(n)}$ 
and then derive the limiting distribution of $\widehat D^{(n)}$. 
To this end, we do some preliminary work. 
Using the Hilbert--Schmidt theorem \citep[Theorem~3.2.1, Example~3.1.15]{MR3364494},  
$g^{(n)}_{i}$ admits the following eigenfunction expansion 
by Assumptions~\ref{asm:sym2}~and~\ref{asm:fnt2},  %\fbox{why is there a prime?}
\[
g^{(n)}_i(\mz,\mz')=\sum_{k=1}^{\infty}\lambda^{(n)}_{i,k}e^{(n)}_{i,k}(\mz)e_{i,k}^{(n)}(\mz'),
\]
where $\lambda^{(n)}_{i,k},~k\in\Z_+$ are all the non-zero eigenvalues of the integral equation
\[
\E (g_i(\mz,\mZ^{(n)}_{i}) e^{(n)}_{i,k}(\mZ^{(n)}_{i}))=\lambda_{i,k}^{(n)} e^{(n)}_{i,k}(\mz)
\]
with $\lambda^{(n)}_{i,1}\ge \lambda^{(n)}_{i,2}\ge \lambda^{(n)}_{i,3}\ge \cdots >0$ 
by Assumption~\ref{asm:nnd2}, 
and orthonormal eigenfunctions $e^{(n)}_{i,k}(\mz),~k\in\Z_+$ are  such that
\begin{equation}\label{eq:ortho}
\E (e^{(n)}_{i,k}(\mZ^{(n)}_{i})e^{(n)}_{i,k'}(\mZ^{(n)}_{i}))=\ind(k=k').
\end{equation}
Since the constant function $1$ is an eigenfunction associated with eigenvalue $0$ 
by Assumption~\ref{asm:dgn2}, %$\E g^{(n)}_{i}(\mz,\mZ^{(n)}_{i})=0$, 
%\eqref{eq:ortho} implies 
% using Hilbert--Schmidt theorem \citep[Theorem~3.2.1]{MR3364494} again yields
using the orthogonality between $e^{(n)}_{i,k}$ and the constant function $1$ \citep[Theorem~3.2.1]{MR3364494} yields
\begin{equation}\label{eq:meanzero}
\E e^{(n)}_{i,k}(\mZ^{(n)}_{i})=0.
\end{equation}
We also define $\lambda_{i,k},~k\in\Z_+$ as all the non-zero eigenvalues of the integral equation 
$
\E g_i(\mz,\mZ_{i}) e_{i,k}(\mZ_{i})=\lambda_{i,k} e_{i,k}(\mz)
$
with $\lambda_{i,1}\ge \lambda_{i,2}\ge \lambda_{i,3}\ge \cdots >0$ 
by Property~\ref{asm:nnd1}, 
and orthonormal eigenfunctions $e_{i,k}(\mz),~k\in\Z_+$ are  such that
$
\E e_{i,k}(\mZ_{i})e_{i,k'}(\mZ_{i})=\ind(k=k').
$
Denote $\mk:=[k_1,k_2]$, 
$\gamma^{(n)}_{\mk}:=\lambda^{(n)}_{1,k_1}\lambda^{(n)}_{2,k_2}$, 
%$\phi^{(n)}_{\mk}(\mz_1,\mz_2):=e^{(n)}_{1,k_1}(\mz_1) e^{(n)}_{2,k_2}(\mz_2)$,
and $\Phi^{(n)}_{\mk}(j_1,j_2):=%\phi^{(n)}_{\mk}(\mz^{(n)}_{1;j_1},\mz^{(n)}_{2;j_2})
e^{(n)}_{1,k_1}(\mz^{(n)}_{1;j_1}) e^{(n)}_{2,k_2}(\mz^{(n)}_{2;j_2})$.

{\bf Step I. }
%by Lemma 4.2 in \citet{MR3541972}, 
By Theorem~4.11.8 in \citet{MR3408971}, we may write
\[
\widehat D^{(n)}=
\frac{1}{n(n-1)}\sum_{j_1\ne j_2}
  \sum_{k_1,k_2=1}^{\infty}
  \gamma^{(n)}_{\mk}
  \Phi^{(n)}_{\mk}(j_1,\pi_{j_1})
  \Phi^{(n)}_{\mk}(j_2,\pi_{j_2}).
\]
%\fbox{$\Phi()$ is a function; we should always use $f(\cdot)$ not $f[\cdot]$ to denote functions.} 
%
%\fbox{You may use a larger parenthesis to differentiate} 
For each integer $K$, we define the ``truncated'' permutation statistic
\[
\widehat D^{(n)}_{K}:=
\frac{1}{n(n-1)}\sum_{j_1\ne j_2}
  \sum_{k_1,k_2=1}^{K}
  \gamma^{(n)}_{\mk}
  \Phi^{(n)}_{\mk}(j_1,\pi_{j_1})
  \Phi^{(n)}_{\mk}(j_2,\pi_{j_2}),
\]
and derive the limiting distribution of $n\widehat D^{(n)}_{K}$ as $n\to\infty$.
Notice that  $n\widehat D^{(n)}_{K}$ can be written as
\[
n\widehat D^{(n)}_{K}=\frac{n}{n-1}\Big\{\sum_{k_1,k_2=1}^{K}\gamma^{(n)}_{\mk} \Big(\sum_{j=1}^n \frac{\Phi^{(n)}_{\mk}(j,\pi_{j})}{\sqrt{n}}\Big)^2
-\sum_{k_1,k_2=1}^{K}\gamma^{(n)}_{\mk} \Big(\frac{\sum_{j=1}^n \{\Phi^{(n)}_{\mk}(j,\pi_{j})\}^2}{n}\Big)\Big\}.\yestag\label{eq:sqsum}
\]

We separately study the two terms on the right-hand side of \eqref{eq:sqsum}, starting from the first term. 
We first establish that, for any fixed $K\in\Z_+$, the random vector 
\[
\bm\Xi^{(n)}_{K^2}:=\Big(
\sum_{j=1}^n \frac{\Phi^{(n)}_{[1,1]}(j,\pi_{j})}{\sqrt{n}},\dots,
\sum_{j=1}^n \frac{\Phi^{(n)}_{[1,K]}(j,\pi_{j})}{\sqrt{n}},\dots,
\sum_{j=1}^n \frac{\Phi^{(n)}_{[K,1]}(j,\pi_{j})}{\sqrt{n}},\dots,
\sum_{j=1}^n \frac{\Phi^{(n)}_{[K,K]}(j,\pi_{j})}{\sqrt{n}}
\Big)^{\top}
\]
has a mean of $\bm{0}$ 
and a variance-covariance matrix of $\frac{n}{n-1}\fI_{K^2}$. 
We have for $\mk=[k_1,k_2]\in\zahl{K}\times\zahl{K}$,
\begin{align*}
   &\E\sum_{j=1}^n\Phi^{(n)}_{\mk}(j,\pi_{j})
    =\frac1{n}\sum_{j_1=1}^{n}\sum_{j_2=1}^{n}\Phi^{(n)}_{\mk}(j_1,j_2)\\
=\;&\frac1{n}\sum_{j_1=1}^{n}e^{(n)}_{1,k_1}(\mz^{(n)}_{1;j_1})
             \sum_{j_2=1}^{n}e^{(n)}_{2,k_2}(\mz^{(n)}_{2;j_2})
    =n\E[e^{(n)}_{1,k_1}(\mZ^{(n)}_{1})]
      \E[e^{(n)}_{2,k_2}(\mZ^{(n)}_{2})]
    =0,\yestag\label{eq:compmean}
\end{align*}
where the last step uses \eqref{eq:meanzero}.
For $\mk=[k_1,k_2]$ and $\mk'=[k'_1,k'_2]\in\zahl{K}\times\zahl{K}$, it holds that
\begin{align*}
   &\E\Big[\sum_{j_1=1}^{n}\Phi^{(n)}_{\mk} (j_1,\pi_{j_1})
           \sum_{j_3=1}^{n}\Phi^{(n)}_{\mk'}(j_3,\pi_{j_3})\Big]\\
=\;&\E\Big[\sum_{j_1=1}^{n}\Phi^{(n)}_{\mk} (j_1,\pi_{j_1})
                           \Phi^{(n)}_{\mk'}(j_1,\pi_{j_1})
         +\sum_{j_1\ne j_3}\Phi^{(n)}_{\mk} (j_1,\pi_{j_1})
                           \Phi^{(n)}_{\mk'}(j_3,\pi_{j_3})\Big]\\
=\;&\frac{1}{n}\sum_{j_1,j_2=1}^{n}\Phi^{(n)}_{\mk} (j_1,j_2)
                             \Phi^{(n)}_{\mk'}(j_1,j_2)\\
 \;&+\frac{1}{n(n-1)}\Big(\sum_{j_1,j_2,j_3,j_4=1}^{n}
                           \Phi^{(n)}_{\mk} (j_1,j_2)
                           \Phi^{(n)}_{\mk'}(j_3,j_4)
    -\sum_{j_1,j_2,j_4=1}^{n}\Phi^{(n)}_{\mk} (j_1,j_2)
                             \Phi^{(n)}_{\mk'}(j_1,j_4)\\
 \;&\qquad\qquad\qquad
    -\sum_{j_1,j_2,j_3=1}^{n}\Phi^{(n)}_{\mk }(j_1,j_2)
                             \Phi^{(n)}_{\mk'}(j_3,j_2)
        +\sum_{j_1,j_2=1}^{n}\Phi^{(n)}_{\mk} (j_1,j_2)
                             \Phi^{(n)}_{\mk'}(j_1,j_2)\Big).\yestag\label{eq:compE}
\end{align*}
Moreover, we deduce from \eqref{eq:compE} and \eqref{eq:compmean} that
\begin{align*}
   &\Cov\Big(\sum_{j_1=1}^{n}\Phi^{(n)}_{\mk }(j_1,\pi_{j_1}),
             \sum_{j_3=1}^{n}\Phi^{(n)}_{\mk'}(j_3,\pi_{j_3})\Big)\\
=\;&\E\Big[\sum_{j_1=1}^{n}\Phi^{(n)}_{\mk }(j_1,\pi_{j_1})
           \sum_{j_3=1}^{n}\Phi^{(n)}_{\mk'}(j_3,\pi_{j_3})\Big]
     -\Big(\E\sum_{j_1=1}^{n}\Phi^{(n)}_{\mk }(j_1,\pi_{j_1})\Big)
      \Big(\E\sum_{j_3=1}^{n}\Phi^{(n)}_{\mk'}(j_3,\pi_{j_3})\Big)\\
=\;&\frac{n^2}{n-1}\Big(
         \frac1{n^2}\sum_{j_1,j_2=1}^{n}\Phi^{(n)}_{\mk }(j_1,j_2)
                                        \Phi^{(n)}_{\mk'}(j_1,j_2)
    -\frac1{n^3}\sum_{j_1,j_2,j_3=1}^{n}\Phi^{(n)}_{\mk }(j_1,j_2)
                                        \Phi^{(n)}_{\mk'}(j_3,j_2)\\
 \;&\mkern35mu-\frac1{n^3}\sum_{j_1,j_2,j_4=1}^{n}
                                        \Phi^{(n)}_{\mk }(j_1,j_2)
                                        \Phi^{(n)}_{\mk'}(j_1,j_4)
+\frac1{n^4}\sum_{j_1,j_2,j_3,j_4=1}^{n}\Phi^{(n)}_{\mk }(j_1,j_2)
                                        \Phi^{(n)}_{\mk'}(j_3,j_4)\Big)\\
%=\;&\frac{n^2}{n-1}\Big\{
%     \Big(\frac1{n}\sum_{j_1=1}^{n}e^{(n)}_{1,k_1 }(\mz^{(n)}_{1;j_1})
%                                   e^{(n)}_{1,k_1'}(\mz^{(n)}_{1;j_1})\Big)
%     \Big(\frac1{n}\sum_{j_2=1}^{n}e^{(n)}_{2,k_2 }(\mz^{(n)}_{2;j_2})
%                                   e^{(n)}_{2,k_2'}(\mz^{(n)}_{2;j_2})\Big)\\
% \;&\mkern36mu-\Big(
%          \frac1{n}\sum_{j_1=1}^{n}e^{(n)}_{1,k_1 }(\mz^{(n)}_{1;j_1})\Big)
%     \Big(\frac1{n}\sum_{j_3=1}^{n}e^{(n)}_{1,k_1'}(\mz^{(n)}_{1;j_3})\Big)
%     \Big(\frac1{n}\sum_{j_2=1}^{n}e^{(n)}_{2,k_2 }(\mz^{(n)}_{2;j_2})
%                                   e^{(n)}_{2,k_2'}(\mz^{(n)}_{2;j_2})\Big)\\
% \;&\mkern36mu-\Big(
%          \frac1{n}\sum_{j_1=1}^{n}e^{(n)}_{1,k_1 }(\mz^{(n)}_{1;j_1})
%                                   e^{(n)}_{1,k_1'}(\mz^{(n)}_{1;j_1})\Big)
%     \Big(\frac1{n}\sum_{j_2=1}^{n}e^{(n)}_{2,k_2 }(\mz^{(n)}_{2;j_2})\Big)
%     \Big(\frac1{n}\sum_{j_4=1}^{n}e^{(n)}_{2,k_2'}(\mz^{(n)}_{2;j_4})\Big)\\
% \;&\mkern36mu+\Big(
%          \frac1{n}\sum_{j_1=1}^{n}e^{(n)}_{1,k_1 }(\mz^{(n)}_{1;j_1})\Big)
%     \Big(\frac1{n}\sum_{j_3=1}^{n}e^{(n)}_{1,k_1'}(\mz^{(n)}_{1;j_3})\Big)
%     \Big(\frac1{n}\sum_{j_2=1}^{n}e^{(n)}_{2,k_2 }(\mz^{(n)}_{2;j_2})\Big)
%     \Big(\frac1{n}\sum_{j_4=1}^{n}e^{(n)}_{2,k_2'}(\mz^{(n)}_{2;j_4})\Big)\Big\}\\
=\;&\frac{n^2}{n-1}
 \Big\{\frac1{n}\sum_{j_1=1}^{n}e^{(n)}_{1,k_1 }(\mz^{(n)}_{1;j_1})
                                e^{(n)}_{1,k_1'}(\mz^{(n)}_{1;j_1})
 -\Big(\frac1{n}\sum_{j_1=1}^{n}e^{(n)}_{1,k_1 }(\mz^{(n)}_{1;j_1})\Big)
  \Big(\frac1{n}\sum_{j_3=1}^{n}e^{(n)}_{1,k_1'}(\mz^{(n)}_{1;j_3})\Big)\Big\}\\
&\mkern46mu\Big\{\frac1{n}\sum_{j_2=1}^{n}e^{(n)}_{2,k_2}(\mz^{(n)}_{2;j_2})e^{(n)}_{2,k_2'}(\mz^{(n)}_{2;j_2})
 -\Big(\frac1{n}\sum_{j_2=1}^{n}e^{(n)}_{2,k_2}(\mz^{(n)}_{2;j_2})\Big)
  \Big(\frac1{n}\sum_{j_4=1}^{n}e^{(n)}_{2,k_2'}(\mz^{(n)}_{2;j_4})\Big)\Big\}\\
=\;&\frac{n^2}{n-1}\Cov\Big(e^{(n)}_{1,k_1}(\mZ^{(n)}_{1}),
                            e^{(n)}_{1,k_1'}(\mZ^{(n)}_{1})\Big)
                   \Cov\Big(e^{(n)}_{2,k_2}(\mZ^{(n)}_{2}),
                            e^{(n)}_{2,k_2'}(\mZ^{(n)}_{2})\Big)\\
=\;&\frac{n^2}{n-1}\ind(k_1=k'_1)\ind(k_2=k'_2)
=   \frac{n^2}{n-1}\ind(\mk=\mk'),
%=\;&\begin{cases}\dfrac{n^2}{n-1}, & \text{ if }\mk=\mk',\\ 0, & \text{ otherwise},\end{cases}
\yestag\label{eq:compcov}
\end{align*}
where the penultimate step uses \eqref{eq:ortho} and \eqref{eq:meanzero}.
Combining \eqref{eq:compmean} and \eqref{eq:compcov} confirms the claim that 
the mean and the variance-covariance matrix of $\sqrt{(n-1)/n}\bm\Xi^{(n)}_{K^2}$ 
are $\bm{0}$ and $\fI_{K^2}$, respectively.

This claim about $\sqrt{(n-1)/n}\bm\Xi^{(n)}_{K^2}$ allows us to 
use the multivariate Berry--Ess{\'e}en theorem for permutation statistics
\citep[Theorem~1]{MR1245764}.
Specifically, we present the version revised by \citet[p.~3]{raic15multivariate}.
Define $\bm\Xi_{K^2}$ as a standard $K^2$-dimensional Gaussian random vector
with independent univariate standard Gaussian entries
\[
\bm\Xi_{K^2}=(
 \xi_{[1,1]},\dots,\xi_{[1,K]},\dots,
 \xi_{[K,1]},\dots,\xi_{[K,K]})^{\top},
\]
and $\cH$ as the family of all measurable convex sets in $\R^{K^2}$.
We obtain that for all $H\in\cH$, there exists a universal constant $c_1$ such that
\begin{align*}
   \;&\Big\lvert\P\Big(\sqrt{\frac{n-1}{n}}\bm\Xi^{(n)}_{K^2}\in H\Big)-\P(\bm\Xi_{K^2}\in H)\Big\rvert\\
\le\;& c_1(K^2)^{1/4}\frac{1}{n}\sum_{j_1,j_2=1}^{n}\Big(\sum_{k_1,k_2=1}^{K}\Big\{\sqrt{\frac{n-1}{n^2}}
  \Phi^{(n)}_{\mk}(j_1,j_2)\Big\}^2\Big)^{3/2}\\
\le\;& c_1K^{1/2}\frac{1}{n^{5/2}}\sum_{j_1,j_2=1}^{n}\Big(\sum_{k_1,k_2=1}^{K}
  \Big\{e^{(n)}_{1,k_1}(\mz^{(n)}_{1;j_1})\Big\}^2
  \Big\{e^{(n)}_{2,k_2}(\mz^{(n)}_{2;j_2})\Big\}^2\Big)^{3/2}
  = O(n^{-1/2}),\yestag\label{eq:mltbe}
\end{align*}
where the last step is due to the facts that $K$ is fixed and that
$\sup_{n}\lVert e^{(n)}_{i,k}\rVert_{\infty}<\infty$ for each $i=1,2$ and any fixed $k$,
as we will show in Lemma~\ref{lem:aux1}(b).
%
%\fbox{which step will be shown later? The first inequality? The first equality? The second equality?}
%
%\fbox{suggestion: create lemmas. The lemma may just state as "(A.10) is true". And then you can cite it here.}
%
%Cram\'er--Wold theorem \fbox{citation} guarantees that 
Notice that for any $a_1,\dots,a_{K^2}\in\R$, the set $(-\infty,a_1]\times\cdots\times(-\infty,a_{K^2}]$ is a convex subset of $\R^{K^2}$.
It follows that
$\sqrt{(n-1)/n}\bm\Xi^{(n)}_{K^2}\stackrel{\sf d}{\longrightarrow}\bm\Xi_{K^2}$,
and thus, $\bm\Xi^{(n)}_{K^2}\stackrel{\sf d}{\longrightarrow}\bm\Xi_{K^2}$
by Slutsky's theorem \citep[Theorem~2.8]{MR1652247}. 
On the other hand,  
since
\begin{equation}\label{eq:gammaconver}
\gamma^{(n)}_{\mk}=\lambda^{(n)}_{1,k_1}\lambda^{(n)}_{2,k_2}\to\lambda_{1,k_1}\lambda_{2,k_2}=\gamma_{\mk}
\end{equation}
by Lemma~\ref{lem:aux1}(a), we have $\bm\Gamma^{(n)}_{K^2}\to\bm\Gamma_{K^2}$ where
\begin{align*}
\bm\Gamma^{(n)}_{K^2}&:=(
 \gamma^{(n)}_{[1,1]},\dots,\gamma^{(n)}_{[1,K]},\dots,
 \gamma^{(n)}_{[K,1]},\dots,\gamma^{(n)}_{[K,K]})^{\top},\\
\text{and}~~~
\bm\Gamma_{K^2}&:=(
 \gamma_{[1,1]},\dots,\gamma_{[1,K]},\dots,
 \gamma_{[K,1]},\dots,\gamma_{[K,K]})^{\top}.
\end{align*}
%
%\fbox{where did you prove this? You need a lemma here}
%
%\fbox{this lemma may just state as "$\gamma^{(n)}_{\mk}$ converges to $\gamma_{\mk}$ uniformly for $\mk\in\zahl{K}\times\zahl{K}$"}
%
We find using the generalized Slutsky's theorem (as a consequence of Theorem~2.7 in \citealp[p.10--11]{MR1652247}) that
\begin{align*}
\sum_{k_1,k_2=1}^{K}\gamma^{(n)}_{\mk} \Big(\sum_{j=1}^n \frac{\Phi^{(n)}_{\mk}(j,\pi_{j})}{\sqrt{n}}\Big)^2
=\;&\bm\Gamma^{(n)}_{K^2}\cdot\Big(\bm\Xi^{(n)}_{K^2}\circ\bm\Xi^{(n)}_{K^2}\Big)\\
\stackrel{\sf d}{\longrightarrow}\;&
\bm\Gamma_{K^2}\cdot\Big(\bm\Xi_{K^2}\circ\bm\Xi_{K^2}\Big)
=\sum_{k_1,k_2=1}^{K}\gamma_{\mk}\xi_{\mk}^2,\yestag\label{eq:term1}
\end{align*}
recognizing the function $f(\mx,\my)=\mx\cdot(\my\circ\my)$ for $\mx,\my\in\R^{K^2}$ as continuous.
This completes the analysis of the first term in \eqref{eq:sqsum}.
%
%\fbox{give me a detailed derivation on how you prove (A.6) converges to $\sum_{k=1}^{K}\lambda_kZ_k^2$; this is the crucial step} 
%
%that the first term on the right-hand side of \eqref{eq:sqsum} converges in distribution to $\sum_{k=1}^{K}\lambda_kZ_k^2$. 
%

We turn to the second term in \eqref{eq:sqsum}. 
Denoting $n^{-1}\sum_{j=1}^n \{\Phi^{(n)}_{\mk}(j,\pi_{j})\}^2$ by $T^{(n)}_{\mk}$, 
we have by Theorem~2 in \citet{MR0044058},
\begin{align*}
\E[T^{(n)}_{\mk}]
&=\frac1{n^2}\sum_{j_1=1}^{n} \Big\{e^{(n)}_{1,k_1}(\mz^{(n)}_{1;j_1})\Big\}^2
             \sum_{j_2=1}^{n} \Big\{e^{(n)}_{2,k_2}(\mz^{(n)}_{2;j_2})\Big\}^2=1,\yestag\\
\Var(T^{(n)}_{\mk})
&=\frac{1}{n-1}
 \Big(\frac{\sum_{j_1=1}^n[\{e^{(n)}_{1,k_1}(\mz^{(n)}_{1;j_1})\}^2-1]^2}{n}\Big)
 \Big(\frac{\sum_{j_2=1}^n[\{e^{(n)}_{2,k_2}(\mz^{(n)}_{2;j_2})\}^2-1]^2}{n}\Big)
 =O(n^{-1}),\yestag\label{eq:hoeffvar}
\end{align*}
where the last step in \eqref{eq:hoeffvar} uses Lemma~\ref{lem:aux1}(b).
Therefore,  
%for the second term in \eqref{eq:sqsum}, 
we have 
\begin{align*}
\E\Big[\sum_{k_1,k_2=1}^{K}\gamma^{(n)}_{\mk} T^{(n)}_{\mk}\Big]
&=\sum_{k_1,k_2=1}^{K}\gamma^{(n)}_{\mk} \E[T^{(n)}_{\mk}]
 =\sum_{k_1,k_2=1}^{K}\gamma^{(n)}_{\mk},\yestag\label{eq:pptt}\\
\text{and}~~~
\Var\Big(\sum_{k_1,k_2=1}^{K}\gamma^{(n)}_{\mk} T^{(n)}_{\mk}\Big)
&\le\Big(\sum_{k_1,k_2=1}^{K}\gamma^{(n)}_{\mk} \sqrt{\Var(T^{(n)}_{\mk})}\Big)^2=O(n^{-1}),\yestag\label{eq:Mink}
\end{align*}
where the first step in \eqref{eq:Mink} applies Minkowski's inequality 
%\citep[Exercise.~1.5.3]{MR2722836}. 
\citep[p.~242]{MR1324786}
and the last step is based on \eqref{eq:gammaconver} and \eqref{eq:hoeffvar}. 
By \citet[Exercise~4.3.5]{DeGroot2012Pas}, it follows
that 
\begin{multline*}
\E\Big[\Big(\sum_{k_1,k_2=1}^{K}\gamma^{(n)}_{\mk} T^{(n)}_{\mk}-\sum_{k_1,k_2=1}^{K}\gamma_{\mk}\Big)^2\Big]
=\Big(\E\Big[\sum_{k_1,k_2=1}^{K}\gamma^{(n)}_{\mk} T^{(n)}_{\mk}\Big]-\sum_{k_1,k_2=1}^{K}\gamma_{\mk}\Big)^2+\Var\Big(\sum_{k_1,k_2=1}^{K}\gamma^{(n)}_{\mk} T^{(n)}_{\mk}\Big)\\
=\Big(\sum_{k_1,k_2=1}^{K}(\gamma^{(n)}_{\mk}-\gamma_{\mk})\Big)^2+\Var\Big(\sum_{k_1,k_2=1}^{K}\gamma^{(n)}_{\mk} T^{(n)}_{\mk}\Big)
=o(1).\yestag
\end{multline*}
Here the second last step uses \eqref{eq:pptt}, 
and the last step is based on \eqref{eq:gammaconver} and \eqref{eq:Mink}.
Hence for the second term in \eqref{eq:sqsum}, we have
\begin{equation}\label{eq:term2}
\sum_{k_1,k_2=1}^{K}\gamma^{(n)}_{\mk} T^{(n)}_{\mk}\stackrel{\sf p}{\longrightarrow}\sum_{k_1,k_2=1}^{K}\gamma_{\mk}.
\end{equation}
%
%\fbox{write more.}
%
Putting the two pieces \eqref{eq:term1} and \eqref{eq:term2} together, % --- two parts in \eqref{eq:sqsum}, 
and using Slutsky's theorem once again, we have
\begin{equation}\label{eq:crucial}
n\widehat D^{(n)}_{K}\stackrel{\sf d}{\longrightarrow}\sum_{k_1,k_2=1}^{K}\gamma_{\mk}(\xi_{\mk}^2-1).
\end{equation}
%This completes the study of \eqref{eq:sqsum}.
This completes Step I.

{\bf Step II. }
We will prove 
$n\widehat D^{(n)}\stackrel{\sf d}{\longrightarrow}\sum_{\mk}\gamma_{\mk}(\xi_{\mk}^2-1)$ 
starting from \eqref{eq:crucial}.
Following arguments of \citet[Chap.~5.5.2]{MR595165}, 
%\citet[Chap.~3.2.2]{MR1075417}, 
we first control $\E|n\widehat D^{(n)}-n\widehat D^{(n)}_{K}|^2$.
Letting 
\[S^{(n)}_{\mk}:=\sum_{j_1\ne j_2}
\Phi^{(n)}_{\mk}(j_1,\pi_{j_1})
\Phi^{(n)}_{\mk}(j_2,\pi_{j_2})
=\sum_{j_1\ne j_2}
e^{(n)}_{1,k_1}(\mz_{1;j_1})
e^{(n)}_{1,k_1}(\mz_{1;j_2})
e^{(n)}_{2,k_2}(\mz_{2;\pi_{j_1}})
e^{(n)}_{2,k_2}(\mz_{2;\pi_{j_2}}),\]
we have
\[
n\widehat D^{(n)}-n\widehat D^{(n)}_{K}
%=\frac{1}{n-1}\sum_{j_1\ne j_2}
%  \sum_{\mk\not\in \zahl{K}\times\zahl{K}}
%  \gamma^{(n)}_{\mk}
%  \Phi^{(n)}_{\mk}(j_1,\pi_{j_1})
%  \Phi^{(n)}_{\mk}(j_2,\pi_{j_2})
=\frac{1}{n-1}
  \sum_{\mk\not\in \zahl{K}\times\zahl{K}}
  \gamma^{(n)}_{\mk}S^{(n)}_{\mk}.\yestag
\]
Equations (2.2)--(2.3) in \citet{MR857081} give
\begin{align*}
\E[S^{(n)}_{\mk}]
=\;&n(n-1)\mu^{(n)}_{1,k_1}\mu^{(n)}_{2,k_2}
=n(n-1)
\Big(-\frac{1}{n-1}\Big)
\Big(-\frac{1}{n-1}\Big)=\frac{n}{n-1},\yestag\label{eq:barbE}\\
\text{and}~~~
\Var(S^{(n)}_{\mk})
=\;&\frac{4n^2(n-2)^2}{(n-1)}
\Big(\frac{\sum_{j_1=1}^{n}\{\zeta^{(n)}_{1,k_1;j_1}/(n-2)\}^2}{n}\Big)
\Big(\frac{\sum_{j_1=1}^{n}\{\zeta^{(n)}_{2,k_2;j_1}/(n-2)\}^2}{n}\Big)\\
\;&+\frac{2n(n-1)^2}{n-3}
\Big(\frac{\sum_{j_1\ne j_2}\{\eta^{(n)}_{1,k_1;j_1,j_2}\}^2}{n(n-1)}\Big)
\Big(\frac{\sum_{j_1\ne j_2}\{\eta^{(n)}_{2,k_2;j_1,j_2}\}^2}{n(n-1)}\Big),\yestag\label{eq:barbVar-half}
%\Cov(S_{n,k},S_{n,k'})
%=\;&\frac{4n^2}{(n-1)(n-2)^2}\Big(\frac{\sum_{i=1}^{n} e_{n,k_1,i}^{*}e_{n,k_1',i}^{*}}{n}\Big)\Big(\frac{\sum_{i=1}^{n} f_{n,k_2,i}^{*}f_{n,k_2',i}^{*}}{n}\Big)\\
%\;&+\frac{2n(n-1)^2}{n-3}\Big(\frac{\sum_{i\ne j} e_{n,k_1;i,j}^{\#}e_{n,k_1';i,j}^{\#}}{n(n-1)}\Big)\Big(\frac{\sum_{i\ne j} f_{n,k_2;i,j}^{\#}f_{n,k_2';i,j}^{\#}}{n(n-1)}\Big),\yestag\label{eq:barbCov}
\end{align*}
where for $i=1,2$,
\begin{align*}
\mu^{(n)}_{i,k}
:=\;&
\frac{1}{n(n-1)}\sum_{j_1\ne j_2}
e^{(n)}_{i,k}(\mz_{i;j_1})
e^{(n)}_{i,k}(\mz_{i;j_2})
=
-\frac{1}{n(n-1)}\sum_{j_1=1}^{n}
\{e^{(n)}_{i,k}(\mz_{i;j_1})\}^2
=-\frac{1}{n-1},\\
\zeta^{(n)}_{i,k;j_1}
:=\;&
\sum_{j_2:j_2\ne j_1}\Big\{
e^{(n)}_{i,k}(\mz^{(n)}_{i;j_1})
e^{(n)}_{i,k}(\mz^{(n)}_{i;j_2})-\mu^{(n)}_{i,k}\Big\}
=-\{e^{(n)}_{i,k}(\mz^{(n)}_{i;j_1})\}^2+1,\\
\text{and}~~~
\eta^{(n)}_{i,k;j_1,j_2}
:=\;&
e^{(n)}_{i,k}(\mz^{(n)}_{i;j_1})
e^{(n)}_{i,k}(\mz^{(n)}_{i;j_2})
-\frac{\zeta_{i,k;j_1}}{n-2}
-\frac{\zeta_{i,k;j_2}}{n-2}-\mu^{(n)}_{i,k}\\
=\;&e^{(n)}_{i,k}(\mz^{(n)}_{i;j_1})
e^{(n)}_{i,k}(\mz^{(n)}_{i;j_2})
+\frac{\{e^{(n)}_{i,k}(\mz^{(n)}_{i;j_1})\}^2-1}{n-2}
+\frac{\{e^{(n)}_{i,k}(\mz^{(n)}_{i;j_2})\}^2-1}{n-2}
+\frac{1}{n-1}.
\end{align*}
To further bound \eqref{eq:barbVar-half}, we apply the following inequalities for $i=1,2$:
\begin{align*}
\sum_{j_1=1}^{n}\{\zeta^{(n)}_{i,k;j_1}\}^2
&=\sum_{j_1=1}^{n}(1-\zeta^{(n)}_{i,k;j_1})^2-n
 =\sum_{j_1=1}^{n}\Big(n-\sum_{j_2:j_2\ne j_1}\{e^{(n)}_{i,k}(\mz^{(n)}_{i;j_2})\}^2\Big)
                  \Big(\{e^{(n)}_{i,k}(\mz^{(n)}_{i;j_1})\}^2\Big)-n\\
&\le n\sum_{j_1=1}^{n}\{e^{(n)}_{i,k}(\mz^{(n)}_{i;j_1})\}^2-n
 = n(n-1),\\
\text{and}~~~
\sum_{j_1\ne j_2}\{\eta^{(n)}_{i,k;j_1,j_2}\}^2
%&=n(n-1)-\frac{n}{n-1}-\frac{n}{n-2}\Big(\sum_{j_1=1}^{n}\{e^{(n)}_{i,k}(\mz^{(n)}_{i;j_1})\}^4-n\Big)\\
&=n(n-1)-\frac{n}{n-1}-\frac{n}{n-2}\sum_{j_1=1}^{n}\Big(\{e^{(n)}_{i,k}(\mz^{(n)}_{i;j_1})\}^2-1\Big)^2\le n(n-1)-\frac{n}{n-1},
\end{align*}
Using the inequalities we deduce that for all $\mk\in\Z_+\times\Z_+$,
\begin{equation}
\Var(S^{(n)}_{\mk})\le\frac{4n^2(n-1)}{(n-2)^2}+\frac{2n(n-1)^2}{n-3}.\label{eq:barbVar}
\end{equation}
Combining \eqref{eq:barbE} and \eqref{eq:barbVar}, we find that for $n\ge 14$,
\begin{align*}
 \;&\E|n\widehat D^{(n)}-n\widehat D^{(n)}_{K}|^2
%=\frac{1}{(n-1)^2}\E\Big[\Big(\sum_{\mk\not\in \zahl{K}\times\zahl{K}}\gamma^{(n)}_{\mk}S^{(n)}_{\mk}\Big)^2\Big]\\
=\frac{1}{(n-1)^2}\Big[\Big(\E\sum_{\mk\not\in \zahl{K}\times\zahl{K}}\gamma^{(n)}_{\mk} S^{(n)}_{\mk}\Big)^2
%+\sum_{\mk\not\in \zahl{K}\times\zahl{K}}\{\gamma^{(n)}_{\mk}\}^2\Var(S^{(n)}_{\mk})
%+\sum_{\substack{\mk\ne\mk'}}\gamma^{(n)}_{\mk}\gamma^{(n)}_{\mk'}\Cov(S^{(n)}_{\mk},S^{(n)}_{\mk'})\Big]\\
+\Var\Big(\sum_{\mk\not\in \zahl{K}\times\zahl{K}}\gamma^{(n)}_{\mk}S^{(n)}_{\mk}\Big)\Big]\\
\le\;&\frac{1}{(n-1)^2}\Big[\Big(\sum_{\mk\not\in \zahl{K}\times\zahl{K}}\gamma^{(n)}_{\mk}\E S^{(n)}_{\mk}\Big)^2
+\Big(\sum_{\mk\not\in \zahl{K}\times\zahl{K}}\gamma^{(n)}_{\mk}\sqrt{\Var(S^{(n)}_{\mk})}\Big)^2\Big]\\
\le\;&\frac{1}{(n-1)^2}\Big[\Big(\frac{n}{n-1}\Big)^2+\frac{4n^2(n-1)}{(n-2)^2}+\frac{2n(n-1)^2}{n-3}\Big]\Big(\sum_{\mk\not\in \zahl{K}\times\zahl{K}}\gamma^{(n)}_{\mk}\Big)^2
\le   3\Big(\sum_{\mk\not\in \zahl{K}\times\zahl{K}}\gamma^{(n)}_{\mk}\Big)^2\\
\le\;&9\Big[\Big(\sum_{\mk\not\in \zahl{K}\times\zahl{K}}\gamma_{\mk}\Big)^2
+\Big(\sum_{\mk\in \zahl{K}\times\zahl{K}}(\gamma^{(n)}_{\mk}-\gamma_{\mk})\Big)^2
+\Big(\sum_{\mk\in\Z_+\times\Z_+}\gamma^{(n)}_{\mk}-\sum_{\mk\in\Z_+\times\Z_+}\gamma_{\mk}\Big)^2\Big]\yestag\label{eq:analog-serfling}.
\end{align*}

We next verify that $\E|n\widehat D^{(n)}-n\widehat D^{(n)}_{K}|^2$
can be made arbitrarily small for all $K$ large enough and all $n\ge N(K)$ with $N(K)$ possibly depending on $K$.
Fix any small $\epsilon>0$. 
The first term in \eqref{eq:analog-serfling} is smaller than $\epsilon/3$ 
as long as $K$ is large enough, since
\[\sum_{\mk\in\Z_+\times\Z_+}\gamma_{\mk}
=\E g_1(\mZ_{1},\mZ_{1}) \cdot \E g_2(\mZ_{2},\mZ_{2})<\infty\]
by Properties~\ref{asm:sym1}--\ref{asm:nnd1} and Mercer's theorem \citep[Theorem~3.11.9(b)]{MR3364494}. 
In view of \eqref{eq:gammaconver},
the second term in \eqref{eq:analog-serfling} 
will be smaller than $\epsilon/3$ 
for each fixed $K$ and all $n\ge N(K)$, 
where $N(K)$ may depend on $K$.
For the third term, combining the facts that $\E g_{i}(\mZ^{(n)}_{i},\mZ^{(n)}_{i})\to \E g_{i}(\mZ_{i},\mZ_{i})$ by the portmanteau lemma \citep[Lemma~2.2]{MR1652247}, and that
\[\E \lvert g^{(n)}_{i}(\mZ^{(n)}_{i},\mZ^{(n)}_{i})- g_{i}(\mZ^{(n)}_{i},\mZ^{(n)}_{i})\rvert \le \lVert g^{(n)}_{i}-g_{i}\rVert_\infty\to 0,\]
by Assumption~\ref{asm:unf}, we deduce for $i=1,2$, that $\E g^{(n)}_{i}(\mZ^{(n)}_{i},\mZ^{(n)}_{i})\to\E g_{i}(\mZ_{i},\mZ_{i})$ as $n\to\infty$.
Recalling Assumptions~\ref{asm:sym2}--\ref{asm:nnd2} 
and Properties~\ref{asm:sym1}--\ref{asm:nnd1}, %\fbox{check!!!}
it holds by Mercer's theorem once again \citep[Theorem~3.11.9(b)]{MR3364494} that 
\begin{align*}
   &\Big(\sum_{\mk\in\Z_+\times\Z_+}\gamma^{(n)}_{\mk}
        -\sum_{\mk\in\Z_+\times\Z_+}\gamma_{\mk}\Big)^2\\
=\;&\Big(\E g^{(n)}_1(\mZ^{(n)}_{1},\mZ^{(n)}_{1})
    \cdot\E g^{(n)}_2(\mZ^{(n)}_{2},\mZ^{(n)}_{2})
        -\E g_1(\mZ_{1},\mZ_{1})
    \cdot\E g_2(\mZ_{2},\mZ_{2})\Big)^2,
\end{align*}
which is smaller than $\epsilon/3$ for $n$ large enough. 
Adding these three terms together yields the result.

We are now ready to prove $n\widehat D^{(n)}\stackrel{\sf d}{\longrightarrow}\sum_{\mk}\gamma_{\mk}(\xi_{\mk}^2-1)$ using L\'{e}vy's continuity theorem \citep[Theorem~26.3]{MR1324786}. We have
\begin{align*}
   \;&\Big\lvert\E\Big[\exp\Big(\sfi tn\widehat D^{(n)}\Big)\Big]-
\E\Big[\exp\Big(\sfi t\sum_{\mk}\gamma_{\mk}(\xi_{\mk}^2-1)\Big)\Big]\Big\rvert\\
\le\;&\Big\lvert\E\Big[\exp\Big(\sfi tn\widehat D^{(n)}\Big)\Big]-
       \E\Big[\exp\Big(\sfi tn\widehat D^{(n)}_{K}\Big)\Big]\Big\rvert\\
  \;&+
\Big\lvert\E\Big[\exp\Big(\sfi tn\widehat D^{(n)}_{K}\Big)\Big]-
\E\Big[\exp\Big(\sfi t\sum_{\mk\in \zahl{K}\times\zahl{K}}\gamma_{\mk}(\xi_{\mk}^2-1)\Big)\Big]\Big\rvert\\
   \;&+\Big\lvert\E\Big[\exp\Big(\sfi t\sum_{\mk\in \zahl{K}\times\zahl{K}}\gamma_{\mk}(\xi_{\mk}^2-1)\Big)\Big]-
\E\Big[\exp\Big(\sfi t\sum_{\mk}\gamma_{\mk}(\xi_{\mk}^2-1)\Big)\Big]\Big\rvert\\
\le\;&|t|\Big(\E|n\widehat D^{(n)}-n\widehat D^{(n)}_{K}|^2\Big)^{1/2}+\Big\lvert\E\Big[\exp\Big(\sfi tn\widehat D^{(n)}_{K}\Big)\Big]-
\E\Big[\exp\Big(\sfi t\sum_{\mk\in \zahl{K}\times\zahl{K}}\gamma_{\mk}(\xi_{\mk}^2-1)\Big)\Big]\Big\rvert\\
   \;&+|t|\Big(2\sum_{\mk\not\in \zahl{K}\times\zahl{K}}\gamma^2_{\mk}\Big)^{1/2}=:I_{n,K}+I\!I_{n,K}+I\!I\!I_{K}.
\yestag\label{eq:serfling199}
\end{align*}
In the last inequality, the first term arises from the bound 
$|\E[e^{\sfi tX}]-\E[e^{\sfi tY}]|\le |t|(\E|X-Y|^2)^{1/2}$,  
and the last term is due to Equation~(4.3.10) in \citet{MR1472486}. 
%derivations in \citet[p.~197--199]{MR595165}. 
Fix $t$, and let arbitrarily small $\epsilon>0$ be given. We have
proven that there exists $K_1$ such that for all $K\ge K_1$ and all
$n\ge N(K)$, where $N(K)$ may depend on $K$, it holds that $I_{n,K}<\epsilon/3$. We can find $K_2$ such that $I\!I\!I_{K}<\epsilon/3$ for all $K\ge K_2$ because 
$\sum_{\mk}\gamma^2_{\mk}=\E[\{g_1(\mZ_{1},\mZ'_{1})\}^2]\E[\{g_2(\mZ_{2},\mZ'_{2})\}^2]<\infty$ 
by Property~\ref{asm:fnt1}. 
Taking $K_0=\max(K_1,K_2)$, we can choose $N_0\ge N(K_0)$ so that $I\!I_{n,K_0}<\epsilon/3$ for all $n\ge N_0$ since \eqref{eq:crucial} holds for $K_0$. Then for all $n\ge N_0$,
\[\Big\lvert\E\Big[\exp\Big(\sfi tn\widehat D^{(n)}\Big)\Big]-
\E\Big[\exp\Big(\sfi t\sum_{\mk}\gamma_{\mk}(\xi_{\mk}^2-1)\Big)\Big]\Big\rvert\le
I_{n,K_0}+I\!I_{n,K_0}+I\!I\!I_{K_0}<\epsilon,\]
and the proof of the theorem is complete.
\end{proof}

%It remains to control \eqref{eq:mltbe}, \eqref{eq:hoeffvar}, and the second term in \eqref{eq:analog-serfling}. 
%
%\fbox{the place to put your lemmas}
%
\begin{lemma}\label{lem:aux1}
%
%\fbox{the statement of the lemma needs not to be self-complete} XXX 
%
For each $i=1,2$ and any fixed $k\in\Z_+$, we have
\begin{enumerate*}
\item[(a)] $\lambda^{(n)}_{i,k}\to\lambda_{i,k}$ as $n\to\infty$;
\item[(b)] $\sup_{n}\lVert e^{(n)}_{i,k}\rVert_{\infty}<\infty$.
\end{enumerate*}
\end{lemma}

\begin{proof}[Proof of Lemma~\ref{lem:aux1}]
We employ results in \citet{MR0220105}.
Consider the Banach space $C(\Omega_i)$ of all continuous functions $f$ on 
$\Omega_i$ equipped with the sup norm 
$\lVert f\rVert_{\infty}:=\sup_{\mz}\lvert f(\mz)\rvert$.
Define operators $\sA$ and $\sA_n$ on $C(\Omega_i)$ for each $i=1,2$ as
\[
(\sA f)(\mz):=\E (g_{i}(\mz,\mZ_{i})f(\mZ_{i}))
~~~\text{and}~~~
(\sA_n f)(\mz):=\E (g^{(n)}_{i}(\mz,\mZ^{(n)}_{i})f(\mZ^{(n)}_{i})).\yestag
\]

We first verify the three assumptions stated in \citet[Sect.~1]{MR0220105}:
\begin{description}[labelwidth=.3in,itemsep=-.5ex,font=\normalfont]
\item[(1)] $\sA$ and $\sA_n,~n\in\Z_+$ are linear operators on Banach space $C(\Omega_i)$ into itself;
\item[(2)] $\lVert \sA_n f -\sA f\rVert_{\infty}\to 0$ for each $f\in C(\Omega_i)$;
\item[(3)] $\{\sA_n,n\in\Z_+\}$ is collectively compact, i.e., the set 
\[\cB:=\Big\{\sA_n f:n\in\Z_+~~\text{and}~~\lVert f\rVert_{\infty} \le 1,~\text{for}~f\in C(\Omega_i)\Big\}\]
has compact closure. 
\end{description}
Note that Assumptions~(2) and (3) together imply that the operator $\sA$ is compact 
\citep[Chap.~1.4]{MR0443383}. 
%\citep[Corollary~3.3]{anselone1967collectively}.
Assumption~(1) is obvious by Property~\ref{asm:cnt1} and Assumption~\ref{asm:cnt2}. 
%\begin{align*}
%&\lvert g_1(\mz,\mz')\rvert\le 2\sup_{\mz_1,\mz_2\in\overline\bS_p}\lVert \mz_1-\mz_2\rVert\le 4,\\
%\text{and}~~~
%&\lvert g_1(\mz,\mz')-g_1(\mz,\mz'')\rvert=\Big\lvert \lVert\mz-\mz'\rVert-\lVert\mz-\mz''\rVert - \E\Big[\lVert\mZ-\mz'\rVert-\lVert\mZ-\mz''\rVert\Big]\Big\rvert\le 2\lVert\mz'-\mz''\rVert.
%\end{align*}
We now verify Assumption~(2).
For each fixed $f\in C(\Omega_i)$ and any fixed $\mz$, the product $g_i(\mz,\cdot)f(\cdot)$ yields a bounded and continuous function, and it follows from the portmanteau lemma \citep[Lemma~2.2]{MR1652247} that  
$\E (g_{i}(\mz,\mZ^{(n)}_{i})f(\mZ^{(n)}_{i}))\to(\sA f)(\mz)$ as $n\to\infty$. 
Since $f$ is continuous, we have $\lVert f\rVert_{\infty}<\infty$. 
We also have
\begin{align*}
      \lvert(\sA_n f)(\mz) 
      -\E (g_{i}(\mz,\mZ^{(n)}_{i})f(\mZ^{(n)}_{i}))\rvert
%  =\;&\lvert\E g^{(n)}_{i}(\mz,\mZ^{(n)}_{i})f(\mZ^{(n)}_{i})
%      -\E g_{i}(\mz,\mZ^{(n)}_{i})f(\mZ^{(n)}_{i})\rvert\\
\le\;&\lVert g^{(n)}_{i}-g_{i}\rVert_\infty\cdot\lVert f\rVert_\infty\to 0,\yestag
\end{align*}
where the last step uses Assumption~\ref{asm:unf}; hence $(\sA_n f)(\mz) \to (\sA f)(\mz)$.
Now, Assumption~(2) holds by Theorem~7.9 and Exercise~7.16 in
\citet{MR0385023} and the fact that the family of functions $\{\sA_n
f:n\in\Z_+\}$ is equicontinuous for each fixed $f\in C(\Omega_i)$,
which can be shown via the following argument. 
Given any small $\epsilon>0$, there exists $\delta>0$ such that $\lVert\mz-\mz'\rVert<\delta$ implies by Assumption~\ref{asm:cnt2} that
%since $g^{(n)}_{i}$ is (uniformly) equicontinuous on $\Omega_i$ by Heine--Cantor theorem \citep[Theorem~4.19]{MR0385023}.
$\lvert g_i^{(n)}(\mz,\mz'')-g_i^{(n)}(\mz',\mz'')\rvert<\epsilon/\lVert f\rVert_{\infty}$ for all $\mz''\in\Omega_i$, where $\lVert f\rVert_{\infty}<\infty$, and thus implies
\begin{align*}
      \lvert(\sA_n f)(\mz)-(\sA_n f)(\mz')\rvert
%  =\;&\lvert\E g^{(n)}_{i}(\mz,\mZ^{(n)}_{i})f(\mZ^{(n)}_{i})-\E g^{(n)}_{i}(\mz',\mZ^{(n)}_{i})f(\mZ^{(n)}_{i})\rvert\\
\le\;&\E\lvert g^{(n)}_{i}(\mz,\mZ^{(n)}_{i})-g^{(n)}_{i}(\mz',\mZ^{(n)}_{i})\rvert
      \cdot\lVert f\rVert_{\infty}
<     \epsilon.\yestag\label{eq:equic}
\end{align*}
For Assumption~(3), observe that the set $\cB$ is bounded and equicontinuous by \eqref{eq:equic}, and thus has compact closure by the Arzel{\`a}--Ascoli theorem %\citep[p.~276]{MR0098966}. 
\citep[Theorem~1.5.3]{MR3364494}.

To prove assertion (a) of the present lemma, we may apply
Theorems~2~and~3 in \citet{MR0220105} to obtain that for any fixed $k$, $\lambda^{(n)}_{i,k}\to\lambda_{i,k}$ as $n\to\infty$. 
%This proves the second term in \eqref{eq:analog-serfling} tends to $0$ for any fixed $K$.  
%
%\fbox{this sentence signals the proof of (A.15). But where are the similar signals for (A.10) and (A.11)?}.
%
%\fbox{suggestion: again, create separate lemmas stated as "the second part of (A.15) is true".}
%

The proof of (b) is separated into two parts. 
In the first part, we show that for each $i=1,2$ and any fixed $k$,
the $e^{(n)}_{i,k}$ are uniformly upper bounded for all sufficiently large $n$. 
Applying Theorem~4 in \citet{MR0220105} yields that, 
for any small $\epsilon>0$, 
there exists a sufficiently large $N$ such that 
for each $n\ge N$,
there exists a (not necessarily unique) eigenfunction $\widetilde e_{i,k}$ %depending on $n$ 
with
$\E (g_i(\mz,\mZ_i) \widetilde e_{i,k}(\mZ_i))=\lambda_{i,k} \widetilde e_{i,k}(\mz)$, 
$\E (\widetilde e_{i,k}(\mZ_i)^2)=1$, 
and
\[\Big\lVert \frac{e^{(n)}_{i,k}}{\lVert e^{(n)}_{i,k}\rVert_{\infty}}
-\frac{\widetilde e_{i,k}}{\lVert \widetilde e_{i,k}\rVert_{\infty}} \Big\rVert_{\infty}<\epsilon.\]
Invoking Properties~\ref{asm:sym1}--\ref{asm:nnd1},
Theorem~3.a.1 in \citet{MR889455} 
guarantees that 
%the eigenfunction $\widetilde e_{i,k}$ is continuous and 
there exists an absolute constant $C_1$ such that
$\lVert \widetilde e_{i,k}\rVert_{\infty}<C_1$ for all $k\in\Z_+$,
and therefore
\[\frac{\lvert e^{(n)}_{i,k}\rvert}{\lVert e^{(n)}_{i,k}\rVert_{\infty}}
\ge\Big(\frac{\lvert \widetilde e_{i,k}\rvert}{\lVert \widetilde e_{i,k}\rVert_{\infty}}-\epsilon\Big)_{+}
\ge\Big(\frac{\lvert \widetilde e_{i,k}\rvert}{C_1}-\epsilon\Big)_{+}.\]
This together with $\sum_{j=1}^{n}\{e^{(n)}_{i,k}(\mz^{(n)}_{i;j})\}^2/n=1$ implies that 
\begin{equation}\label{eq:band}
\lVert e^{(n)}_{i,k}\rVert_{\infty}^2 \le \Big[\frac{1}{n}\sum_{j=1}^{n}
\Big(\frac{\lvert \widetilde e_{i,k}(\mz^{(n)}_{i;j})\rvert}{C_1}-\epsilon\Big)_{+}^2\Big]^{-1}.
\end{equation}
In order to prove that the $e^{(n)}_{i,k}$ are uniformly upper bounded for 
any fixed $k\in\Z_+$ and all $n$ large enough, 
%$\sup_n\lVert e^{(n)}_{i,k}\rVert_{\infty}<\infty$ for any fixed $k\in\Z_+$, 
%(and hence to bound \eqref{eq:mltbe}), 
it suffices to control the right-hand side of \eqref{eq:band}.
Consider an orthonormal basis associated with eigenvalue $\lambda_{i,k}$:
$\{e_{i,k_1},\dots,e_{i,k_\ell}\}$, 
where $\ell$ is finite since 
\[\ell\lambda_{i,k}\le \sum_{k'=1}^{\infty}\lambda_{i,k'}=\E g_i(\mZ_i,\mZ_i)<\infty,\]
%where the second last step is 
by Properties~\ref{asm:sym1}--\ref{asm:nnd1} and Mercer's theorem \citep[Theorem~3.11.9(b)]{MR3364494}. 
Then $\widetilde e_{i,k}$ can be represented by 
\[\widetilde e_{i,k}=\sum_{v=1}^{\ell} \alpha_{v}e_{i,k_v},~~~\text{where}~\sum_{v=1}^{\ell} \alpha_{v}^2=1.\yestag\label{eq:Cauchy}\]
First, notice that there exists $N_1\ge N$ such that for all $n\ge N_1$,
\[
\Big\lvert\frac{1}{n}\sum_{j=1}^{n} e_{i,k_{v}}(\mz^{(n)}_{i;j})e_{i,k_{v'}}(\mz^{(n)}_{i;j})
-\ind(v=v')\Big\rvert<\epsilon,~~~\text{for all}~v, v'\in[\ell],\yestag\label{eq:unifconver}
\]
using the continuity of the eigenfunctions $e_{i,k_v}$, which holds by Property~\ref{asm:cnt1} and 
Corollary~2 in \citet[p.~34]{MR1864085}, 
%Mercer's theorem \citep[Theorem~3.11.9(c)]{MR3364494}, 
together with the portmanteau lemma \citep[Lemma~2.2]{MR1652247}. 
Then combining \eqref{eq:Cauchy} and \eqref{eq:unifconver}, 
it holds that for $n\ge N_1$, 
\begin{align*}
\;&\frac{1}{n}\sum_{j=1}^{n}
\Big(\frac{\lvert \widetilde e_{i,k}(\mz^{(n)}_{i;j})\rvert}{C_1}-\epsilon\Big)_{+}^2
\ge\frac{1}{C_1^2}\cdot\frac{1}{n}\sum_{j=1}^{n}\{ \widetilde e_{i,k}(\mz^{(n)}_{i;j})\}^2
-\frac{2\epsilon}{C_1}\cdot\frac{1}{n}\sum_{j=1}^{n}\lvert \widetilde e_{i,k}(\mz^{(n)}_{i;j})\rvert\\
=\;&\frac{1}{C_1^2}\cdot\sum_{v=1}^{\ell}\frac{\alpha_{v}^2}{n}\sum_{j=1}^{n}\{ e_{i,k_{v}}(\mz^{(n)}_{i;j})\}^2
+\frac{2}{C_1^2}\cdot\sum_{v<v'}\frac{\alpha_{v}\alpha_{v'}}{n}\sum_{j=1}^{n}e_{i,k_{v}}(\mz^{(n)}_{i;j})e_{i,k_{v'}}(\mz^{(n)}_{i;j})
-\frac{2\epsilon}{C_1}\cdot\frac{1}{n}\sum_{j=1}^{n}\lvert \widetilde e_{i,k}(\mz^{(n)}_{i;j})\rvert\\
\ge\;&\frac{1}{C_1^2}\cdot\sum_{v=1}^{\ell}\alpha_{v}^2(1-\epsilon)
-\frac{2}{C_1^2}\cdot\sum_{v<v'}\lvert\alpha_{v}\alpha_{v'}\rvert\epsilon
-2\epsilon
=\frac{1}{C_1^2}-\frac{\epsilon}{C_1^2}\Big(\sum_{v=1}^{\ell} \lvert\alpha_{v}\rvert\Big)^2-2\epsilon
\ge \frac{1-\epsilon\ell}{C_1^2}-2\epsilon.\yestag
\end{align*}
This completes the first part by taking sufficiently small $\epsilon$.

For the remaining part, we are to show that $\sup_{n< N_1}\lVert
e^{(n)}_{i,k}\rVert_{\infty}<\infty$.  
Using Assumption~\ref{asm:cnt1}, and once again, %\fbox{check!!!}
Corollary~2 in \citet[p.~34]{MR1864085}, 
%Mercer's theorem \citep[Theorem~3.11.9(c)]{MR3364494}, 
the eigenfunctions $e^{(n)}_{i,k}$, $n< N_1$, are seen to be
continuous.  The remaining fact thus holds because $\Omega_i$ is compact 
and $N_1$ is finite.  With this last step, 
the proof of the lemma is completed.
\end{proof}

%\begin{lemma}\label{lem:aux2}
%For each $i=1,2$, we have
%$\E g^{(n)}_{i}(\mZ^{(n)}_{i},\mZ^{(n)}_{i})\to\E g_{i}(\mZ_{i},\mZ_{i})$.
%\end{lemma}
%\begin{proof}
%We have $\E g_{i}(\mZ^{(n)}_{i},\mZ^{(n)}_{i})\to \E g_{i}(\mZ_{i},\mZ_{i})$ 
%by the portmanteau lemma \citep[Lemma~2.2]{MR1652247}. 
%The result follows by noticing
%\[\E \lvert g^{(n)}_{i}(\mZ^{(n)}_{i},\mZ^{(n)}_{i})- g_{i}(\mZ^{(n)}_{i},\mZ^{(n)}_{i})\rvert \le \lVert g^{(n)}_{i}-g_{i}\rvert.\]
%\end{proof}

\subsubsection{Proof of Theorem \ref{cor:general}}\label{subsec:general}

\begin{proof}[Proof of Theorem \ref{cor:general}]
We consider the Hoeffding decomposition with respect to the product measure $\Pr_{\mZ_1^{(n)}}\times \Pr_{\mZ_2^{(n)}}$:
\begin{equation*}%\label{eqn:Hdec}
 \widehat \Pi^{(n)}
=\sum_{\ell=2}^{m}\underbrace{\mbinom{m}{\ell}\mbinom{n}{\ell}^{-1}
  \sum_{1\le i_1<\cdots< i_\ell\le n}\widetilde h_{\ell}\Big(
  (\mz^{(n)}_{1;i_1},\mz^{(n)}_{2;\pi_{i_1}}),\ldots,
  (\mz^{(n)}_{1;i_m},\mz^{(n)}_{2;\pi_{i_m}});
  \Pr_{\mZ_1^{(n)}}\times \Pr_{\mZ_2^{(n)}}\Big)}_{\displaystyle \widetilde D^{(n)}_{\ell}}.
\end{equation*}
We have proven in Theorem~\ref{thm:general} that 
\[\mbinom{m}{2}^{-1}n\widetilde D^{(n)}_{2}\stackrel{\sf d}{\longrightarrow}
\sum_{k_1,k_2=1}^{\infty}\lambda_{1,k_1}\lambda_{2,k_2}(\xi_{k_1,k_2}^2-1)\]
as $n\to\infty$. 
In order to prove that $n\widehat \Pi^{(n)}$ and $n\widetilde D^{(n)}_{2}$ have the same limiting distribution, 
we only need to show that $n\widetilde D^{(n)}_{\ell}\stackrel{\sf p}{\longrightarrow} 0$
for $\ell=3,\dots,m$ and apply Slutsky's theorem \citep[Theorem~2.8]{MR1652247}. 
To this end, it suffices to establish that $\E[(n\widetilde D^{(n)}_{\ell})^2]=O(n^{-1})$ 
for $\ell=3,\dots,m$. 

We start from the scenario $\ell=3$ and proceed in two steps, in which
we show that  (i) $\E[n\widetilde D^{(n)}_{3}]=O(n^{-1})$, and (ii) $\Var(n\widetilde D^{(n)}_{3})=O(n^{-1})$.
By symmetry,
\[\widetilde D^{(n)}_{3}=\mbinom{m}{3} (n)_3^{-1}\sum_{[i_1,i_2,i_3]\in I^{n}_{3}}
\widetilde h_{3}\Big(
(\mz^{(n)}_{1;i_1},\mz^{(n)}_{2;\pi_{i_1}}),
(\mz^{(n)}_{1;i_2},\mz^{(n)}_{2;\pi_{i_2}}),
(\mz^{(n)}_{1;i_3},\mz^{(n)}_{2;\pi_{i_3}});
  \Pr_{\mZ_1^{(n)}}\times \Pr_{\mZ_2^{(n)}}\Big).\yestag\label{eq:n3protect}\]
One readily verifies $\lVert\widetilde h_{3}\rVert_\infty\le 2^3\lVert h\rVert_\infty$.
To simplify notation, let %$\Delta^{(n)}_{3}(i_1,j_1;i_2,j_2;i_3,j_3)$ denote 
\begin{equation}\label{eq:defh3}
\Delta^{(n)}_{3}(i_1,j_1;i_2,j_2;i_3,j_3):=\widetilde h_{3}\Big(
(\mz^{(n)}_{1;i_1},\mz^{(n)}_{2;j_1}),
(\mz^{(n)}_{1;i_2},\mz^{(n)}_{2;j_2}),
(\mz^{(n)}_{1;i_3},\mz^{(n)}_{2;j_3});
\Pr_{\mZ_1^{(n)}}\times \Pr_{\mZ_2^{(n)}}\Big),
\end{equation}
and adopt the convention that replacing an index of $\Delta^{(n)}_{3}$
by a ``$\bullet$'' means averaging over this index.  In particular,
%Moreover, we define averages of $\Delta^{(n)}_{3}(i_1,j_1;i_2,j_2;i_3,j_3)$, in which replacing an index at any position by a ``$\bullet$''
%sign denotes average over all indices at this position:
\begin{align*}
\Delta^{(n)}_{3}(\bullet,j_1;i_2,j_2;i_3,j_3)
&:=\frac1n\sum_{i_1=1}^n\Delta^{(n)}_{3}(i_1,j_1;i_2,j_2;i_3,j_3),\\
\Delta^{(n)}_{3}(\bullet,\bullet;i_2,j_2;i_3,j_3)
&:=\frac1{n^2}\sum_{i_1=1}^n\sum_{j_1=1}^n\Delta^{(n)}_{3}(i_1,j_1;i_2,j_2;i_3,j_3),\\
  %\cdots~~~
  \Delta^{(n)}_{3}(\bullet,\bullet;\bullet,\bullet;\bullet,\bullet)
&:=\frac1{n^6}\sum_{i_1=1}^n\cdots\sum_{j_3=1}^n\Delta^{(n)}_{3}(i_1,j_1;i_2,j_2;i_3,j_3),\yestag\label{eq:avg}
\end{align*}
and other averages are defined similarly. We obtain using the definition \eqref{eq:defh3} that
\begin{align*}
               \Delta^{(n)}_{3}[\bullet,\bullet;i_2,j_2;i_3,j_3]&=&
    \mkern-70mu\Delta^{(n)}_{3}[i_1,j_1;\bullet,\bullet;i_3,j_3]&=&
    \mkern-70mu\Delta^{(n)}_{3}[i_1,j_1;i_2,j_2;\bullet,\bullet]&=0, \\
               \Delta^{(n)}_{3}[\bullet,\bullet;\bullet,\bullet;i_3,j_3]&=&
    \mkern-70mu\Delta^{(n)}_{3}[\bullet,\bullet;i_2,j_2;\bullet,\bullet]&=&
    \mkern-70mu\Delta^{(n)}_{3}[i_1,j_1;\bullet,\bullet;\bullet,\bullet]&=0, \\
                                                                        & &
    \mkern-70mu                                                         & &
    \mkern-70mu\Delta^{(n)}_{3}[\bullet,\bullet;\bullet,\bullet;\bullet,\bullet]&=0.
\yestag\label{eq:vanish}
\end{align*}

{\bf Step I. }
We show that $\E[n\widetilde D^{(n)}_{3}]=O(n^{-1})$. 
In view of \eqref{eq:n3protect}, we have
\[\mbinom{m}{3}^{-1}(n)_3\cdot\widetilde D^{(n)}_{3}= \sum_{[i_1,i_2,i_3]\in I^{n}_{3}}
\Delta^{(n)}_{3}(i_1,\pi_{i_1};i_2,\pi_{i_2};i_3,\pi_{i_3}).\]
Applying \eqref{eq:vanish}, direct calculation yields
\begin{align*}
 \;&\E\sum_{[i_1,i_2,i_3]\in I^{n}_{3}}\Delta^{(n)}_{3}(i_1,\pi_{i_1};i_2,\pi_{i_2};i_3,\pi_{i_3})\\
=\;&\frac{1}{(n)_3}\sum_{\substack{[i_1,i_2,i_3]\in I^{n}_{3},[j_1,j_2,j_3]\in I^{n}_{3}}}\Delta^{(n)}_{3}(i_1,j_1;i_2,j_2;i_3,j_3)\\
=\;&\frac{1}{(n)_3}\sum_{\substack{[i_1,i_2]\in I^{n}_{2},[j_1,j_2]\in I^{n}_{2}}}\Big\{
             -n\Delta^{(n)}_{3}(i_1,j_1;i_2,j_2;i_1,\bullet)
             -n\Delta^{(n)}_{3}(i_1,j_1;i_2,j_2;\bullet,j_1)\\
& \mkern180mu-n\Delta^{(n)}_{3}(i_1,j_1;i_2,j_2;i_2,\bullet)
             -n\Delta^{(n)}_{3}(i_1,j_1;i_2,j_2;\bullet,j_2)\\
&  \mkern180mu+\Delta^{(n)}_{3}(i_1,j_1;i_2,j_2;i_1,j_1)
              +\Delta^{(n)}_{3}(i_1,j_1;i_2,j_2;i_1,j_2)\\
&  \mkern180mu+\Delta^{(n)}_{3}(i_1,j_1;i_2,j_2;i_2,j_1)
              +\Delta^{(n)}_{3}(i_1,j_1;i_2,j_2;i_2,j_2)\Big\}\\
=\;&\frac{1}{(n-1)_2}\sum_{i_1\in\zahl{n},j_1\in\zahl{n}}
        \Big\{n\Delta^{(n)}_{3}(i_1,j_1;i_1,\bullet;i_1,\bullet)
             +n\Delta^{(n)}_{3}(i_1,j_1;\bullet,j_1;i_1,\bullet)
              -\Delta^{(n)}_{3}(i_1,j_1;i_1,j_1;i_1,\bullet)\\
& \mkern150mu+n\Delta^{(n)}_{3}(i_1,j_1;i_1,\bullet;\bullet,j_1)
             +n\Delta^{(n)}_{3}(i_1,j_1;\bullet,j_1;\bullet,j_1)
              -\Delta^{(n)}_{3}(i_1,j_1;i_1,j_1;\bullet,j_1)\Big\}\\
& +\frac{1}{(n-1)_2}\sum_{i_2\in\zahl{n},j_2\in\zahl{n}}
        \Big\{n\Delta^{(n)}_{3}(i_2,\bullet;i_2,j_2;i_2,\bullet)
             +n\Delta^{(n)}_{3}(\bullet,j_2;i_2,j_2;i_2,\bullet)
              -\Delta^{(n)}_{3}(i_2,j_2;i_2,j_2;i_2,\bullet)\\
& \mkern152mu+n\Delta^{(n)}_{3}(i_2,\bullet;i_2,j_2;\bullet,j_2)
             +n\Delta^{(n)}_{3}(\bullet,j_2;i_2,j_2;\bullet,j_2)
              -\Delta^{(n)}_{3}(i_2,j_2;i_2,j_2;\bullet,j_2)\Big\}\\
&+\frac{1}{(n)_3}\sum_{\substack{[i_1,i_2]\in I^{n}_{2},[j_1,j_2]\in I^{n}_{2}}}
         \Big\{\Delta^{(n)}_{3}(i_1,j_1;i_2,j_2;i_1,j_1)
              +\Delta^{(n)}_{3}(i_1,j_1;i_2,j_2;i_1,j_2)\\
&  \mkern180mu+\Delta^{(n)}_{3}(i_1,j_1;i_2,j_2;i_2,j_1)
              +\Delta^{(n)}_{3}(i_1,j_1;i_2,j_2;i_2,j_2)\Big\}
\;=\; O(n),\yestag
\end{align*}
where the implicit constant depends only on $\lVert h\rVert_\infty$.
This completes Step I.

{\bf Step II. }
We prove that $\Var(n\widetilde D^{(n)}_{3})=O(n^{-1})$. 
Notice that
\[\sum_{[i_1,i_2,i_3]\in I^{n}_{3}}\Delta^{(n)}_{3}(i_1,\pi_{i_1};i_2,\pi_{i_2};i_3,\pi_{i_3})=A_3-A_2-A_1,\]
where
\begin{align*}
A_3&:=\sum_{i_1=1}^{n}\sum_{i_2=1}^{n}\sum_{i_3=1}^{n}
               \Delta^{(n)}_{3}(i_1,\pi_{i_1};i_2,\pi_{i_2};i_3,\pi_{i_3}),\\
A_2&:=\sum_{[i_1,i_2]\in I^{n}_{2}}\Big\{
               \Delta^{(n)}_{3}(i_1,\pi_{i_1};i_1,\pi_{i_1};i_2,\pi_{i_2})
              +\Delta^{(n)}_{3}(i_1,\pi_{i_1};i_2,\pi_{i_2};i_1,\pi_{i_1})
              +\Delta^{(n)}_{3}(i_2,\pi_{i_2};i_1,\pi_{i_1};i_1,\pi_{i_1})\Big\},\\
A_1&:=\sum_{i_1=1}^{n}
               \Delta^{(n)}_{3}(i_1,\pi_{i_1};i_1,\pi_{i_1};i_1,\pi_{i_1}).
\end{align*}
We set
\begin{align*}
  \;&\widetilde\Delta^{(n)}_{3}(i_1,j_1;i_2,j_2;i_3,j_3)\\
:=\;&          \Delta^{(n)}_{3}(i_1,j_1;i_2,j_2;i_3,j_3)
              -\Delta^{(n)}_{3}(\bullet,j_1;i_2,j_2;i_3,j_3)
              -\cdots
              -\Delta^{(n)}_{3}(i_1,j_1;i_2,j_2;i_3,\bullet)\\
  \;&         +\Delta^{(n)}_{3}(\bullet,\bullet;i_2,j_2;i_3,j_3)
              +\Delta^{(n)}_{3}(\bullet,j_1;\bullet,j_2;i_3,j_3)
              +\cdots
              +\Delta^{(n)}_{3}(i_1,j_1;i_2,j_2;\bullet,\bullet)\\
  \;&         -\cdots
              +\Delta^{(n)}_{3}(\bullet,\bullet;\bullet,\bullet;\bullet,\bullet).\yestag\label{eq:daSilva}
\end{align*}
Combining \eqref{eq:vanish} and \eqref{eq:daSilva}, we deduce that
\[A_3=\sum_{i_1=1}^{n}\sum_{i_2=1}^{n}\sum_{i_3=1}^{n}
     \widetilde\Delta^{(n)}_{3}(i_1,\pi_{i_1};i_2,\pi_{i_2};i_3,\pi_{i_3}).\]
Here, $A_3$ can be decomposed as $A_3=\widetilde A_3+\widetilde A_2+\widetilde A_1$, where
\begin{align*}
\widetilde A_3&:=\sum_{[i_1,i_2,i_3]\in I^{n}_{3}}
     \widetilde\Delta^{(n)}_{3}(i_1,\pi_{i_1};i_2,\pi_{i_2};i_3,\pi_{i_3}),\\
\widetilde A_2&:=\sum_{[i_1,i_2]\in I^{n}_{2}}\Big\{
     \widetilde\Delta^{(n)}_{3}(i_1,\pi_{i_1};i_1,\pi_{i_1};i_2,\pi_{i_2})
    +\widetilde\Delta^{(n)}_{3}(i_1,\pi_{i_1};i_2,\pi_{i_2};i_1,\pi_{i_1})
    +\widetilde\Delta^{(n)}_{3}(i_2,\pi_{i_2};i_1,\pi_{i_1};i_1,\pi_{i_1})\Big\},\\
\widetilde A_1&:=\sum_{i_1=1}^{n}
     \widetilde\Delta^{(n)}_{3}(i_1,\pi_{i_1};i_1,\pi_{i_1};i_1,\pi_{i_1}).
\end{align*}
Hence,
\[\sum_{[i_1,i_2,i_3]\in I^{n}_{3}}
               \Delta^{(n)}_{3}(i_1,\pi_{i_1};i_2,\pi_{i_2};i_3,\pi_{i_3})
=\widetilde A_3+(\widetilde A_2-A_2)+(\widetilde A_1-A_1).\]
Using 
$\widetilde\Delta^{(n)}_{3}(\bullet,j_1;i_2,j_2;i_3,j_3)
%=\cdots=\widetilde\Delta^{(n)}_{3}(i_1,j_1;i_2,j_2;i_3,\bullet)
=\cdots=\widetilde\Delta^{(n)}_{3}(\bullet,\bullet;\bullet,\bullet;\bullet,\bullet)=0,$
  a straightforward calculation confirms that $\Var(\widetilde A_3)=O(n^3)$. First, for $i_1,i_2,i_3,i_1',i_2',i_3'$ distinct, we have
\begin{align*}
\;&\E[
     \widetilde\Delta^{(n)}_{3}(i_1,\pi_{i_1};i_2,\pi_{i_2};i_3,\pi_{i_3})
     \widetilde\Delta^{(n)}_{3}(i_1',\pi_{i_1'};i_2',\pi_{i_2'};i_3',\pi_{i_3'})]\\
=\;&\frac{1}{(n)_6}\sum_{[j_1,j_2,j_3,j_1',j_2',j_3']\in I^{n}_{6}}
     \widetilde\Delta^{(n)}_{3}(i_1,j_1;i_2,j_2;i_3,j_3)
     \widetilde\Delta^{(n)}_{3}(i_1',j_1';i_2',j_2';i_3',j_3')\\
=\;&-\frac{1}{(n)_6}\sum_{[j_1,j_2,j_3,j_1',j_2']\in I^{n}_{5}}
                  \;\sum_{j_3'\in\{j_1,j_2,j_3,j_1',j_2'\}}
     \widetilde\Delta^{(n)}_{3}(i_1,j_1;i_2,j_2;i_3,j_3)
     \widetilde\Delta^{(n)}_{3}(i_1',j_1';i_2',j_2';i_3',{j_3'})
\end{align*}
where
\begin{align*}
\;&-\frac{1}{(n)_6}\sum_{[j_1,j_2,j_3,j_1',j_2']\in I^{n}_{5}}
     \widetilde\Delta^{(n)}_{3}(i_1,j_1;i_2,j_2;i_3,j_3)
     \widetilde\Delta^{(n)}_{3}(i_1',j_1';i_2',j_2';i_3',{j_1})\\
=\;&\frac{1}{(n)_6}\sum_{[j_1,j_2,j_3,j_1']\in I^{n}_{4}}
                   \;\sum_{j_2'\in\{j_1,j_2,j_3,j_1'\}}
     \widetilde\Delta^{(n)}_{3}(i_1,j_1;i_2,j_2;i_3,j_3)
     \widetilde\Delta^{(n)}_{3}(i_1',j_1';i_2',{j_2'};i_3',{j_1}),\yestag\label{eq:inter1}\\
\;&-\frac{1}{(n)_6}\sum_{[j_1,j_2,j_3,j_1',j_2']\in I^{n}_{5}}
     \widetilde\Delta^{(n)}_{3}(i_1,j_1;i_2,j_2;i_3,j_3)
     \widetilde\Delta^{(n)}_{3}(i_1',j_1';i_2',j_2';i_3',{j_1'})\\
=\;&\frac{1}{(n)_6}\sum_{[j_1,j_2,j_3,j_1']\in I^{n}_{4}}
                   \;\sum_{j_2'\in\{j_1,j_2,j_3,j_1'\}}
     \widetilde\Delta^{(n)}_{3}(i_1,j_1;i_2,j_2;i_3,j_3)
     \widetilde\Delta^{(n)}_{3}(i_1',j_1';i_2',{j_2'};i_3',{j_1'}),\yestag\label{eq:inter2}
\end{align*}
and other summands can be rewritten similarly. 
Moreover, we have in \eqref{eq:inter1} that
\begin{align*}
\;&\frac{1}{(n)_6}\sum_{[j_1,j_2,j_3,j_1']\in I^{n}_{4}}
     \widetilde\Delta^{(n)}_{3}(i_1,j_1;i_2,j_2;i_3,j_3)
     \widetilde\Delta^{(n)}_{3}(i_1',j_1';i_2',{j_1};i_3',{j_1})\\
=\;&-\frac{1}{(n)_6}\sum_{[j_1,j_2,j_1']\in I^{n}_{3}}\;\sum_{j_3\in\{j_1,j_2,j_1'\}}
     \widetilde\Delta^{(n)}_{3}(i_1,j_1;i_2,j_2;i_3,{j_3})
     \widetilde\Delta^{(n)}_{3}(i_1',j_1';i_2',{j_1};i_3',{j_1}),\\
\;&\frac{1}{(n)_6}\sum_{[j_1,j_2,j_3,j_1']\in I^{n}_{4}}
     \widetilde\Delta^{(n)}_{3}(i_1,j_1;i_2,j_2;i_3,j_3)
     \widetilde\Delta^{(n)}_{3}(i_1',j_1';i_2',{j_1'};i_3',{j_1})\\
=\;&-\frac{1}{(n)_6}\sum_{[j_1,j_2,j_1']\in I^{n}_{3}}\;\sum_{j_3\in\{j_1,j_2,j_1'\}}
     \widetilde\Delta^{(n)}_{3}(i_1,j_1;i_2,j_2;i_3,{j_3})
     \widetilde\Delta^{(n)}_{3}(i_1',j_1';i_2',{j_1'};i_3',{j_1}),
\end{align*}
and similar equations for all the other summands.   In \eqref{eq:inter2},
\begin{align*}
\;&\frac{1}{(n)_6}\sum_{[j_1,j_2,j_3,j_1']\in I^{n}_{4}}
     \widetilde\Delta^{(n)}_{3}(i_1,j_1;i_2,j_2;i_3,j_3)
     \widetilde\Delta^{(n)}_{3}(i_1',j_1';i_2',{j_1};i_3',{j_1'})\\
=\;&-\frac{1}{(n)_6}\sum_{[j_1,j_2,j_1']\in I^{n}_{3}}\;\sum_{j_3\in\{j_1,j_2,j_1'\}}
     \widetilde\Delta^{(n)}_{3}(i_1,j_1;i_2,j_2;i_3,{j_3})
     \widetilde\Delta^{(n)}_{3}(i_1',j_1';i_2',{j_1};i_3',{j_1'}),\\
\;&\frac{1}{(n)_6}\sum_{[j_1,j_2,j_3,j_1']\in I^{n}_{4}}
     \widetilde\Delta^{(n)}_{3}(i_1,j_1;i_2,j_2;i_3,j_3)
     \widetilde\Delta^{(n)}_{3}(i_1',j_1';i_2',{j_1'};i_3',{j_1'})\\
=\;&-\frac{1}{(n)_6}\sum_{[j_1,j_2,j_1']\in I^{n}_{3}}\;\sum_{j_3\in\{j_1,j_2,j_1'\}}
     \widetilde\Delta^{(n)}_{3}(i_1,j_1;i_2,j_2;i_3,{j_3})
     \widetilde\Delta^{(n)}_{3}(i_1',j_1';i_2',{j_1'};i_3',{j_1'}),
\end{align*}
and similar equations for all the other summands. 
It follows that
\[\sum_{[i_1,i_2,i_3,i_1',i_2',i_3']\in I^{n}_{6}}\E[
     \widetilde\Delta^{(n)}_{3}(i_1,\pi_{i_1};i_2,\pi_{i_2};i_3,\pi_{i_3})
     \widetilde\Delta^{(n)}_{3}(i_1',\pi_{i_1'};i_2',\pi_{i_2'};i_3',\pi_{i_3'})]=O(n^3).\]
Similar calculations for the cases when the pairs $[i_1,i_2,i_3]$ and $[i_1',i_2',i_3']$ have one, two, or three indices in common, give a total contribution of at most $O(n^3)$.
Adding these together shows that $\Var(\widetilde A_3)=O(n^3)$.
This together with
$\Var(\widetilde A_2-A_2)=O(n^3)$ (similar to \citealp[p.~2212]{MR1474091};
\citealp[Lemma~3.1]{MR2205339}) 
and $\Var(\widetilde A_1-A_1)=O(n)$ \citep[Theorem~2]{MR0044058} implies that 
$\Var(n\widetilde D^{(n)}_{3})=O(n^{-1})$. 

Taken together the two steps we carried out prove that $\E[(n\widetilde D^{(n)}_{3})^2]=O(n^{-1})$. 
The proofs for $\E[(n\widetilde D^{(n)}_{\ell})^2]=O(n^{-1})$, $\ell=4,\dots,m$, are very similar and hence omitted.
\end{proof}

\subsection{Proofs for Section \ref{sec:computation} of the main paper}

\subsubsection{Proof of Theorem \ref{thm:matching}}\label{sec:proof-thm-matching}

\begin{proof}[Proof of Theorem \ref{thm:matching}]
Introducing the dummy variables $x_{ij}$ with
\[
x_{ij}=\begin{cases} 1 & \text{if edge $(\ms_i,\mt_j)$ is in the
    % optimal
    matching,}\\
0 & \text{otherwise,}
\end{cases}
\]
the LSAP can be formulated as a linear program:
\begin{align*}
&\min_{x_{ij}}     &&\mkern-20mu\sum_{i,j}c_{ij}x_{ij}\\
&\text{subject to} &&\mkern-20mu\sum_{j=1}^{n} x_{ij}=1,~\text{for}~i\in\zahl{n};
~\sum_{i=1}^{n} x_{ij}=1,~\text{for}~j\in\zahl{n};
~x_{ij}\in\{0,1\},~\text{for}~i,j\in\zahl{n}.
\end{align*}
Then an edge $(\ms_i,\mt_j)$ is in the optimal matching if and only if
the solution to the linear program has $x_{ij}=1$.  The dual linear program is
\begin{align*}
&\max_{\alpha_i,\beta_j} &&\mkern-120mu\sum_{i}\alpha_{i}+\sum_{j}\beta_{j}\\
&\text{subject to}       &&\mkern-120mu\alpha_{i}+\beta_{j}\le c_{ij},~\text{for}~i, j\in\zahl{n};
~\alpha_{i},\beta_{j}~\text{unconstrained}.
\end{align*}
The sufficient and necessary condition for an optimal solution to the
LSAP is
%\[x_{ij}(c_{ij}-\alpha_{i}-\beta_{j})=0,~\text{for}~i,j\in\zahl{n}.\]
\begin{align*}
&\alpha_{i}+\beta_{j}\le c_{ij},  & &\mkern-150mu\text{for}~i,j\in\zahl{n},  \\
&\alpha_{i}+\beta_{j}= c_{ij},    & &\mkern-150mu\text{for}~x_{ij}=1.
\end{align*}

We introduce a few more terms that are convenient for our description.
A \emph{matching} is a subset of edges whose vertices are disjoint.
A matching $M$ is \emph{$1$-feasible} if the dual variables satisfy that
\begin{align*}
&\alpha_{i}+\beta_{j}\le c_{ij}+1, & &\mkern-150mu\text{for}~i,j\in\zahl{n},  \\
&\alpha_{i}+\beta_{j}= c_{ij},     & &\mkern-150mu\text{for}~(\ms_i,\mt_j)\in M.
\end{align*}
A \emph{$1$-optimal matching} is a $1$-feasible perfect matching.
An edge $(\ms_i,\mt_j)$ is called \emph{admissible} with regard to a matching $M$ if $\alpha_i+\beta_j=c_{ij}+\ind((\ms_i,\mt_j)\not\in M)$.
An \emph{admissible graph} is the union of a matching $M$ and the set of all admissible edges.
A vertex is called \emph{exposed} if it is not incident to any edge in the current matching.
An \emph{alternating path} is one that starts with an exposed vertex
and alternatingly traverses edges in the matching and not.
An \emph{alternating tree} is a rooted tree whose paths are alternating paths from its root.
A \emph{labelled} vertex is one that belongs to any alternating tree.
An \emph{augmenting path} is an alternating path between two exposed vertices.

For every $\ms_i\in S$, $\mt_j\in T$, let $c^*_{ij}=(n+1)c_{ij}$. It is equivalent to find the optimal matching for the weights $c^*_{ij}$ and that for the weights $c_{ij}$. Let ${b_1b_2\cdots b_k}_{(2)}$ stand for the binary representation of $c^*_{ij}$, where $k\le \lfloor\log_2((n+1)N)\rfloor+1$. 
We initialize the weights $c^{(0)}_{ij}$ and the dual variables $\alpha^{(0)}_{i},\beta^{(0)}_{j}$, $i,j\in\zahl{n}$ to zero and the matching $M$ to empty matching.
The scaling algorithm proceeds in $k$ stages. 
At the $r$-th stage, we go through match routines to find a $1$-optimal matching, 
where the weight $c^{(r)}_{ij}$ of edge $(\ms_i,\mt_j)$ has the binary representation ${b_1b_2\cdots b_r}_{(2)}$ (and thus is equal to $2c^{(r-1)}_{ij}$ or $2c^{(r-1)}_{ij}+1$),
starting from dual variables $\alpha^{(r)}_{i}:=2\alpha^{(r-1)}_{i},\beta^{(r)}_{j}:=2\beta^{(r-1)}_{j}$, $i,j\in\zahl{n}$.

The match routine computes a $1$-optimal matching in several phases, each of which consists of augmenting the matching and doing a Hungarian search. 
Let $M$ be the current matching initialized to empty matching.
We will omit the superscript index $(r)$ when there is no confusion.

{\bf Step I.}
We first obtain a maximal set $\cP$ of vertex-disjoint augmenting
paths in the admissible graph by performing a depth first search. The
depth first search marks every vertex visited; initially no vertex is
marked. We grow an augmenting path $P$ starting from an exposed vertex
$\mt_j\in T$ by searching all admissible edges and finding an edge
$(\mt_j,\ms_i)$ where $\ms_i\in S$ is not marked. If such $\ms_i$
exists, we mark $\ms_i$, add edge $(\mt_j,\ms_i)$ to $P$, and then (1)
if $\ms_i$ is also exposed, add the augmenting path $P$ to $\cP$, and
start finding the next augmenting path; (2) if $\ms_i$ is matched to
$\mt_k$ ($k\ne j$ since $\ms_i$ has not been marked until this step),
we mark $\mt_k$, add edge $(\ms_i,\mt_k)$ to $P$, and continue
searching from $\mt_k$. If there is no $\ms_i$ unmarked, we delete the
last two edges in path $P$ and (1) restart searching if $P$ is not
empty; (2) initialize a new path otherwise. We repeat these steps
until we have gone through all exposed vertices in $T$.
Then for each path $P\in\cP$, we augment the matching $M$ by replacing edges in the even step  with the ones in the odd steps, 
 and decrease dual variables $\alpha_i$ by $1$ for all $\ms_i\in A\cap P$ to maintain $1$-feasibility. If the new matching is perfect, the routine halts, otherwise we do a Hungarian search as below.

{\bf Step II.}
For each exposed vertex $\mt_j\in T$, we grow an alternating tree rooted at $\mt_j$ such that each vertex in $S\cup T$ that in this tree is reachable from the root via an alternating path consisting only of admissible edges. For a vertex in $S$ (resp. $T$) in an alternating tree, the path from the root is augmenting (resp. not augmenting). Let $LS$ (resp. $LT$) denote the set of vertices in $S$ (resp. $T$) that are labelled. At the beginning of Hungarian search, $LT$ is defined as the set of the exposed vertices in $T$ and $LS=\varnothing$. Define
\[
\delta=\min_{\ms_i\in S-LS,\,\mt_j\in LT} \Big\{c_{ij}+\ind((\ms_i,\mt_j)\not\in M)-\alpha_i-\beta_j\Big\}.
\]
Depending on whether $\delta=0$ or $\delta>0$, one of the following steps is taken:

{Case 1.} 
$\delta=0$ (find an augmenting path or add to alternating trees). Let $(\ms_i,\mt_j)$ for $\ms_i\in S-LS$ and $\mt_j\in LT$ be an admissible edge, where the existence is guaranteed by $\delta=0$. If $\ms_i$ is exposed, an augmenting path has been found 
and the Hungarian search ends. If $\ms_i$ is matched to $\mt_k$ for some $k\ne j$ (notice that $\ms_i$ cannot be matched to $\mt_j$ since $\ms_i$ is not labelled currently), we add the edges $(\mt_j,\ms_i)$ and $(\ms_i,\mt_k)$ to all the alternating trees that involve $\mt_j$,  update $LS$ and $LT$ by adding vertices $\ms_i$ and $\mt_k$ respectively, and recompute $\delta$.

{Case 2.} 
$\delta>0$ (update the dual solution). We decrease $\alpha_i$ by $\delta$ for each $\ms_i\in LS$, increase $\beta_j$ by $\delta$ for each $\mt_j\in LT$, and recompute $\delta$.

In summary, there are $O(\log(nN))$ stages. At each stage, one routine consists of  $O(\sqrt{n})$ phases, and each phase runs in $O(n^2)$ time. The overall running time is $O(n^{5/2}\log(nN))$. 
\end{proof}

\subsubsection{Proof of Theorem \ref{thm:eigvcomp}}

\begin{proof}[Proof of Theorem \ref{thm:eigvcomp}]
In order to prove $Q_{1-\alpha}^{(M)} \to Q_{1-\alpha}$ as $M_R\to\infty$ and $M_S\to\infty$, 
it suffices to show that
\[\sum_{k=1}^{(M-1)^2}\lambda_k^{(M)}(\xi_k^2-1)\stackrel{\sf d}{\longrightarrow}\sum_{k=1}^{\infty}\lambda_k(\xi_k^2-1).\]
We only need to show the convergence of moment-generating functions: 
\begin{equation}\label{eq:mgf1}
\E\Big[\exp\Big(t\sum_{k=1}^{(M-1)^2}\lambda_k^{(M)}(\xi_k^2-1)\Big)\Big]
\to\E\Big[\exp\Big(t\sum_{k=1}^{\infty}\lambda_k(\xi_k^2-1)\Big)\Big]
\end{equation}
as $M\to\infty$, for all $t\in[-r,r]$ and some $r>0$, by arguments in \citet[p.~390]{MR1324786}. Notice that \eqref{eq:mgf1} is equivalent to
\begin{equation}\label{eq:mgf2}
\prod_{k=1}^{(M-1)^2}\frac{\Big(1-2t\lambda_k^{(M)}\Big)^{-1/2}}{\exp\Big(\lambda_k^{(M)}\Big)}
\to\prod_{k=1}^{\infty}\frac{(1-2t\lambda_k)^{-1/2}}{\exp(\lambda_k)}.
\end{equation}
We have by Item~(vi) in \citet{MR3813995} that $\lambda_k>0$ and
\[\sum_{k=1}^{\infty}\lambda_k=\E\lVert\mU-\mU_{*}\rVert\cdot \E\lVert\mV-\mV_{*}\rVert<\infty,\]
where $\mU\sim U_p$, $\mV\sim U_q$, and $\mU_{*}$ and $\mV_{*}$ are independent copies of $\mU$ and $\mV$, respectively.
This implies that the right-hand side of \eqref{eq:mgf2} converges to a nonzero real number for every $t\in[-r,r]$ where $r$ is some fixed small positive number
%\citep[Theorem~1]{MR1720455} 
\citep[Theorem~15.5]{MR0924157}. 
This together with the fact that, $\lambda_k^{(M)}\to\lambda_k$ for each fixed $k$ as $M\to\infty$ by \eqref{eq:gammaconver}, concludes \eqref{eq:mgf2}.
\end{proof}

\section{A particular construction of $\cG^{d}_{n_0,n_R,n_S}$}\label{supp:B}

Assuming {$d\ge 2$}, we give a particular construction
of $n_S$ distinct unit vectors $\{\mr_1,\dots,\mr_{n_S}\}$
such that the uniform discrete distribution on this set 
converges weakly to the uniform distribution on $\cS_{d-1}$.
To this end, 
%We have the finite sample $(\mX_i)_{i=1}^{n}$ which yields empirical measure weakly converging to $P$, and as a counterpart,
%We  
%with corresponding empirical measure converging to $U_d$ 
%as follows: 
let us first factorize $n$ into the following form:
\begin{align}\label{eq:key-supp}
n_S=\prod\nolimits_{m=1}^{d-1}n_m,
~~~~~~n_1,n_2,\dots,n_{d-1}\in\Z_+,
~~~\text{with}~n_1,n_2,\dots,n_{d-1}\to\infty~~~\text{as}~n_S\to\infty.
\end{align}
A factorization of $n$ satisfying \eqref{eq:key} and \eqref{eq:key-supp} together will always exist. Indeed, letting $n_{*}:=\lfloor n^{1/(2d-2)}\rfloor$, one possibility is to take $n_1,n_2,\dots,n_{d-1}=n_{*}$, $n_S=n_*^{d-1}$, $n_R=\lfloor n/n_S\rfloor$ (noticing $n_S\le n_R$), and $n_0=n-n_Rn_S$.

%We fix a factorization for any given $n$ where
%$n_R,n_1,\dots,n_{d-1}$ are chosen to be as close to each other as possible.
%Consider the augmented grid given by combining $n_0$ copies of the origin $\bm{0}$ whenever $n_0>0$ and $n_Rn_S$ points obtained as the intersection between $n_R$ hyperspheres centered at the origin, with radii $1/(n_R+1),\dots,n_R/(n_R+1)$
%and $n_S$ unit vectors ``as uniform as possible'' as described below.
To construct deterministic points in the unit ball, we consider
spherical coordinates.  Let $\mt=(t_1,\dots,t_d)^\top\in\R^d$ be a
vector in
Cartesian coordinates.   Its spherical coordinates
$(r,\varphi_1,\ldots,\varphi_{d-1})^\top$ are defined implicitly
as
\begin{align*}
t_{1}&=r\cos(\varphi_{1}),~~~
t_{2}=r\sin(\varphi_{1})\cos(\varphi_{2}),\\
%t_{3}&=r\sin(\varphi_{1})\sin(\varphi_{2})\cos(\varphi_{3}),\\
&\vdots \\
t_{d-1}&=r\sin(\varphi_{1})\cdots \sin(\varphi_{d-2})\cos(\varphi_{d-1}),\\
t_{d}&=r\sin(\varphi_{1})\cdots \sin(\varphi_{d-2})\sin(\varphi_{d-1}),\yestag\label{eq:car}
\end{align*}
where $r\in[0,\infty)$, $\varphi_{1},\dots,\varphi_{d-2}\in[0,\pi]$,
and $\varphi_{d-1}\in[0,2\pi)$. Notice that the inverse transform is
unique, while the transform is not unique in some special cases: if
$r=0$, then $\varphi_{1},\dots,\varphi_{d-1}$ are arbitrary; if
$\varphi_{m}\in\{0,\pi\}$, then $\varphi_{m+1},\dots,\varphi_{d-1}$
are arbitrary.  To avoid any ambiguity, we make the
spherical coordinates unique by specifying that arbitrary
coordinates are zero in these cases.

%The approximation problem is then settled in two steps. In the first
%step, and specified in the following lemma, 
The following lemma constructs a set of points
on the unit sphere such that the uniform discrete distribution on this
set will weakly converge to the uniform distribution over $\cS_{d-1}$.

\begin{lemma}\label{lem:unif}
{When $d\ge2$, }for each $m\in\zahl{d-1}$, let
$u_{m,j}=(2j-1)/(2n_{m})$ for $j\in\zahl{n_m}$, and define the function $g_m: [0,\pi] \to \mathbb{R}$ as
\[g_m(\theta):=%\int_{0}^{\theta}\sin^m u du=
\begin{cases}
{\displaystyle \frac{1}{2^{m-1}}\sum_{k=0}^{(m-1)/2}(-1)^{\{(m-1)/2-k\}}\mbinom{m}{k}\frac{1-\cos\{(m-2k)\theta\}}{m-2k}}, & \text{if $m$ is odd},\\[1em]
{\displaystyle \frac{1}{2^m}{\mbinom{m}{{m}/{2}}}\theta+{\frac{1}{2^{m-1}}}\sum_{k=0}^{{m}/{2}-1}(-1)^{({m}/{2}-k)}\mbinom{m}{k}\frac{\sin\{(m-2k)\theta\}}{m-2k}}, & \text{if $m$ is even}.
\end{cases}
\yestag\label{eq:sinintegral}
\]
Let 
\begin{equation}\label{eq:coord}
  \varphi_{m,j}=
  \begin{cases}
    g_{d-1-m}^{-1}\Big(\frac{\sqrt{\pi}\Gamma((m+1)/2)}{\Gamma(m/2+1)}u_{m,j}\Big),&\text{for
    } m\in\zahl{d-2} \;\;\text{and}\;\; j\in\zahl{n_m},\\
    2\pi u_{d-1,j}, &\text{for } m=d-1\;\;\text{and}\;\; j\in\zahl{n_{d-1}}.
  \end{cases}
  % \begin{cases}
  %   g_{d-1-m}^{-1}\Big(\frac{\sqrt{\pi}\Gamma((m+1)/2)}{\Gamma(m/2+1)}u_{m,j}\Big),&~\text{for}~m\in\zahl{d-2}~~~\text{and}~~~j\in\zahl{n_m},\\
  %   2\pi u_{d-1,j}, &~~~\text{for}~m=d-1~~~\text{and}~~~j\in\zahl{n_{d-1}},
  % \end{cases}
\end{equation}
Then the uniform discrete distribution on the set %\fbox{they are also vectors, right?} 
$\{\mt_{j_1,\ldots,j_{d-1}};
j_1\in\zahl{n_1},\ldots,j_{d-1}\in\zahl{n_{d-1}}\}$ of points with spherical coordinates $(1,\varphi_{1,j_1},\dots,\varphi_{d-1,j_{d-1}})^\top$
weakly converges to the uniform distribution over $\cS_{d-1}$ as $n_1,\dots,n_{d-1}\to\infty$.
\end{lemma}

The above construction might look mysterious at the first
  sight. Indeed, to construct an asymptotically uniform grid over
  $\cS_{d-1}$, it is tempting to take a product of univariate uniform
  grids over all spherical coordinates. Unfortunately, points picked
  in this way can be shown to concentrate at the poles, and hence
  cannot serve the desired purpose.  Instead, a more elaborate
  construction such as the one in Lemma \ref{lem:unif} is needed.
%As a matter of fact, this also partially explains why the construction in Lemma \ref{lem:unif} appears so complicated. 

%\subsubsection{Proof of Lemma \ref{lem:unif}}

\begin{proof}[Proof of Lemma \ref{lem:unif}]
We proceed in three steps. 
First, we give an alternative form of the uniform discrete distribution on the points
$\{\mt_{j_1,\ldots,j_{d-1}}; j_1\in\zahl{n_1},\ldots,j_{d-1}\in\zahl{n_{d-1}}\}$ with spherical coordinates $(1,\varphi_{1,j_1},\dots,\varphi_{d-1,j_{d-1}})^\top$.
Next, we find this uniform distribution's limiting distribution as $n_1,\dots,n_{d-1}\to\infty$. 
Lastly, we prove that this limiting distribution is uniformly distributed over the unit sphere $\cS_{d-1}$.

First, let $Z_m^{(n_m)}$ be random variables uniformly discrete distributed on the points $\{u_{m,j},j\in\zahl{n_m}\}$ for all $m\in\zahl{d-1}$  such that $Z_1^{(n_1)},\dots,Z_{d-1}^{(n_{d-1})}$ are mutually independent.
Notice that the uniform discrete distribution on the points $\{\mt_{j_1,\ldots,j_{d-1}}; j_1\in\zahl{n_1},\ldots,j_{d-1}\in\zahl{n_{d-1}}\}$ with spherical coordinates $(1,\varphi_{1,j_1},\dots,\varphi_{d-1,j_{d-1}})^\top$ is identical to
the distribution given by random spherical coordinates $(1,\varPhi_{1}^{(n_1)},\dots,\varPhi_{d-1}^{(n_{d-1})})^\top$, where
\[
\varPhi_{m}^{(n_m)}=\begin{cases}
g_{d-1-m}^{-1}(\frac{\sqrt{\pi}\Gamma((m+1)/2)}{\Gamma(m/2+1)}Z_m^{(n_m)}),&~\text{for}~m\in\zahl{d-2},\\
2\pi Z_{d-1}^{(n_{d-1})}, &~\text{for}~m=d-1.
\end{cases}\yestag\label{eq:empunifalter}
\]

Second, we determine the limit of 
the distribution with random spherical coordinates \eqref{eq:empunifalter}
as $n_1,\dots,n_{d-1}\to\infty$.
Let $Z_1,\dots,Z_{d-1}$ be independent random variables that are uniformly distributed on $(0,1)$. 
We have $Z_m^{(n_m)}\stackrel{\sf d}{\longrightarrow}Z_m$ for $m\in\zahl{d-1}$ as $n_m\to\infty$ by the following argument: 
\[\Pr(Z_m^{(n_m)}\le x)=\frac{\lfloor n_mx+1/2\rfloor}{n_m}\to x =\Pr(Z_m\le x),~~~\text{for}~x\in(0,1),\yestag\label{eq:simplefact}\]
as $n_m\to\infty$.
Accordingly, the limiting distribution of \eqref{eq:empunifalter} is 
given by random spherical coordinates $(1,\varPhi_{1},\dots,\varPhi_{d-1})^\top$, where
\[
\varPhi_{m}=\begin{cases}
g_{d-1-m}^{-1}(\frac{\sqrt{\pi}\Gamma((m+1)/2)}{\Gamma(m/2+1)}Z_m),&~\text{for}~m\in\zahl{d-2},\\
2\pi Z_{d-1}, &~\text{for}~m=d-1,
\end{cases}\yestag\label{eq:unifalter}
\]
due to the continuous mapping theorem \citep[Theorem~2.3]{MR1652247}. 
%Cram\'{e}r--Wold device \citep[Theorem~29.4]{MR1324786},
%L\'{e}vy's continuity theorem \citep[Theorem~26.3]{MR1324786}, 
%and the multiplicative property of characteristic functions of the sum of independent random variables \citep[Eq.~(26.12)]{MR1324786}.

Lastly, we show that the distribution given by random spherical coordinates \eqref{eq:unifalter} 
is uniformly distributed over the unit sphere $\cS_{d-1}$.
The area element of $\cS_{d-1}$, denoted by $\d_{\cS_{d-1}}V$, can be written in terms of spherical coordinates as 
\[
\d_{\cS_{d-1}}V
%=\sin^{d-2}(\varphi_{1})\sin^{d-3}(\varphi_{2})\cdots \sin(\varphi_{d-2}) \d\varphi _{1} \d\varphi_{2}\cdots \d\varphi_{d-1}
=\Big(\prod_{m=1}^{d-2} \sin^{d-1-m}(\varphi_{m})\d\varphi_m\Big) \cdot \d\varphi_{d-1}
=\Big(\prod_{m=1}^{d-2} \d(g_{d-1-m}(\varphi_m))\Big) \cdot \d\varphi_{d-1},
\yestag\label{eq:whyunif1}
\]
where the first equality is by \citet{MR1530579} and the last equality uses the trigonometric power-reduction formulas %\citep[p.~445]{MR0685759}
\citep[p.~388]{MR0910542}. 
Here $g_m(\theta)$ is defined as \eqref{eq:sinintegral}.
%Using the trigonometric power-reduction formulas %\citep[p.~445]{MR0685759}
%\citep[p.~388]{MR0910542}, 
%\[\sin^{m}\theta=
%\begin{cases}
%{\displaystyle \frac{1}{2^{m-1}}\sum_{k=0}^{(m-1)/2}(-1)^{\{(m-1)/2-k\}}\mbinom{m}{k}\sin\{(m-2k)\theta\}}, &~\text{if $m$ is odd},\\
%{\displaystyle \frac{1}{2^m}{\mbinom{m}{m/2}}+{\frac{2}{2^{m}}}\sum_{k=0}^{{m}/{2}-1}(-1)^{(m/{2}-k)}\mbinom{m}{k}\cos\{(m-2k)\theta\}}, &~\text{if $m$ is even},
%\end{cases}
%%https://math.stackexchange.com/questions/1968727/integrating-sinnx
%%https://en.wikipedia.org/wiki/List_of_trigonometric_identities#Power-reduction_formulae
%%http://mathworld.wolfram.com/TrigonometricPowerFormulas.html
%\]
%we deduce that 
%\[
%\d_{\cS_{d-1}}V
%=\Big(\prod_{m=1}^{d-2} \d(g_{d-1-m}(\varphi_m))\Big) \cdot \d\varphi_{d-1},
%\yestag\label{eq:whyunif1}
%\] 
%where $g_m(\theta)$ is defined as \eqref{eq:sinintegral}.
%\[g_m(\theta):=
%\begin{cases}
%{\displaystyle \frac{2}{2^m}\sum_{k=0}^{(m-1)/2}(-1)^{\{(m-1)/2-k\}}\mbinom{m}{k}\frac{1-\cos\{(m-2k)\theta\}}{m-2k}}, & \text{ if $m$ is odd},\\
%{\displaystyle \frac{1}{2^m}{\mbinom{m}{{m}/{2}}}\theta+{\frac{2}{2^{m}}}\sum_{k=0}^{{m}/{2}-1}(-1)^{({m}/{2}-k)}\mbinom{m}{k}\frac{\sin\{(m-2k)\theta\}}{m-2k}}, & \text{ if $m$ is even}.
%\end{cases}
%\]
%For each integer $m$, $g_m(\theta)$ is monotonically increasing on $[0,\pi]$ and ranges from $0$ to ${\sqrt{\pi}\Gamma((m+1)/2)}/{\Gamma(m/2+1)}$ \citep[p.~381]{MR0910542}.
The  transformation corresponding to \eqref{eq:unifalter} is
\[
\varphi_{m}=\begin{cases}
g_{d-1-m}^{-1}(\frac{\sqrt{\pi}\Gamma((m+1)/2)}{\Gamma(m/2+1)}z_m),&~\text{for}~m\in\zahl{d-2},\\
2\pi z_{d-1}, &~\text{for}~m=d-1,
\end{cases}\yestag\label{eq:unifaltersamll}
\]
which is a bijection between $(0,1)$ and $(0,\pi)$ for $m\in\zahl{d-2}$ 
\citep[p.~381]{MR0910542},
and a bijection between $(0,1)$ and $(0,2\pi)$ for $m=d-1$. 
In view of \eqref{eq:unifaltersamll}, we have 
\[
\frac{\d(g_{d-1-m}(\varphi_m))}{\d z_m}=\frac{\sqrt{\pi}\Gamma((m+1)/2)}{\Gamma(m/2+1)},~\text{for}~m\in\zahl{d-2},
~~~\text{and}~~~
\frac{\d \varphi_{d-1}}{\d z_{d-1}}=2\pi.\yestag\label{eq:whyunif2}
\]
Plugging \eqref{eq:whyunif2} into \eqref{eq:whyunif1} yields
\[\d_{\cS_{d-1}}V=2\pi\prod_{m=1}^{d-2}\frac{\sqrt{\pi}\Gamma((m+1)/2)}{\Gamma(m/2+1)}
\cdot\prod_{m=1}^{d-1}\d z_{m}.\]
This together with the fact that \eqref{eq:unifalter} ranges over $(0,\pi)$ for $m\in\zahl{d-2}$ and ranges over $(0,2\pi)$ for $m=d-1$
proves the distribution given by random spherical coordinates \eqref{eq:unifalter} 
is uniformly distributed over $\cS_{d-1}$.
\end{proof}

%To pick a random point on the surface of a unit sphere, it is incorrect to select spherical coordinates theta and phi from uniform distributions theta in [0,2pi) and phi in [0,pi], since the area element dOmega=sinphidthetadphi is a function of phi, and hence points picked in this way will be "bunched" near the poles. 
%The construction in Lemma 2.1 looks complicated is that a product of regular univariate grids over each polar coordinate would concentrate at the poles, hence does not lead to an asymptotically uniform limit. 

To obtain a particular construction of $\cG^{d}_{n_0,n_R,n_S}$, 
we expand the above approximation over the sphere 
to an approximating augmented grid for the ball. %{, and complete the construction of an approximating augmented grid when $d=1$. }
% To this end, below
% let's define the augmented grid (as a set)
% $\cG^{d}_{n_0,n_R,\bm{n_S}}$. %, such that the uniform distribution over $\cG^{d}_{n_0,n_R,\bm{n_S}}$ (defined later) weakly converges to $U_d$.

\begin{definition}\label{def:aug}
Assuming $d\ge2$, 
  let $r_{j}=j/(n_R+1)$ for $j\in\zahl{n_R}$, and define
  $\varphi_{m,j}$ for $m\in\zahl{d-1},j\in\zahl{n_m}$ as in
  \eqref{eq:coord}.  With notation
  $\bm{n_S}:=(n_1,\dots,n_{d-1})^{\top}$, the augmented grid
  $\cG^{d}_{n_0,n_R,\bm{n_S}}$ is the multiset consisting of $n_0$
  copies of the origin $\bm{0}$ whenever $n_0>0$ and the points
  $\mt_{j_R,j_1,\ldots,j_{d-1}}$ for $j_R\in\zahl{n_R},j_1\in\zahl{n_1},\ldots,j_{d-1}\in\zahl{n_{d-1}}$
  that have spherical coordinates
  $(r_{j_R},\varphi_{1,j_1},\dots,\varphi_{d-1,j_{d-1}})^\top$.
%{When $d=1$, let $n_S=2$, $n_R=\lfloor n/n_S\rfloor$, $n_0=n-n_Rn_S$, and $r_{j}=j/(n_R+1)$ for $j\in\zahl{n_R}$.  The augmented grid $\cG^{d}_{n_0,n_R,n_S}$ is the multiset consisting of $n_0$
%copies of the origin $0$ whenever $n_0>0$ and the points
%$\{\pm r_{j}: j\in\zahl{n_R}\}$.}
\end{definition}

The following proposition is an immediate corollary of
Lemma~\ref{lem:unif}.

\begin{proposition}\label{prop:unif-supp}
  The uniform discrete distribution on the augmented grid
  $\cG^{d}_{n_0,n_R,\bm{n_S}}$, which assigns mass $n_0/n$ to the
  origin and mass $1/n$ to every other grid point, 
%   The distribution with probability mass $n_0/n$ at the origin $\bm{0}$ and probability mass $1/n$ at each of points $\{\mt_{j_R,j_1,\ldots,j_{d-1}}; j_R\in\zahl{n_R},j_1\in\zahl{n_1},\ldots,j_{d-1}\in\zahl{n_{d-1}}\}$ with spherical coordinates $(r_{j_R},\varphi_{1,j_1},\dots,\varphi_{d-1,j_{d-1}})^\top$, where $r_{j}=j/(n_R+1)$ for $j\in\zahl{n_R}$ and $\varphi_{m,j}$'s are defined in \eqref{eq:coord}, 
% called the uniform distribution over the augmented grid $\cG^{d}_{n_0,n_R,\bm{n_S}}$, 
weakly converges to $U_d$.
\end{proposition}

\section{Additional numerical results}

\subsection{Critical values $Q_{1-\alpha}$}

We provide critical values $Q_{1-\alpha}$ at significance levels $\alpha=0.1,0.05,0.01$ 
for the dimensions $(p,q)=(1,1),(1,2),\ldots,(10,10)$  in Table~\ref{tab:critv}. Here all the critical values are estimated numerically with accuracy $5\cdot 10^{-3}$ using the method described in Section~\ref{subsec:eigv}.

\subsection{Additional simulation results}\label{sec:additional-sim}

\begin{example}\label{eg:power-add1}
The data are drawn from Example \ref{eg:sim-power1}(b) with $p=q=10$ and $30$. 
\end{example}

\begin{example}\label{eg:power-add}
The data are drawn such that $\mX$ is generated from standard multivariate normal distribution, and $Y_{ik}=\log(X_{ik}^2)$ for $i\in\zahl{n},k\in\zahl{p}$, with sample size
$n\in\{54,108,216,432\}$, dimensions $p=q\in\{2,3,5,7\}$. 
\end{example}

The values reported in Figure~\ref{fig:power-add1} and Table~\ref{tab:power-add} are based on $1,000$ simulations at the nominal significance level of $0.05$. Compared to Figure~\ref{fig:power1b}, Figure~\ref{fig:power-add1} further confirms that, compared to its competitors, the proposed test appears to be more sensitive to dimension. Table \ref{tab:power-add}  further showed that, in the setup of Example \ref{eg:power-add}, the test via distance covariance with marginal ranks achieves 
the highest power, while the proposed test works well as long as $n\ge 216$
even when $p=q=7$.

{
\renewcommand{\tabcolsep}{4pt}
\renewcommand{\arraystretch}{1.10}
\begin{table}[H]
\centering
\caption{Critical values $Q_{1-\alpha}$ at significance levels $\alpha=0.1,0.05,0.01$ for $(p,q)=(1,1),(1,2),\ldots,\\(10,10)$.}
\label{tab:critv}
{
\small
\begin{subtable}{\textwidth}
\caption{Critical values at significance level of $0.1$}
\centering
\begin{tabular}{c|cccccccccc}
\\[-2em]
\diagbox{$p$}{$q$} & $1$ & $2$ & $3$ & $4$ & $5$ & $6$ & $7$ & $8$ & $9$ & $10$    \\
\hline
$1$  & 0.306 & 0.215 & 0.172 & 0.149 & 0.133 & 0.122 & 0.113 & 0.106 & 0.101 & 0.095 \\
$2$  & 0.215 & 0.145 & 0.114 & 0.098 & 0.087 & 0.080 & 0.075 & 0.069 & 0.065 & 0.063 \\
$3$  & 0.172 & 0.114 & 0.090 & 0.077 & 0.069 & 0.063 & 0.059 & 0.055 & 0.052 & 0.049 \\
$4$  & 0.149 & 0.098 & 0.077 & 0.066 & 0.059 & 0.054 & 0.049 & 0.046 & 0.044 & 0.042 \\
$5$  & 0.133 & 0.087 & 0.069 & 0.059 & 0.052 & 0.047 & 0.044 & 0.041 & 0.039 & 0.037 \\
$6$  & 0.122 & 0.080 & 0.063 & 0.054 & 0.047 & 0.044 & 0.040 & 0.037 & 0.036 & 0.034 \\
$7$  & 0.113 & 0.075 & 0.059 & 0.049 & 0.044 & 0.040 & 0.037 & 0.035 & 0.034 & 0.032 \\
$8$  & 0.106 & 0.069 & 0.055 & 0.046 & 0.041 & 0.037 & 0.035 & 0.033 & 0.031 & 0.030 \\
$9$  & 0.101 & 0.065 & 0.052 & 0.044 & 0.039 & 0.036 & 0.034 & 0.031 & 0.030 & 0.028 \\
$10$ & 0.095 & 0.063 & 0.049 & 0.042 & 0.037 & 0.034 & 0.032 & 0.030 & 0.028 & 0.027 \\
\end{tabular}
\end{subtable}
\\[1em]
\begin{subtable}{\textwidth}
\caption{Critical values at significance level of $0.05$}
\centering
\begin{tabular}{c|cccccccccc}
\\[-2em]
\diagbox{$p$}{$q$} & $1$ & $2$ & $3$ & $4$ & $5$ & $6$ & $7$ & $8$ & $9$ & $10$    \\
\hline
$1$  & 0.490 & 0.320 & 0.249 & 0.211 & 0.187 & 0.172 & 0.156 & 0.146 & 0.139 & 0.130 \\
$2$  & 0.320 & 0.205 & 0.159 & 0.135 & 0.119 & 0.110 & 0.101 & 0.095 & 0.088 & 0.085 \\
$3$  & 0.249 & 0.159 & 0.124 & 0.105 & 0.093 & 0.086 & 0.079 & 0.073 & 0.069 & 0.066 \\
$4$  & 0.211 & 0.135 & 0.105 & 0.089 & 0.079 & 0.072 & 0.066 & 0.062 & 0.059 & 0.056 \\
$5$  & 0.187 & 0.119 & 0.093 & 0.079 & 0.070 & 0.064 & 0.059 & 0.055 & 0.052 & 0.049 \\
$6$  & 0.172 & 0.110 & 0.086 & 0.072 & 0.064 & 0.058 & 0.054 & 0.050 & 0.047 & 0.045 \\
$7$  & 0.156 & 0.101 & 0.079 & 0.066 & 0.059 & 0.054 & 0.049 & 0.047 & 0.044 & 0.042 \\
$8$  & 0.146 & 0.095 & 0.073 & 0.062 & 0.055 & 0.050 & 0.047 & 0.044 & 0.041 & 0.039 \\
$9$  & 0.139 & 0.088 & 0.069 & 0.059 & 0.052 & 0.047 & 0.044 & 0.041 & 0.039 & 0.037 \\
$10$ & 0.130 & 0.085 & 0.066 & 0.056 & 0.049 & 0.045 & 0.042 & 0.039 & 0.037 & 0.035 \\
\end{tabular}
\end{subtable}
\\[1em]
\begin{subtable}{\textwidth}
\caption{Critical values at significance level of $0.01$}
\centering
\begin{tabular}{c|cccccccccc}
\\[-2em]
\diagbox{$p$}{$q$} & $1$ & $2$ & $3$ & $4$ & $5$ & $6$ & $7$ & $8$ & $9$ & $10$    \\
\hline
$1$  & 0.945 & 0.563 & 0.421 & 0.349 & 0.303 & 0.273 & 0.250 & 0.232 & 0.219 & 0.208 \\
$2$  & 0.563 & 0.338 & 0.255 & 0.213 & 0.186 & 0.168 & 0.156 & 0.144 & 0.136 & 0.130 \\
$3$  & 0.421 & 0.255 & 0.194 & 0.162 & 0.142 & 0.131 & 0.119 & 0.111 & 0.105 & 0.100 \\
$4$  & 0.349 & 0.213 & 0.162 & 0.136 & 0.119 & 0.107 & 0.100 & 0.092 & 0.088 & 0.082 \\
$5$  & 0.303 & 0.186 & 0.142 & 0.119 & 0.105 & 0.095 & 0.088 & 0.083 & 0.077 & 0.072 \\
$6$  & 0.273 & 0.168 & 0.131 & 0.107 & 0.095 & 0.088 & 0.079 & 0.073 & 0.071 & 0.066 \\
$7$  & 0.250 & 0.156 & 0.119 & 0.100 & 0.088 & 0.079 & 0.073 & 0.069 & 0.066 & 0.061 \\
$8$  & 0.232 & 0.144 & 0.111 & 0.092 & 0.083 & 0.073 & 0.069 & 0.064 & 0.060 & 0.059 \\
$9$  & 0.219 & 0.136 & 0.105 & 0.088 & 0.077 & 0.071 & 0.066 & 0.060 & 0.057 & 0.055 \\
$10$ & 0.208 & 0.130 & 0.100 & 0.082 & 0.072 & 0.066 & 0.061 & 0.059 & 0.055 & 0.052 \\
\end{tabular}
\end{subtable}
}
\end{table}
}

\begin{figure}[H]
\centering
\includegraphics[width=\textwidth,trim={0 2.5in 0 0},clip]{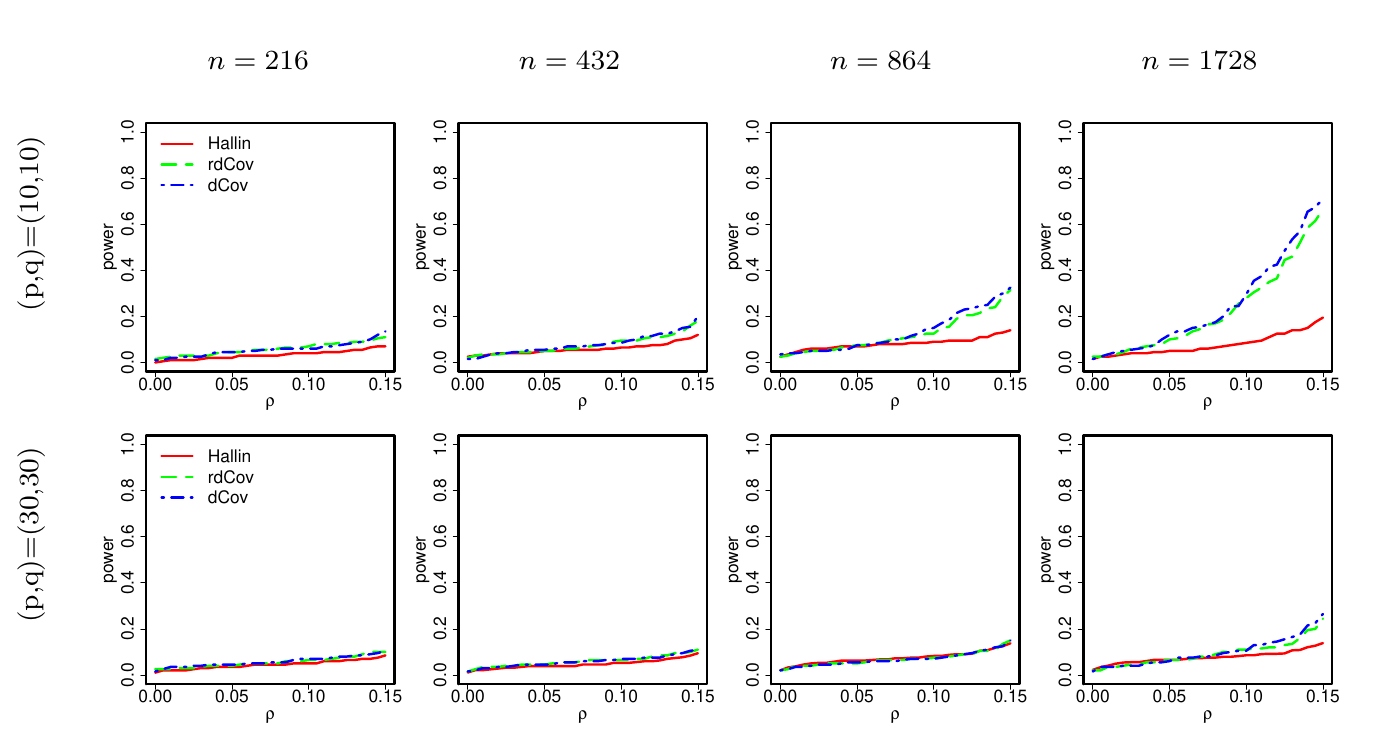}
\caption{Empirical powers of the three competing tests 
in Example~\ref{eg:power-add1}. 
The $y$-axis represents the power based on 1,000 replicates and 
the $x$-axis represents the level of a desired signal.}\label{fig:power-add1}
\end{figure}

{
\renewcommand{\tabcolsep}{4pt}
\renewcommand{\arraystretch}{1.10}
\begin{table}[H]
\centering
\caption{Empirical powers of the proposed test as well as two competing tests in Example~\ref{eg:power-add}.}
\label{tab:power-add}{
\small
\begin{tabular}{ccC{.55in}C{.55in}C{.55in}C{.55in}}
$(p,q)$ & $n$     & Hallin &  rdCov  &   dCov  \\
\hline
$(2,2)$ &   $54$  & 1.000  &  1.000  &  1.000   \\
$(2,2)$ &  $108$  & 1.000  &  1.000  &  1.000   \\
$(2,2)$ &  $216$  & 1.000  &  1.000  &  1.000   \\
$(2,2)$ &  $432$  & 1.000  &  1.000  &  1.000   \\
$(3,3)$ &   $54$  & 0.777  &  0.997  &  0.981   \\
$(3,3)$ &  $108$  & 1.000  &  1.000  &  1.000   \\
$(3,3)$ &  $216$  & 1.000  &  1.000  &  1.000   \\
$(3,3)$ &  $432$  & 1.000  &  1.000  &  1.000   \\
$(5,5)$ &   $54$  & 0.238  &  0.811  &  0.693   \\
$(5,5)$ &  $108$  & 0.888  &  1.000  &  1.000   \\
$(5,5)$ &  $216$  & 1.000  &  1.000  &  1.000   \\
$(5,5)$ &  $432$  & 1.000  &  1.000  &  1.000   \\
$(7,7)$ &   $54$  & 0.144  &  0.612  &  0.436   \\
$(7,7)$ &  $108$  & 0.496  &  0.973  &  0.950   \\
$(7,7)$ &  $216$  & 0.998  &  1.000  &  1.000   \\
$(7,7)$ &  $432$  & 1.000  &  1.000  &  1.000   \\
\end{tabular}}
%Results are averaged over $5000$ simulated data sets.
\end{table}
}

\clearpage

{\small
\bibliographystyle{apalike}
\bibliography{HrankdCov_ams}
}

%\newpage{}

%\section{Discussion}\label{sec:discussion}

%\subsection{Discussion of distance covariance}

\end{document}